\documentclass[final,sort&compress,3p]{elsarticle}

\usepackage{amssymb}
\usepackage{amsmath}
\usepackage{bm}
\usepackage{amsthm}
\usepackage{caption}
\usepackage{subcaption}
\usepackage{xcolor}
\usepackage{soul}

\def\RR{\mathbb{R}}

\def\on{\big|}

\newcommand{\bfm}[1]{\bm{#1}}
\newcommand{\mc}[1]{\mathcal{#1}}
\newcommand{\defset}[2]{\left\{ #1: \ #2 \right\}}
\newcommand{\sprod}[2]{\left\langle#1, #2\right\rangle}
\newcommand{\norm}[1]{\| #1 \|}
\newcommand{\T}{\mc{T}}
\newcommand{\Q}{\mc{Q}}
\newcommand{\C}[1]{{C}^{#1}}
\newcommand{\G}[1]{{G}^{#1}}
\newcommand{\bV}{\bfm{{V}}}
\newcommand{\bR}{\bfm{{R}}}
\newcommand{\bS}{\bfm{{S}}}
\newcommand{\bQ}{\bfm{\mathrm{Q}}}
\newcommand{\bE}{\boldsymbol{\mc{E}}}
\newcommand{\V}[1]{\bV^{(#1)}}
\newcommand{\pd}{p}
\newcommand{\poly}{\mathbb P}
\newcommand{\bern}[2]{B^{#1}_{#2}}
\newcommand{\bernT}[2]{B^{\mbox{{\tiny{$\triangle$}}}, #1}_{#2}}
\newcommand{\bernB}[2]{{B}^{#1}_{#2}}
\newcommand{\w}{\omega}
\newcommand{\parDomT}{\triangle_0}
\newcommand{\parDomQ}{\Box_0}
\newcommand{\parDer}[2]{D_{#1}^{#2}}
\newcommand{\Der}[2]{D_{#1}^{#2}}
\newcommand{\grad}[1]{\mbox{grad}\, #1}
\newcommand{\jac}[1]{{\rm J} {#1}}
\newcommand{\Hess}[1]{\mbox{H} {#1}}
\newcommand{\orth}[1]{{#1}^{\perp}}
\newcommand{\proj}{\Pi}
\newcommand{\indOmega}{\mathcal{I}_{\Omega}}
\newcommand{\indE}{\mathcal{I}_{\mc E}}
\newcommand{\indV}{\mathcal{I}_{\bfm{V}}}
\newcommand{\indQ}{\mathcal{I}_{\mbox{{\tiny{$\Box$}}}}}
\newcommand{\indT}{\mathcal{I}_{\mbox{{\tiny{$\triangle$}}}}}
\newcommand{\A}{\mathcal{A}}
\newcommand{\sigmaBox}{\sigma^{\mbox{{\tiny{$\Box$}}}}}
\newcommand{\sigmaTri}{\sigma^{\mbox{{\tiny{$\triangle$}}}}}

\newtheorem{definition}{Definition}
\newtheorem{theorem}{Theorem}
\newtheorem{proposition}{Proposition}
\newtheorem{lemma}{Lemma}
\newtheorem{corollary}{Corollary}
\newtheorem{remark}{Remark}

\journal{}

\begin{document}

\begin{frontmatter}


\title{A super-smooth $\C{1}$ spline space over planar mixed triangle and quadrilateral meshes} 

\author[UL,AB,UR]{Jan Gro\v{s}elj}
\ead{jan.groselj@fmf.uni-lj.si}
\author[vil,RICAM]{Mario Kapl}
\ead{m.kapl@fh-kaernten.at}
\author[UL,UR]{Marjeta Knez}
\ead{marjetka.knez@fmf.uni-lj.si}
\author[JKU]{\corref{cor}Thomas Takacs}
\ead{thomas.takacs@jku.at}
\author[UP]{Vito Vitrih}
\ead{vito.vitrih@upr.si}
\address[UL]{FMF, University of Ljubljana, Jadranska 19, 1000 Ljubljana, Slovenia}
\address[AB]{Abelium d.o.o., Kajuhova 90, 1000 Ljubljana, Slovenia}
\address[UR]{IMFM, Jadranska 19, 1000 Ljubljana, Slovenia}
\address[vil]{Department of Engineering $\&$ IT, Carinthia University of Applied Sciences, Europastra\ss{}e 4, 9524 Villach, Austria}
\address[RICAM]{RICAM, Austrian Academy of Sciences, Altenberger Str. 69, 4040 Linz, Austria}
\address[JKU]{Institute of Applied Geometry, Johannes Kepler University Linz, Altenberger Str. 69, 4040 Linz, Austria}
\address[UP]{UP FAMNIT and UP IAM, University of Primorska, Glagolja\v ska 8, 6000 Koper, Slovenia}

\cortext[cor]{Corresponding author.}

\begin{abstract}
In this paper we introduce a $C^1$ spline space over mixed meshes composed of triangles and quadrilaterals, suitable for FEM-based or isogeometric analysis. In this context, a mesh is considered to be a partition of a planar polygonal domain into triangles and/or quadrilaterals. The proposed space combines the Argyris triangle, cf.~\cite{ArFrSc68}, with 
the $C^1$ quadrilateral element introduced in~\cite{BrSu05,KaSaTa20} for polynomial degrees $\pd\geq 5$. 
The space is assumed to be $C^2$ at all vertices and $C^1$ across edges,  and the splines are uniquely determined by $C^2$-data at the vertices, values and normal derivatives at chosen points on the edges, and values at some additional points in the interior of the elements.

The motivation for combining the Argyris triangle element with a recent $C^1$ quadrilateral construction, inspired by isogeometric analysis, is two-fold: on one hand, the 
ability to connect triangle and quadrilateral finite elements in a $C^1$ fashion is non-trivial and of theoretical interest. We provide not only approximation error bounds 
but also numerical tests verifying the results. On the other hand, the construction facilitates the meshing process by allowing more flexibility while remaining $C^1$ everywhere. 
This is for instance relevant when trimming of tensor-product B-splines is performed.

In the presented construction we assume to have (bi)linear element mappings and piecewise polynomial function spaces of arbitrary degree $\pd\geq 5$. The basis is simple to 
implement and the obtained results are optimal with respect to the mesh size for $L^\infty$, $L^2$ as well as Sobolev norms $H^1$ and $H^2$.
\end{abstract}


\end{frontmatter}

\section{Introduction}
\label{sec:intro}

Isogeometric analysis (IGA) was introduced in~\cite{HuCoBa05} to apply numerical analysis directly to the B-spline or NURBS representation of CAD models. Since IGA is based on 
B-spline representations, it is capable to generate smooth discretizations, which are used to discretize higher order PDEs, e.g.~\cite{TaDeQu14}. Applications which need 
higher order smoothness (at least $C^1$) are, e.g. Kirchhoff--Love shell formulations~\cite{KiBlLi09,KiBaHs10}, the Navier--Stokes--Korteweg equation~\cite{GoHuNo10}, or the 
Cahn--Hilliard equation~\cite{GoCaBa08}.

CAD models are composed of several B-spline or NURBS patches, which are smooth in the interior. To obtain higher order smoothness for complicated geometries one has to 
additionally impose smoothness across patch interfaces. One possibility is a manifold-like setting, which merges two patches across an interface in a $C^k$ fashion, 
creating overlapping charts, and remains $C^0$ near extraordinary vertices, where several patches meet, see~\cite{BuJuMa15,SaTaVa16}. Such approaches need to be modified to 
increase the smoothness at extraordinary points. One may introduce singularities and define a suitable, locally modified space~\cite{NgKaPe14,NgPe16,ToSpHu17}.

These approaches are strongly related to constructions based on subdivision surfaces~\cite{PeRe08}, such as~\cite{RiAuFe16,ZhSaCi18} based on Catmull--Clark subdivision over 
quadrilateral meshes or~\cite{CiOrSch00,JuMaPeRu16} based on Loop subdivision over triangle meshes.

Smooth function spaces over general quadrilateral meshes for surface design predate IGA, such as~\cite{GrMa87,Pe90,Re95}. These constructions rely on the concept of 
geometric continuity, that is, in the case of $G^1$, surfaces that are tangent continuous without having a $C^1$ parametrization. Also before IGA (or around the same time), 
several approaches were developed for numerical analysis of higher order problems over quad meshes, such as the Bogner--Fox--Schmit element~\cite{BoFoSc65}, the elements 
developed by Brenner and Sung~\cite{BrSu05} for $p\geq6$, or the constructions in~\cite{Ma01,BeMa14}. See Figure~\ref{fig:elements} for a visualization of the Bogner--Fox--Schmit and Brenner--Sung elements. Recently, a family of $C^1$ quadrilateral finite elements was 
described in~\cite{KaSaTa20}.

\begin{figure}[htb]
\centering\footnotesize
\includegraphics[height=.2\textheight,clip]{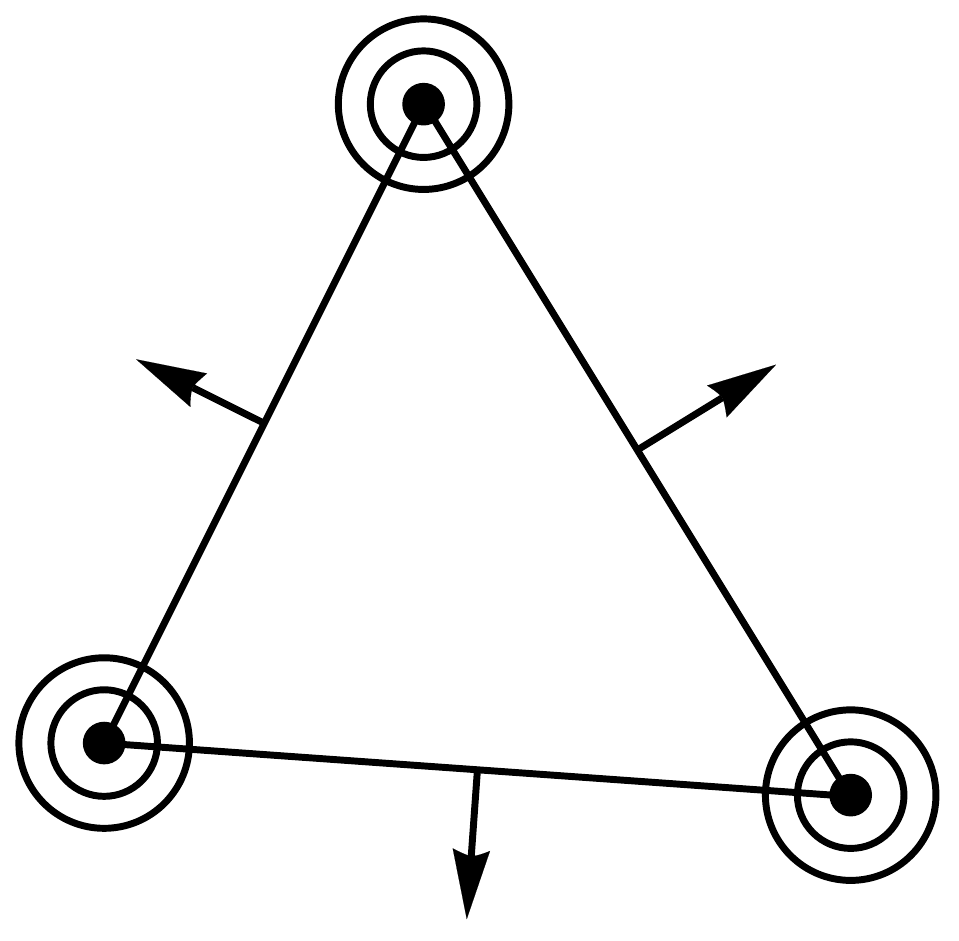} \qquad
\includegraphics[height=.2\textheight]{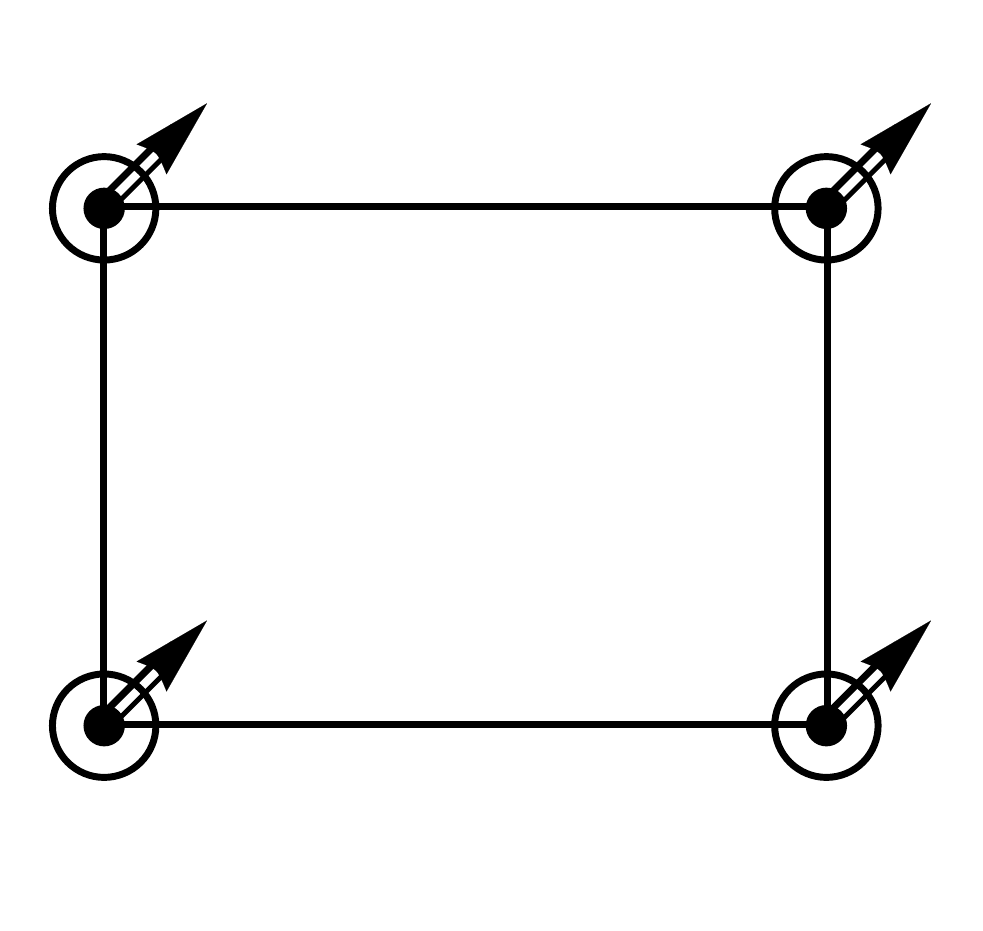} \qquad
\includegraphics[height=.2\textheight,clip]{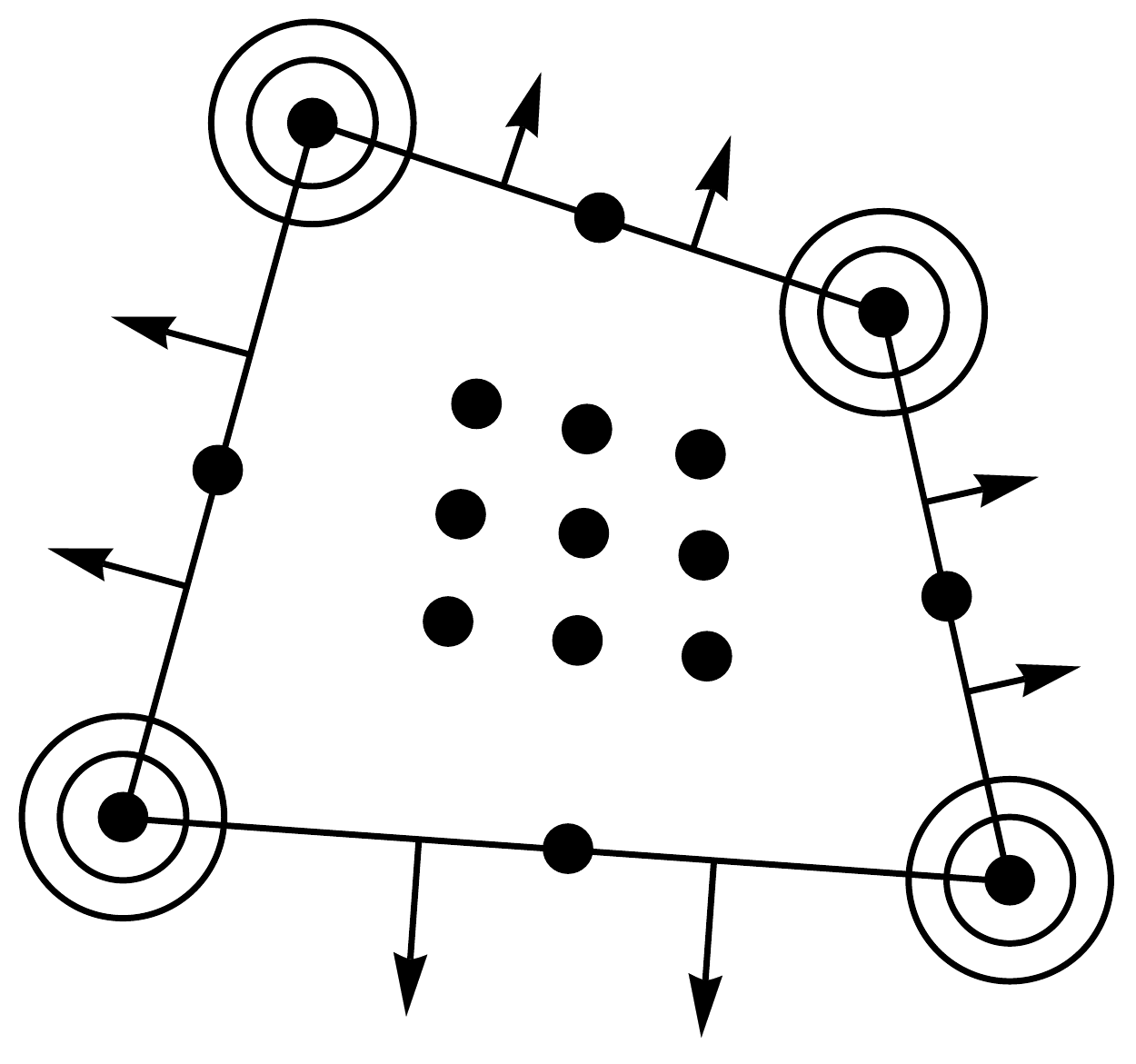} 
\caption{The Argyris element of degree $5$ (left), the Bogner--Fox--Schmit element of bi-degree $3$ (center) and the Brenner--Sung element of bi-degree $6$ (right). Bullets, small circles and large circles denote the interpolation of function values, gradients and Hessians, respectively. A simple arrow denotes the evaluation of a normal derivative and a double stroke arrow denotes the evaluation of the mixed second derivative. Note that the Bogner--Fox--Schmit element is defined only for rectangles, whereas the Brenner--Sung element is defined for any regular quadrilateral.}
\label{fig:elements}
\end{figure}

Due to the increased interest in IGA, the connection between $C^1$ isoparametric functions and $G^1$ surfaces was (re)discovered in the IGA context 
by~\cite{KaViJu15,GrPe15,CoSaTa16}. As a consequence, $C^1$ isogeometric spaces over quadrilateral meshes or multi-patch B-spline configurations were studied 
extensively, see also~\cite{KaBuBeJu16,MoViVi16,KaSaTa17,BlMoVi17,KaSaTa17b,KaSaTa18,ChAnRa18,KaSaTa19,BlMoXu20}.

Already before the introduction of IGA, $C^1$ splines over triangles were introduced and also used for numerical analysis. The first approach to obtain $C^1$ discretizations 
for analysis was the Argyris finite element~\cite{ArFrSc68}, see also~\cite{Ci02,BrSc07} and Figure~\ref{fig:elements}. Other constructions for splines over triangulations followed~\cite{WaMe96,Di97,HaBo00}, see also the 
book~\cite{LaSc07}. Recently, also due to IGA, there is more interest in splines over triangulations~\cite{KaByJu11,SpMaPeSa12,Sp13,JaQi14}. 

There are some straightforward connections between splines over triangulations and splines over quadrangulations. Triangular and tensor-product B\'ezier patches are 
two alternative generalizations of B\'ezier curves and are related in the following way. Triangular patches can be interpreted as singular tensor-product patches, 
where one edge is collapsed to a single point~\cite{Hu01,Ta2014}.

In this paper we combine the $C^1$ constructions for triangles and quadrilaterals for degrees $\pd \geq 5$. We extend the idea of Hermite interpolation with Argyris triangle 
elements from~\cite{JaKa17} as well as the framework from~\cite{KaSaTa20} for quad meshes to mixed quad-triangle meshes. The approach is similar to the most general 
setting in~\cite{MoViVi16}. However, we focus on having given a physical mesh instead of a topological one and we provide a construction which is local to the elements.

Such a mixed triangle and quadrilateral mesh can be relevant for many applications. Fluid-structure interaction problems are often solved by combining two different PDE 
formulations, discretized differently, within a single setup~\cite{BaCaHu08}. Combining triangle and quadrilateral patches in a geometrically continuous fashion is also of relevance for the geometric design of surfaces, see e.g.~\cite{Pe95}. Mixed meshes may also arise from trimming~\cite{KiSeYo09,ScMaJu17}. Most CAD software relies 
on trimming procedures to perform Boolean operations. In that case a B-spline patch is modified by modifying its parameter domain, which is usually a box. The parameter 
domain is then given as a part of the full box, where some parts are cut out by so-called trimming curves. These trimming curves divide the B\'ezier (polynomial) elements of the 
spline patch into inner, outer and cut elements, where the outer elements and outer parts of cut elements are discarded. In that case the resulting mesh 
is composed mostly of quadrilaterals, with (in general) triangles, quadrilaterals and pentagons as cut elements near the trimming boundary. For practical purposes 
(e.g. to simplify quadrature) the cut elements are often split into triangles. Thus, this procedure results in a mixed triangle and quadrilateral mesh (see Figure~\ref{fig:trimming_example}). We also want to point out~\cite{SoCo19}, where the authors combine tensor-product volumes 
as an outer layer with tetrahedral B\'ezier elements inside the domain, with an extra layer of pyramidal elements in between.

\begin{figure}[htb]
\centering
\includegraphics[width=9.0cm,clip]{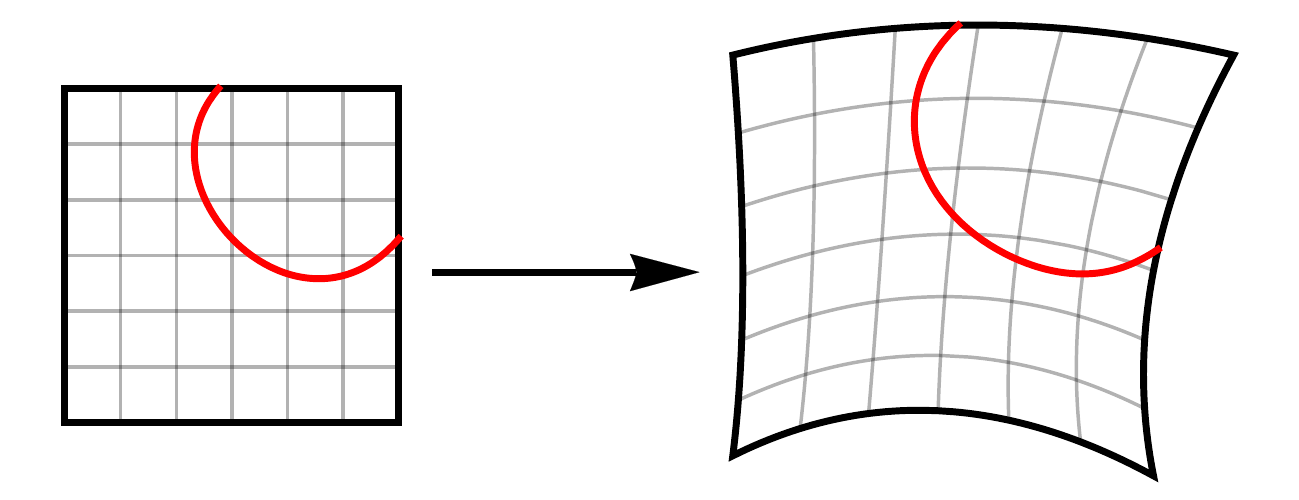} 
\caption{An example of a trimmed patch, where the trimming curve is prescribed in the parameter domain. Note that some of the cut elements are triangles, some are quadrilaterals 
and some are pentagons.}
\label{fig:trimming_example}
\end{figure}

The paper is organized as follows. In Section~\ref{sec:ArgyrisLikeSpace} we first introduce the notation and mesh configuration which will be 
used throughout the paper. After that we define the $C^1$ space over mixed triangle and quadrilateral meshes. Section~\ref{sec:contCond} is devoted to the investigation of continuity conditions 
across interfaces that need to be considered in the construction of splines from such space. This allows us to analyze properties of the space in Section~\ref{sec:SpaceAnalysis} by introducing and studying a projection operator onto the space. The operator is defined via an interpolation problem and serves us to show that the space is of optimal approximation order. The latter is also verified 
numerically in Section~\ref{sec:examples}. We present the conclusions and possible extensions in Section~\ref{sec:Conclusion}.

\section{The Argyris-like space $\A_{\pd}$}
\label{sec:ArgyrisLikeSpace}

The aim of this section is to introduce necessary notation to describe a domain partition consisting of triangles and quadrilaterals. Over such a mixed mesh we then define a spline space, which can be regarded as an extension of the well-known Argyris space. 

\subsection{Mixed triangle and quadrilateral meshes}

We consider an open and connected domain~$\Omega \subset \RR^2$, whose closure~$\overline{\Omega}$ is the disjoint union of triangular or 
quadrilateral elements~$\Omega^{(i)}$, $i \in \indOmega$, edges~$\bE^{(i)}$, $i \in \indE$, and vertices~$\V{i}$, $i \in \indV$, that is
\begin{equation} \label{DomainQmega}
 \overline{\Omega} = \Big( \dot{\bigcup}_{i \in \indOmega} \Omega^{(i)} \Big)  \; \dot{\cup} \Big( \dot{\bigcup}_{i \in \indE} \bE^{(i)} \Big)  \; \dot{\cup} 
 \Big( \dot{\bigcup}_{i \in \indV} \V{i} \Big),
\end{equation}
which implies that no hanging vertices exist. Each vertex~$\V{i}$ is a point in the plane,
\[
 \V{i}\in\mathbb{R}^2, \mbox{ for all }i \in \indV,
\]
and each edge is given by two vertices, i.e.
\[
 \bE^{(i)} = \bE\left(\V{i_1},\V{i_2}\right) = \left\{ (1-v)\V{i_1} + v\V{i_2}: v\in \left(0,1\right) \right\},
\]
for all $i \in \indE$, where $i_1,i_2\in \indV$. Each element $\Omega^{(i)}$, with $i\in \indOmega=\indT \dot{\cup} \indQ$, is either a triangle or a quadrilateral,
where $\indT$ and $\indQ$ are the sets of indices of the triangle and quadrilateral 
elements~$\Omega^{(i)}$, respectively. 
We assume that elements are always open sets and use the notation $\Omega^{(i)} = \T(\V{i_1},\V{i_2},\V{i_3})$ for triangle elements and $\Omega^{(i)} = 
\Q(\V{i_1},\V{i_2},\V{i_3},\V{i_4})$ for quadrilateral elements.
For all $i \in \indT$ we have 
\[
 \overline{\Omega^{(i)}} = \{ (1-u-v)\V{i_1} + u\V{i_2}+v\V{i_3}: (u,v)\in \parDomT\}, 
\]
where $i_1,i_2,i_3\in \indV$ and  
\[
 \parDomT :=\defset{(u,v)\in \RR^2}{u\in \left[0,1\right], \; 0 \leq v \leq 1-u},
\]
whereas for all $i \in \indQ$ we have 
\[
 \overline{\Omega^{(i)}} = \{ (1-u)(1-v)\V{i_1} + u(1-v)\V{i_2}+uv\V{i_3}+(1-u)v\V{i_4}: (u,v)\in \parDomQ\}, 
\]
where $i_1,i_2,i_3,i_4\in \indV$ and $\parDomQ=\left[0,1\right]^2$. We call such a collection of vertices, edges and triangles as well as quadrilateral 
elements a {\it mixed triangle and quadrilateral mesh}, or in short {\it a mixed mesh}.

We denote by $\bfm{F}^{(i)}: \parDomT \to \overline{\Omega^{(i)}}$, with $i \in \indT$, 
and $\bfm{F}^{(i)}: \parDomQ \to \overline{\Omega^{(i)}}$, with $i \in \indQ$, the parametrizations of the elements, which are linear mappings in case 
of triangles and bilinear in case of quadrilaterals. The parametrizations are always assumed to be regular. An example of a mixed mesh, together with the 
mappings~$\bfm{F}^{(i)}$, is shown in Fig.~\ref{fig:mixed_mesh}.

\begin{figure}[htb]
\centering
\includegraphics[width=9.0cm,clip]{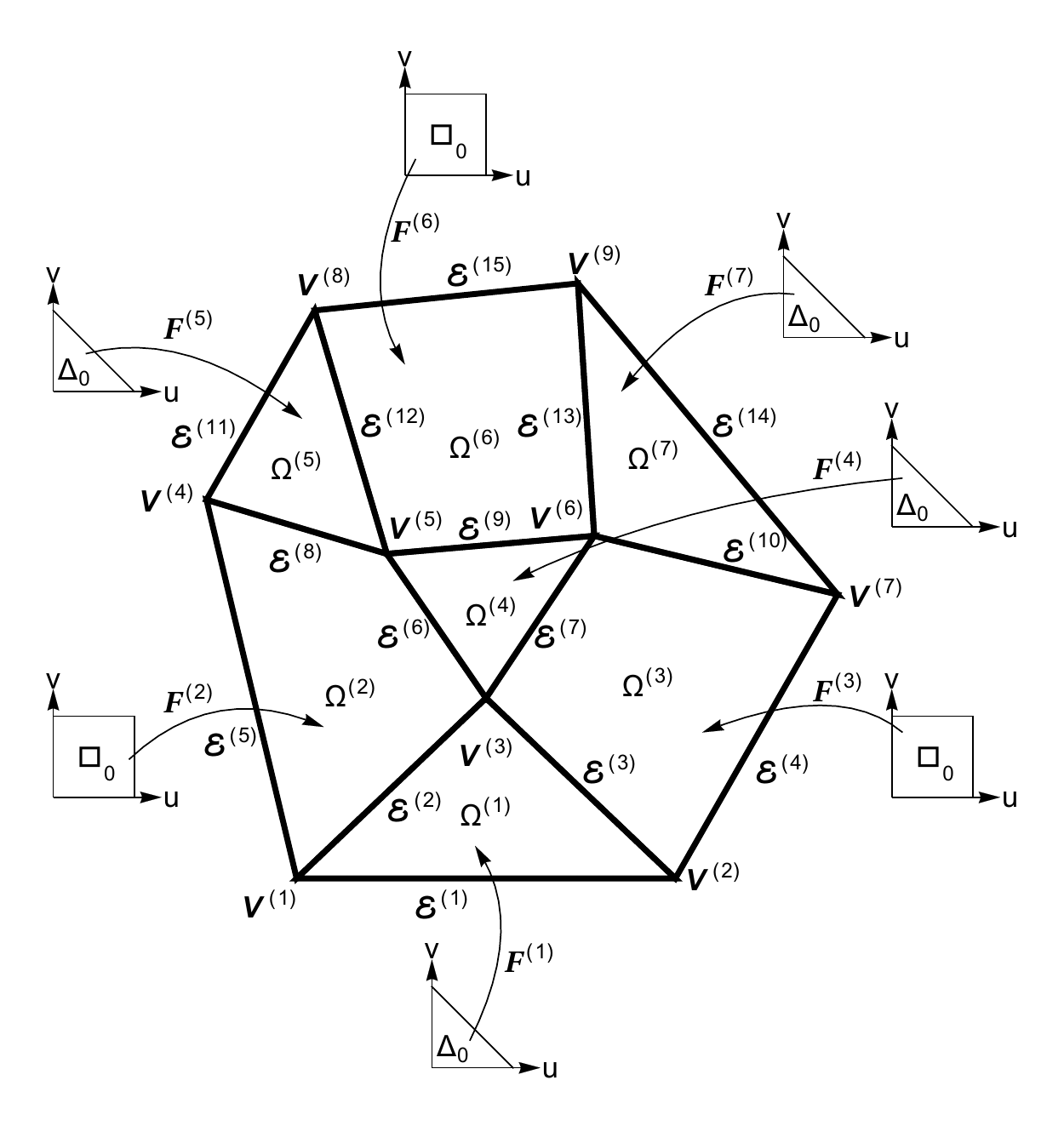} 
\caption{A mixed mesh on a domain~$\Omega$ consisting of triangular and quadrilateral elements~$\Omega^{(i)}$, $i \in \indOmega$, with edges~$\bE^{(i)}$, 
$i \in \indE$, vertices~$\V{i}$, $i \in \indV$, and the associated geometry mappings~$\bfm{F}^{(i)}$, $i \in \indOmega$.}
\label{fig:mixed_mesh}
\end{figure}

\subsection{The Argyris-like space~$\A_{\pd}$ over mixed meshes}

Given an integer~$\pd \geq 1$, let $\poly_\pd$ be the space of univariate polynomials of degree~$\pd$ on the unit interval~$[0,1]$, let $\poly^2_\pd$ denote the space of 
bivariate polynomials of total degree~$\pd$ on the parameter triangle~$\parDomT$ and let $\poly^2_{\pd,\pd}$ be the space of bivariate polynomials of bi-degree~$(\pd,\pd)$ on the unit 
square~$\parDomQ$. We denote by $\bern{\pd}{i}$, $\bernT{\pd}{i,j,k}$ and $\bernB{\pd,\pd}{i,j}$ the corresponding univariate, triangle and tensor-product Bernstein 
bases given by 
\[
\bern{\pd}{i}(u) = \binom{\pd}{i} u^i(1-u)^{\pd-i}, \quad i=0,1,\ldots,\pd,
\]
\[
 \bernT{\pd}{i,j,k}(u,v) = \frac{\pd!}{i! j! k!} u^i v^j(1-u-v)^{k}, \quad i,j,k = 0,1,\ldots,\pd, \; i+j+k=\pd,
\]
and
\[
\bernB{\pd,\pd}{i,j}(u,v) = \bern{\pd}{i}(u)\bern{\pd}{j}(v), \quad i,j=0,1,\ldots,\pd,
\]
respectively.

We are interested in the construction and study of a particular super-smooth $C^1$ spline space of degree $p\geq 5$ on the mixed multi--patch domain~$\Omega$ defined as
\begin{equation} \label{splineSpaceAm}
 \A_{\pd} :=\defset{\varphi \in \C{1}{\left(\overline{\Omega}\right)}}{ 
 \varphi\circ\bfm{F}^{(i)} \in 
\begin{cases}
\poly^2_\pd, & i \in \indT\\
\poly^2_{\pd,\pd}, & i \in \indQ
 \end{cases}; \;
 \varphi \in \C{2}\left(\V{i}\right), \, i \in \indV;  \;   
 \parDer{\bfm{n}_i}{}{\varphi}\on_{\bE^{(i)}} 
 \in 
 \poly_{\pd-1}, \, i \in \indE
 } .
\end{equation}
Here $\parDer{\bfm{n}_i}{}{\varphi}$
denotes the derivative in the direction of the unit vector $\bfm{n}_i$ orthogonal to the edge $\bE^{(i)}$. In what follows, we denote the orthogonal vectors by 
$\orth{(x,y)} := (y,-x)$.  In case of a triangle mesh, that is $\indOmega=\indT$, the space~$\A_{\pd}$ corresponds to the classical Argyris triangle finite element space of 
degree~$\pd$, cf. \cite{ArFrSc68}. Therefore, we refer to the space~$\A_{\pd}$ as the {\em (mixed triangle and quadrilateral) Argyris-like space} of degree~$\pd$. Note that 
the space~$\A_{\pd}$ was also discussed in \cite{KaSaTa19} for the case of a quadrilateral mesh (i.e. $\indOmega=\indQ$) for $\pd=5,6$. Moreover, in~\cite{BrSu05} the 
corresponding quadrilateral element was introduced for $\pd\geq 6$. These constructions were summarized, unified and extended to specific macro-elements for $p=3,4$ in~\cite{KaSaTa20}.

\section{Continuity conditions}
\label{sec:contCond}

We study the $\C{1}$ continuity conditions relating two neighboring elements from a mixed mesh. After presenting some general results, we devote attention to three specific cases of interest, 
i.e., a quadrilateral--triangle, triangle--triangle, and quadrilateral--quadrilateral join.

\subsection{General conditions across the interface}

Without loss of generality let two neighboring elements from the mixed mesh be denoted by $\Omega^{(1)}$ 
and $\Omega^{(2)}$ and parameterized by (bi)linear geometry mappings 
$$\bfm{F}^{(1)}: \mathcal{D}^{(1)}\to \overline{\Omega^{(1)}} \quad {\rm and} \quad 
\bfm{F}^{(2)}:\mathcal{D}^{(2)}\to \overline{\Omega^{(2)}}, 
$$ 
where  $\mathcal{D}^{(1)}$ and $\mathcal{D}^{(2)}$ denote $\parDomQ$ or $\parDomT$.
Suppose that $\overline{\Omega^{(1)}}$ and $\overline{\Omega^{(2)}}$ have a common interface 
$\overline{\bE}$ attained at 
\begin{equation} \label{eq:continuity}
 \bfm{F}^{(1)}(0,v) = \bfm{F}^{(2)}(0,v), \quad v\in [0,1].
\end{equation}
The graph 
$\Phi \subset \left(\overline{\Omega^{(1)}} \cup \overline{\Omega^{(2)}}\right) \times \RR$ 
of any function 
\begin{equation} \label{eq:fun-phi}
\varphi:\overline{\Omega^{(1)}} \cup \overline{\Omega^{(2)}} \to \RR, \quad
\varphi(x,y) = \begin{cases}
\varphi^{(1)}(x,y), & (x,y)\in \overline{\Omega^{(1)}}\\
\varphi^{(2)}(x,y), & (x,y)\in \overline{\Omega^{(2)}} \setminus \overline{\bE}\\
\end{cases},
\end{equation}
is composed of two patches given by parameterizations
\begin{equation*}
\Phi^{(1)} := \begin{bmatrix}
\bfm{F}^{(1)}\\
f^{(1)}
\end{bmatrix}: \mathcal{D}^{(1)} \to \RR^3,\quad 
\Phi^{(2)} := \begin{bmatrix}
\bfm{F}^{(2)}\\
f^{(2)}
\end{bmatrix}: \mathcal{D}^{(2)} \to \RR^3,
\end{equation*}
where $f^{(\ell)} = \varphi^{(\ell)} \circ \bfm{F}^{(\ell)}$, $\ell=1,2$.
Along the common interface the function $\varphi$ is $\C{1}$ continuous if and only if its graph $\Phi$ is $\G{1}$ continuous, i.e.
\begin{align*}
\Phi^{(1)}(0,v) = \Phi^{(2)}(0,v), \quad \det
\left[
\parDer{u}{}{\Phi^{(2)}}(0,v),\,  \parDer{u}{}{\Phi^{(1)}}(0,v),\,  \parDer{v}{}{\Phi^{(1)}}(0,v)
\right]= 0,
\end{align*}
cf.~\cite{KaViJu15,GrPe15,CoSaTa16}. 
Equivalently, it must hold that
\begin{align} 
& f^{(1)}(0,v) = f^{(2)}(0,v), \label{eq:geom-cont-0}\\
& \alpha_1(v) \parDer{u}{}{f^{(2)}}(0,v) - \alpha_2(v) \parDer{u}{}{f^{(1)}}(0,v) + \alpha_3(v) \parDer{v}{}{f^{(1)}}(0,v) = 0,\label{eq:geom-cont}
\end{align}
where
\begin{equation*} \label{def:gluing-alpha}
\alpha_1(v):=  \det \jac{\bfm{F}^{(1)}}(0,v), \quad 
\alpha_2(v):=  \det \jac{\bfm{F}^{(2)}}(0,v), \quad
\alpha_3(v):= \det\left[\parDer{u}{}{\bfm{F}^{(2)}}(0,v), \, \parDer{u}{}{\bfm{F}^{(1)}(0,v)}
\right]
\end{equation*}
are the so called {\it gluing functions} for the interface $\bE$. Note that $\alpha_i(v) \not= 0$ for $v\in [0,1]$, $i=1,2$.
The next two lemmas will further be needed. 
Note that the first one is for the cubic case similar to the construction in \cite{Pe90}.
\begin{lemma} \label{lemma-connection}
For any two  bijective and regular $\C{1}$ geometry mappings $\bfm{F}^{(1)}$ and $\bfm{F}^{(2)}$, such that 
equation~\eqref{eq:continuity} is satisfied,
it holds that 
\begin{equation} \label{eq:lemma1}
\begin{split}
& \det \jac{\bfm{F}^{(2)}}(0,v) \, \orth{\left(\parDer{u}{}{\bfm{F}^{(1)}} (0,v)\right)}
 - \det \jac{\bfm{F}^{(1)}}(0,v) \, \orth{\left(\parDer{u}{}{\bfm{F}^{(2)}} (0,v)\right)}
=\\
&  \det\left[\parDer{u}{}{\bfm{F}^{(2)}}(0,v), \, \parDer{u}{}{\bfm{F}^{(1)}(0,v)} \right]
\, \orth{\left(\parDer{v}{}{\bfm{F}^{(1)}} (0,v)\right)}.
\end{split}
\end{equation}
\end{lemma}
\begin{proof}
The result follows from the equalities
$$
\det \left(\bfm{b}, \bfm{c}\right) \bfm{a} +  \det \left(\bfm{c}, \bfm{a}\right) \bfm{b} + \det \left(\bfm{a}, \bfm{b}\right) \bfm{c} = \bfm{0}, \quad \det \left(\bfm{a}, \bfm{b}\right) = \det \left(\orth{\bfm{a}}, \orth{\bfm{b}}\right), 
\quad \det \left(\bfm{a}, \bfm{b}\right) = -\det \left(\bfm{b}, \bfm{a}\right),
$$
which are true for any three planar vectors $\bfm{a}$, $\bfm{b}$ and $\bfm{c}$. Namely, 
\eqref{eq:lemma1} is obtained from the first equality applied to  
$\bfm{a}=\orth{\left(\parDer{u}{}{\bfm{F}^{(1)}} (0,v)\right)}$, $\bfm{b}=\orth{\left(\parDer{v}{}{\bfm{F}^{(1)}} (0,v)\right)}$ and $\bfm{c}=\orth{\left(\parDer{u}{}{\bfm{F}^{(2)}} (0,v)\right)}$
using also that $\parDer{v}{}\bfm{F}^{(1)}(0,v) = \parDer{v}{}\bfm{F}^{(2)}(0,v)$, which is true by assumption~\eqref{eq:continuity}. 
 \end{proof}
\begin{lemma} \label{lemma-formulaForDirDer}
Suppose that $\bfm{F}^{(\ell)}$ is a bijective and regular $\C{1}$ geometry mapping,  
$\varphi^{(\ell)}$ is a $\C{1}$ continuous function and $f^{(\ell)} = \varphi^{(\ell)} \circ \bfm{F}^{(\ell)}$. Further, let 
$\bfm{n}$ be any chosen unit vector.
The directional derivative 
$\parDer{\bfm{n}}{}{\varphi^{(\ell)}}(x,y)$ at a point $(x,y) = \bfm{F}^{(\ell)}(u,v)$ is in local coordinates equal to
\begin{equation} \label{directional-der}
{\w}_{\bfm{n}}^{(\ell)}(u,v)   := \sprod{\bfm{n}}{ \bfm{G}^{(\ell)}(u,v)},
\end{equation}
where
\begin{equation} \label{directional-der2}
 \bfm{G}^{(\ell)}(u,v)  := \frac{1}{\det \jac{\bfm{F}^{(\ell)}}(u,v)} \left(
\parDer{u}{}{f^{(\ell)}(u,v)} \orth{\left(\parDer{v}{}{\bfm{F}^{(\ell)}}(u,v) \right)}
- \parDer{v}{}{f^{(\ell)}(u,v)} \orth{\left(\parDer{u}{}{\bfm{F}^{(\ell)}}(u,v) \right)}
\right).
\end{equation}
\end{lemma}
\begin{proof}
By differentiating the function $f^{(\ell)} = \varphi^{(\ell)} \circ \bfm{F}^{(\ell)}$ we obtain
$$
\left(\parDer{u}{}{f^{(\ell)}(u,v)},
\parDer{v}{}{f^{(\ell)}(u,v)}
\right) = 
\left(\parDer{x}{}{\varphi^{(\ell)}\left(\bfm{F}^{(\ell)}(u,v)\right)},
\parDer{y}{}{\varphi^{(\ell)}\left(\bfm{F}^{(\ell)}(u,v)\right)}
\right)
\jac{\bfm{F}^{(\ell)}}(u,v).
$$
Multiplying this equality by the inverse of the Jacobian matrix $\jac{\bfm{F}^{(\ell)}}(u,v)$ we
get that
$$ \bfm{G}^{(\ell)}(u,v)  = 
\left(\parDer{x}{}{\varphi^{(\ell)}\left(\bfm{F}^{(\ell)}(u,v)\right)},
\parDer{y}{}{\varphi^{(\ell)}\left(\bfm{F}^{(\ell)}(u,v)\right)}
\right).
$$
Applying the scalar product $\sprod{\bfm{n}}{ \, \cdot\,}$ on this equation gives
$${\w}_{\bfm{n}}^{(\ell)}(u,v)  = \sprod{\bfm{n}}{
\left(\parDer{x}{}{\varphi^{(\ell)}\left(\bfm{F}^{(\ell)}(u,v)\right)},
\parDer{y}{}{\varphi^{(\ell)}\left(\bfm{F}^{(\ell)}(u,v)\right)}
\right)
}
= \parDer{\bfm{n}}{}{\varphi^{(\ell)}}(x,y),
$$
where  $(x,y) = \bfm{F}^{(\ell)}(u,v)$ which concludes the proof.
\end{proof}
Let us now choose a unit vector $\bfm{n}$ orthogonal to the common edge $\bE$.
Multiplying \eqref{eq:lemma1}
by $\sprod{\bfm{n}}{\cdot}$ we get that
\begin{equation}\label{alpha3}
\alpha_3(v) = \alpha_2(v) \beta_1(v) - \alpha_1(v) \beta_2(v),
\end{equation}
where
\begin{equation} \label{def:betas}
\beta_\ell(v) = \frac{1}{\beta(v)} \sprod{\bfm{n}}{\orth{\left(\parDer{u}{}{\bfm{F}^{(\ell)}}(0,v) \right)}}, \quad
{\beta}(v) := \sprod{\bfm{n}}{\orth{\left(\parDer{v}{}{\bfm{F}^{(1)}}(0,v) \right)}} =
\sprod{\bfm{n}}{\orth{\left(\parDer{v}{}{\bfm{F}^{(2)}}(0,v) \right)}}.
\end{equation}
The assumption that  parameterizations $\bfm{F}^{(1)}$ and $\bfm{F}^{(2)}$ are in $\poly^2_{1}$ or in $\poly^2_{1,1}$
 implies that $\beta \in \poly_{0}$ is a nonzero constant, the degree of $\alpha_1, \alpha_2, \beta_1, \beta_2$ is less or equal to $1$, and the degree of $\alpha_3$ is 
 less or equal to $2$. 
By \eqref{alpha3} 
the condition \eqref{eq:geom-cont} rewrites to 
\begin{equation} \label{eq:geom-cont-1}
\alpha_1(v) \left(\parDer{u}{}{f^{(2)}}(0,v) - \beta_2(v)  \parDer{v}{}{f^{(1)}}(0,v)\right)=
\alpha_2(v) \left(\parDer{u}{}{f^{(1)}}(0,v) -\beta_1(v) \parDer{v}{}{f^{(1)}}(0,v)\right).
\end{equation}
On the other hand, using Lemma~\ref{lemma-formulaForDirDer} for $u=0$,
the definitions in \eqref{def:betas} and the identity $\det \jac{\bfm{F}^{(\ell)}}(u,v)= \alpha_\ell(v)$,
one can see that the directional derivative 
$\parDer{\bfm{n}}{}{\varphi^{(\ell)}}$ at points on the boundary is in local 
coordinates equal to
\begin{equation}\label{def:wi}
\begin{split}
{\w}_{\bfm{n}}^{(\ell)}(0,v) = \frac{\beta}{\alpha_\ell(v)} \left(
\parDer{u}{}{f^{(\ell)}(0,v)}  - {\beta}_\ell(v)
\parDer{v}{}{f^{(\ell)}(0,v)} 
\right), \quad \ell=1,2.
\end{split}
\end{equation}
So, from \eqref{eq:geom-cont-1} and \eqref{def:wi} it follows that the $\G{1}$ continuity condition \eqref{eq:geom-cont-1} is simply equal to 
\begin{equation*} \label{eq:geom-cont-2}
{\w}_{\bfm{n}}^{(1)}(0,v) = {\w}_{\bfm{n}}^{(2)}(0,v), \quad v\in [0,1],
\end{equation*}
cf.~\cite{Pe02}. The additional assumption $\parDer{\bfm{n}}{}{\varphi^{(\ell)}}\on_{\bE} \in  \poly_{\pd-1}$ implies that
\begin{equation} \label{eq:geom-cont-3}
{\w}_{\bfm{n}}^{(1)}(0,v) = {\w}_{\bfm{n}}^{(2)}(0,v) = \beta\sum_{j=0}^{\pd-1} {d}_{j} \bern{\pd-1}{j}(v), \quad v\in [0,1],
\end{equation}
for some coefficients $d_j$, $j=0,1,\dots,\pd-1$.

In the following subsections we analyze the continuity conditions across the common edge in more detail for three different types of element pairs: quadrilateral--triangle, 
triangle--triangle and  quadrilateral--quadrilateral. According to the simplified notation introduced in this section, we denote the common edge of $\Omega^{(1)}$ and $\Omega^{(2)}$ by $\bE= \bE(\bfm{V}^{(1)}, \bfm{V}^{(2)})$, and we choose $\bfm{n}$ as the vector orthogonal to $\bE$, i.e., 
$$\bfm{n} = \orth{\left(\V{2}-\V{1} \right)}/\norm{\V{2}-\V{1}}.$$

\subsection{Quadrilateral--triangle}  \label{subSec:quad-tria}

Suppose that ${\Omega}^{(1)}$ is a quadrilateral and ${\Omega}^{(2)}$ a triangle,
\begin{equation}\label{domain-case1}
{\Omega}^{(1)} = \Q\left(\V{1}, \V{2}, \V{3}, \V{4}\right), \quad 
{\Omega}^{(2)} = \T\left(\V{5}, \V{2}, \V{1}\right),
\end{equation}
and let the geometry mappings be equal to
\begin{equation} \label{geomMapping-case1}
\begin{split}
& \bfm{F}^{(1)}(u,v) = (1-u)(1-v) \V{1} +  (1-u) v \V{2} + u v \V{3} +  u(1-v) \V{4},\\
& \bfm{F}^{(2)}(u,v) =  u  \V{5}  +  v \V{2}  +  (1-u-v) \V{1}.
\end{split}
\end{equation}
This configuration is visualized in Figure~\ref{fig:tri_quad}.
\begin{figure}[htb]
 \centering
 \begin{picture}(120,105)
  \put(0,10){\scalebox{-1}[1]{\includegraphics[width=.25\textwidth,clip]{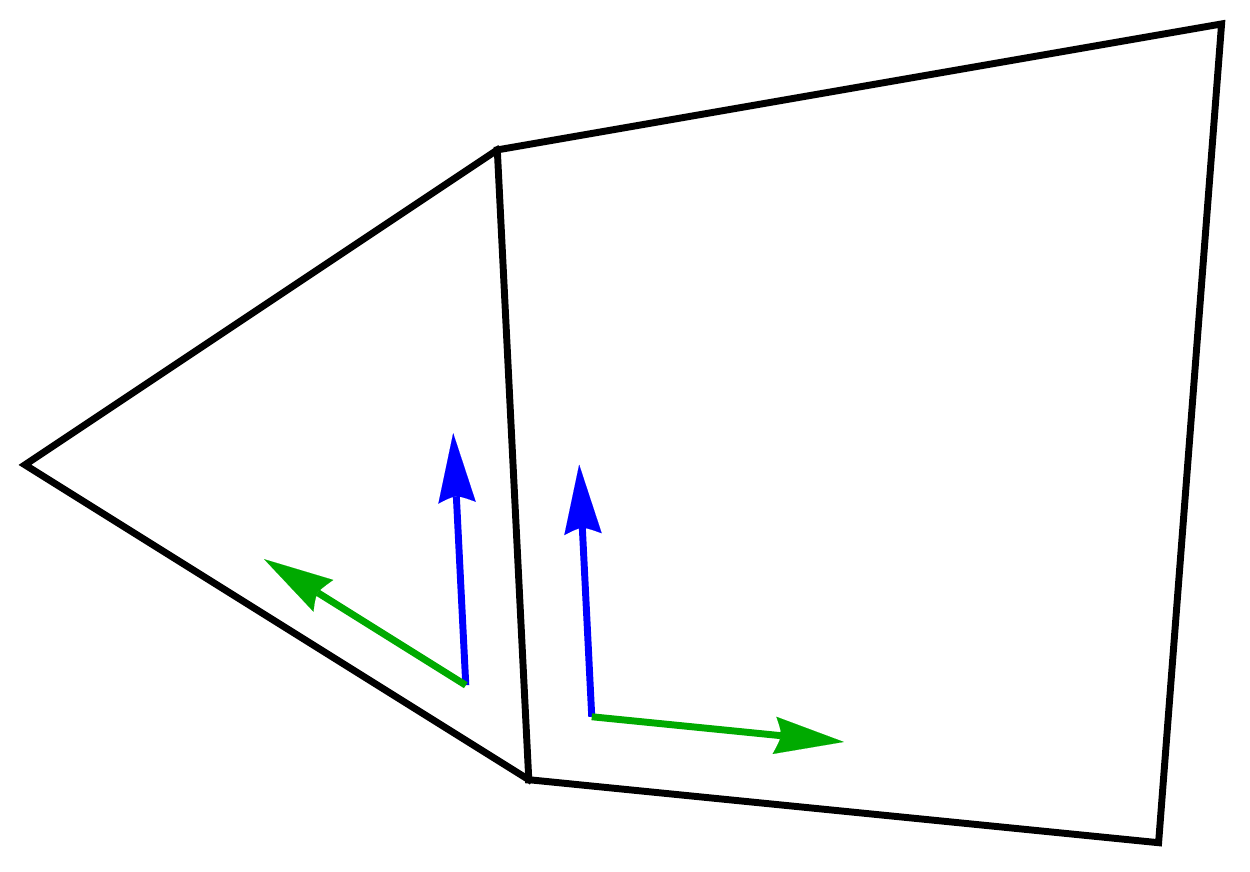}}}
  \put(60,5){$\V{1}$}
  \put(65,85){$\V{2}$}
  \put(0,95){$\V{3}$}
  \put(0,0){$\V{4}$}
  \put(120,45){$\V{5}$}
  \put(28,45){$\Omega^{(1)}$}
  \put(82,45){$\Omega^{(2)}$}
 \end{picture}
\caption{A quadrilateral--triangle pair with parameter directions for $u$ (green) and $v$ (blue).}
\label{fig:tri_quad}
\end{figure}

\begin{remark}
Note that the parameters of the geometry mapping $\bfm{F}^{(2)}$ define the triple $(u,v,1-u-v)$ which represent the barycentric coordinates
of the point $\bfm{F}^{(2)}(u,v)$ with respect to the triangle $\T\left(\V{5}, \V{2}, \V{1}\right)$.
\end{remark}
The gluing functions $\alpha_1$, $\alpha_3$ are in this case linear polynomials
$\alpha_i(v) = (1-v)\,\alpha_{i,0} + v \,\alpha_{i,1}$,  $ i=1,3$,  
with coefficients
\begin{align*}
& \alpha_{1,0} = \det\left[\V{4}-\V{1},\V{2}-\V{1}\right], \quad  
\alpha_{1,1} = \det\left[\V{3}-\V{2},\V{2}-\V{1}\right], \\
& \alpha_{3,0} = \det\left[\V{5}-\V{1},\V{4}-\V{1}\right], \quad
\alpha_{3,1} = \det\left[\V{5}-\V{1},\V{3}-\V{2}\right],
\end{align*}
while $\alpha_2$ reduces to a nonzero constant, $\alpha_2 = \det\left[\V{5}-\V{1},\V{2}-\V{1}\right]$. 
Moreover, from \eqref{def:wi} we see that $\beta, \beta_2 \in \poly_0$,  $\beta_1 \in \poly_{1}$, 
\begin{align*}
& \beta = \norm{\V{2}-\V{1}}, \quad 
\beta_{2} = \frac{\sprod{\V{2}-\V{1}}{\V{5}-\V{1}}}{\norm{\V{2}-\V{1}}^2},\\
& \beta_1(v)=(1-v)\,\beta_{1,0} + v \,\beta_{1,1}, \quad
\beta_{1,0} = \frac{\sprod{\V{2}-\V{1}}{\V{4}-\V{1}}}{\norm{\V{2}-\V{1}}^2}, \quad 
\beta_{1,1} = \frac{\sprod{\V{2}-\V{1}}{\V{3}-\V{2}}}{\norm{\V{2}-\V{1}}^2}.
\end{align*}
Let $f^{(1)}$ be a bivariate polynomial of bi-degree $(\pd,\pd)$ and let $f^{(2)}$ be a bivariate polynomial  of 
total degree~$\pd$, expressed in a Bernstein basis as
\begin{equation}\label{pol-f1-case1}
f^{(1)}(u,v) = \sum_{i,j=0}^\pd {b}^{(1)}_{i,j} \bern{\pd}{i}(u) \bern{\pd}{j}(v), \quad
f^{(2)}(u,v) = \sum_{i+j+k = \pd} {b}^{(2)}_{i,j,k} \bernT{\pd}{i,j,k}(u,v).
\end{equation}
It is straightforward to compute 
\begin{align*}
& f^{(1)}(0,v) = \sum_{j=0}^\pd {b}^{(1)}_{0,j} \bern{\pd}{j}(v), \quad
f^{(2)}(0,v) = \sum_{j=0}^\pd {b}^{(2)}_{0,j,\pd-j} \bern{\pd}{j}(v), \\
& \parDer{u}{}{f^{(1)}}(0,v) = \pd \sum_{j=0}^{\pd} \left({b}^{(1)}_{1,j} - {b}^{(1)}_{0,j} \right) \bern{\pd}{j}(v),
\quad
 \parDer{v}{}{f^{(1)}}(0,v) = \pd \sum_{j=0}^{\pd-1} \left({b}^{(1)}_{0,j+1} - {b}^{(1)}_{0,j} \right) \bern{\pd-1}{j}(v),\\
 & \parDer{u}{}{f^{(2)}}(0,v) = \pd \sum_{j=0}^{\pd-1} \left({b}^{(2)}_{1,j,\pd-1-j} - {b}^{(2)}_{0,j,\pd-j} \right) \bern{\pd-1}{j}(v),  \;\,
 \parDer{v}{}{f^{(2)}}(0,v) = \pd \sum_{j=0}^{\pd-1} \left({b}^{(2)}_{0,j+1,\pd-1-j} - {b}^{(2)}_{0,j,\pd-j} \right) \bern{\pd-1}{j}(v).
\end{align*}
Thus the condition \eqref{eq:geom-cont-0} is fulfilled iff
\begin{equation} \label{eq:G0cond-case1}
{b}^{(1)}_{0,j} = {b}^{(2)}_{0,j,\pd-j} = {c}_j, \quad j=0,1,\dots, \pd,
\end{equation}
for any chosen  values ${c}_j$.
Since in this case $\alpha_2$ and $\beta_2$ are constants, we can see that 
${w}_{\bfm{n}}^{(2)}(0,\cdot) \in \poly_{\pd-1}$, so the condition 
  $\parDer{\bfm{n}}{}{\varphi^{(i)}}\on_{\bE} \in  \poly_{\pd-1}$ is automatically fulfilled. 
Relations  \eqref{def:wi} and \eqref{eq:geom-cont-3} imply 
\begin{subequations}\label{case1-eq-1}
\begin{equation} \label{case1-eq-1a}
\parDer{u}{}{f^{(1)}}(0,v) -\beta_1(v) \parDer{v}{}{f^{(1)}}(0,v) = \alpha_1(v) \sum_{j=0}^{\pd-1} d_j 
\bern{\pd-1}{j}(v)
\end{equation}
and
\begin{equation} \label{case1-eq-1b}
\parDer{u}{}{f^{(2)}}(0,v) - \beta_2 \, \parDer{v}{}{f^{(2)}}(0,v) = \alpha_2 \sum_{j=0}^{\pd-1} d_j 
\bern{\pd-1}{j}(v),
\end{equation}
\end{subequations}
which together with \eqref{eq:G0cond-case1}, \eqref{case1-eq-1}  and
\begin{align*}
& \alpha_1(v) \sum_{j=0}^{\pd-1} d_j \bern{\pd-1}{j}(v) = \sum_{j=0}^\pd 
\frac{1}{\pd} \left(
(\pd-j) \, \alpha_{1,0} \,d_j + j \, \alpha_{1,1}\, d_{j-1}
\right) \bern{\pd}{j}(v),
\\
&
\beta_1(v) \parDer{v}{}{f^{(1)}}(0,v)  = 
\sum_{j=0}^\pd \left( (\pd-j) \, \beta_{1,0} \left(
{b}^{(1)}_{0,j+1} - {b}^{(1)}_{0,j}\right) + 
j\, \beta_{1,1} \left(
{b}^{(1)}_{0,j} - {b}^{(1)}_{0,j-1}\right)
\right)\bern{\pd}{j}(v),
\end{align*}
determines the B\'ezier ordinates 
\begin{equation} \label{eq:G1cond-case1-Q}
{b}^{(1)}_{1,j}  = c_{j} + \frac{1}{\pd}\left(
\frac{1}{\pd} \left(
(\pd-j) \, \alpha_{1,0} \,d_j + j \, \alpha_{1,1}\, d_{j-1}
\right)
+  (\pd-j) \, \beta_{1,0} \left(
c_{j+1} - c_{j}\right) + 
j\, \beta_{1,1} \left(
c_{j} - c_{j-1}\right) 
\right) 
\end{equation}
for $ j=0,1,\dots,\pd$, and
\begin{equation} \label{eq:G1cond-case1-T}
 {b}^{(2)}_{1,j,\pd-1-j} = {c}_j + 
 \beta_2 \left(c_{j+1}-c_j\right) +  \frac{1}{\pd} \alpha_2\, d_j, \quad j=0,1,\dots,\pd-1, 
\end{equation}
where ${c}_{-1} :=  0$, $ {c}_{\pd+1} :=  0$.
We can summarize the obtained results in the following proposition.
\begin{proposition} \label{proposition-case1}
Assume that the two neighboring patches, corresponding geometry mappings and functions $f^{(1)}$, $f^{(2)}$ are given by \eqref{domain-case1}, \eqref{geomMapping-case1} and 
\eqref{pol-f1-case1}. Then the isoparametric function \eqref{eq:fun-phi} is $\C{1}$ continuous across the common interface iff the control ordinates satisfy  \eqref{eq:G0cond-case1}, \eqref{eq:G1cond-case1-Q} and \eqref{eq:G1cond-case1-T} for any chosen $2p+1$ coefficients
$\left(c_i\right)_{i=0}^\pd$ and $\left(d_i\right)_{i=0}^{\pd-1}$.
\end{proposition}

\subsection{Triangle--triangle}  \label{subSec:tria-tria}

Suppose that ${\Omega}^{(1)}$ and ${\Omega}^{(2)}$ are both triangles,
\begin{equation}\label{domain-case2}
{\Omega}^{(1)} = \T\left(\V{3}, \V{2}, \V{1}\right), \quad 
{\Omega}^{(2)} = \T\left(\V{5}, \V{2}, \V{1}\right),
\end{equation}
and the geometry mappings equal 
\begin{equation} \label{geomMapping-case2}
\begin{split}
& \bfm{F}^{(1)}(u,v) = u  \V{3}  +  v \V{2}  +  (1-u-v) \V{1},  \quad \bfm{F}^{(2)}(u,v) =  u  \V{5}  +  v \V{2}  +  (1-u-v) \V{1}.
\end{split}
\end{equation}
This configuration is visualized in Figure~\ref{fig:tri_tri}.
\begin{figure}[htb]
 \centering
 \begin{picture}(120,90)
  \put(0,10){\scalebox{-1}[1]{\includegraphics[width=.25\textwidth,clip]{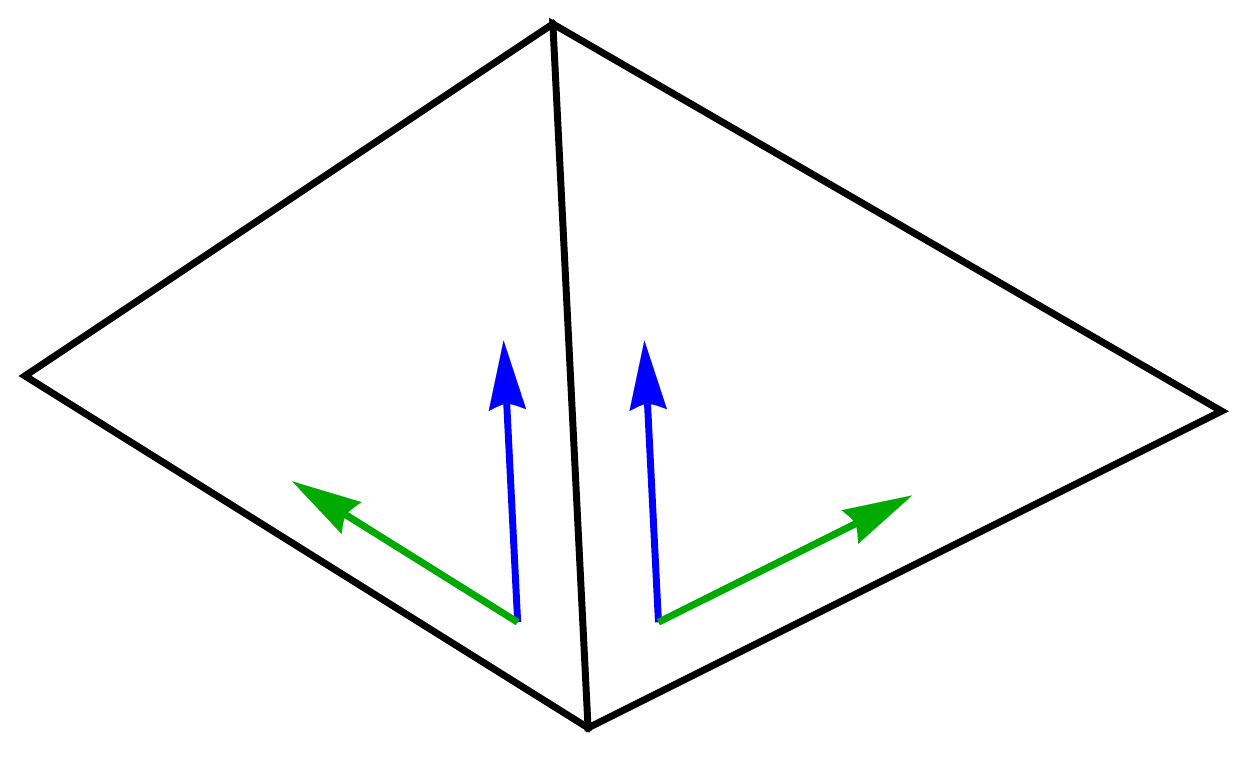}}}
  \put(60,0){$\V{1}$}
  \put(65,80){$\V{2}$}
  \put(-15,40){$\V{3}$}
  \put(120,40){$\V{5}$}
  \put(32,40){$\Omega^{(1)}$}
  \put(78,40){$\Omega^{(2)}$}
 \end{picture}
\caption{A triangle--triangle pair with parameter directions for $u$ (green) and $v$ (blue).}
\label{fig:tri_tri}
\end{figure}

The gluing functions $\alpha_i$, $i=1,2,3$, as well as $\beta_1$, $\beta_2$ are in this case constants, 
\begin{align*}
& \alpha_1 = \det\left[\V{3}-\V{1},\V{2}-\V{1}\right],\quad
\alpha_2 = \det\left[\V{5}-\V{1},\V{2}-\V{1}\right],\\ &\alpha_3 = \det\left[\V{5}-\V{1},\V{3}-\V{1}\right],\\
& \beta_1 = \frac{1}{\norm{\V{2}-\V{1}}} \sprod{\V{2}-\V{1}}{\V{3}-\V{1}},\quad
\beta_2 = \frac{1}{\norm{\V{2}-\V{1}}} \sprod{\V{2}-\V{1}}{\V{5}-\V{1}}.
\end{align*}
Further, let 
\begin{equation}\label{pol-f1-case2}
f^{(\ell)}(u,v) = \sum_{i+j+k = \pd} {b}^{(\ell)}_{i,j,k} \bernT{\pd}{i,j,k}(u,v), \quad \ell=1,2.
\end{equation} 
Also in this case the condition 
 $\parDer{\bfm{n}}{}{\varphi^{(\ell)}}\on_{\bE} \in  \poly_{\pd-1}$ is fulfilled automatically and independently of the geometry of a mesh, and it is straightforward to derive 
 the following result. 
\begin{proposition}  \label{proposition-case2}
Assume that the two neighboring patches, corresponding geometry mappings and functions $f^{(1)}$, $f^{(2)}$ are given by \eqref{domain-case2}, \eqref{geomMapping-case2} and 
\eqref{pol-f1-case2}. Then the isoparametric function \eqref{eq:fun-phi} is $\C{1}$ continuous across the common interface iff the control ordinates satisfy  
\begin{align}
& {b}^{(1)}_{0,j,\pd-j} = {b}^{(2)}_{0,j,\pd-j} = {c}_j, \quad j=0,1,\dots, \pd, \label{eq:Prop2cond1} \\
&  {b}^{(\ell)}_{1,j,\pd-1-j} = {c}_j + 
 \beta_\ell \left(c_{j+1}-c_j\right) +  \frac{1}{\pd} \alpha_\ell\, d_j, \quad  j=0,1,\dots,\pd-1, \quad \ell=1,2,
 \label{eq:Prop2cond2}
 \end{align}
 for any chosen $2\pd+1$ coefficients
 $\left(c_i\right)_{i=0}^\pd$ and $\left(d_i\right)_{i=0}^{\pd-1}$. 
\end{proposition}

\begin{proof}
Since $f^{(\ell)}(0,v) = \sum_{j=0}^\pd {b}^{(\ell)}_{0,j,\pd-j} \bern{\pd}{j}(v)$, $\ell=1,2$, we get \eqref{eq:Prop2cond1}.
Additionally, for $\ell=1,2,$ we have
\begin{equation} \label{eq:Prop2proofDer}
\parDer{u}{}{f^{(\ell)}}(0,v) = \pd \sum_{j=0}^{\pd-1} \left({b}^{(\ell)}_{1,j,\pd-1-j} - {b}^{(\ell)}_{0,j,\pd-j} \right) \bern{\pd-1}{j}(v),  \quad
 \parDer{v}{}{f^{(\ell)}}(0,v) = \pd \sum_{j=0}^{\pd-1} \left({b}^{(\ell)}_{0,j+1,\pd-1-j} - {b}^{(\ell)}_{0,j,\pd-j} \right) \bern{\pd-1}{j}(v),
\end{equation}
and \eqref{def:wi}--\eqref{eq:geom-cont-3} imply 
\begin{equation} \label{eq:Prop2proofRelat}
\parDer{u}{}{f^{(\ell)}}(0,v) - \beta_\ell \, \parDer{v}{}{f^{(\ell)}}(0,v) = \alpha_\ell \sum_{j=0}^{\pd-1} d_j 
\bern{\pd-1}{j}(v), \quad \ell=1,2,
\end{equation}
for constants $\alpha_1$, $\alpha_2$, $\beta_1$, $\beta_2$.
Inserting \eqref{eq:Prop2proofDer} into \eqref{eq:Prop2proofRelat} finally proves \eqref{eq:Prop2cond2}.
\end{proof}

\subsection{Quadrilateral--quadrilateral}  \label{subSec:quad-quad}

Suppose that ${\Omega}^{(1)}$ and ${\Omega}^{(2)}$ are both quadrilaterals,
\begin{equation}\label{domain-case3}
{\Omega}^{(1)} = \Q\left(\V{1}, \V{2}, \V{3}, \V{4}\right), \quad 
{\Omega}^{(2)} = \Q\left(\V{1}, \V{2}, \V{5}, \V{6}\right),
\end{equation}
and the geometry mappings equal 
\begin{equation} \label{geomMapping-case3}
\begin{split}
& \bfm{F}^{(1)}(u,v) = (1-u)(1-v) \V{1} + (1-u) v \V{2} + u v \V{3} +  u(1-v) \V{4},\\
& \bfm{F}^{(2)}(u,v) = (1-u)(1-v) \V{1} + (1-u) v \V{2} + u v \V{5} + u(1-v) \V{6}.
\end{split}
\end{equation}
This configuration is visualized in Figure~\ref{fig:quad_quad}.
\begin{figure}[htb]
 \centering
 \begin{picture}(150,105)
  \put(0,10){\scalebox{-1}[1]{\includegraphics[width=.30\textwidth,clip]{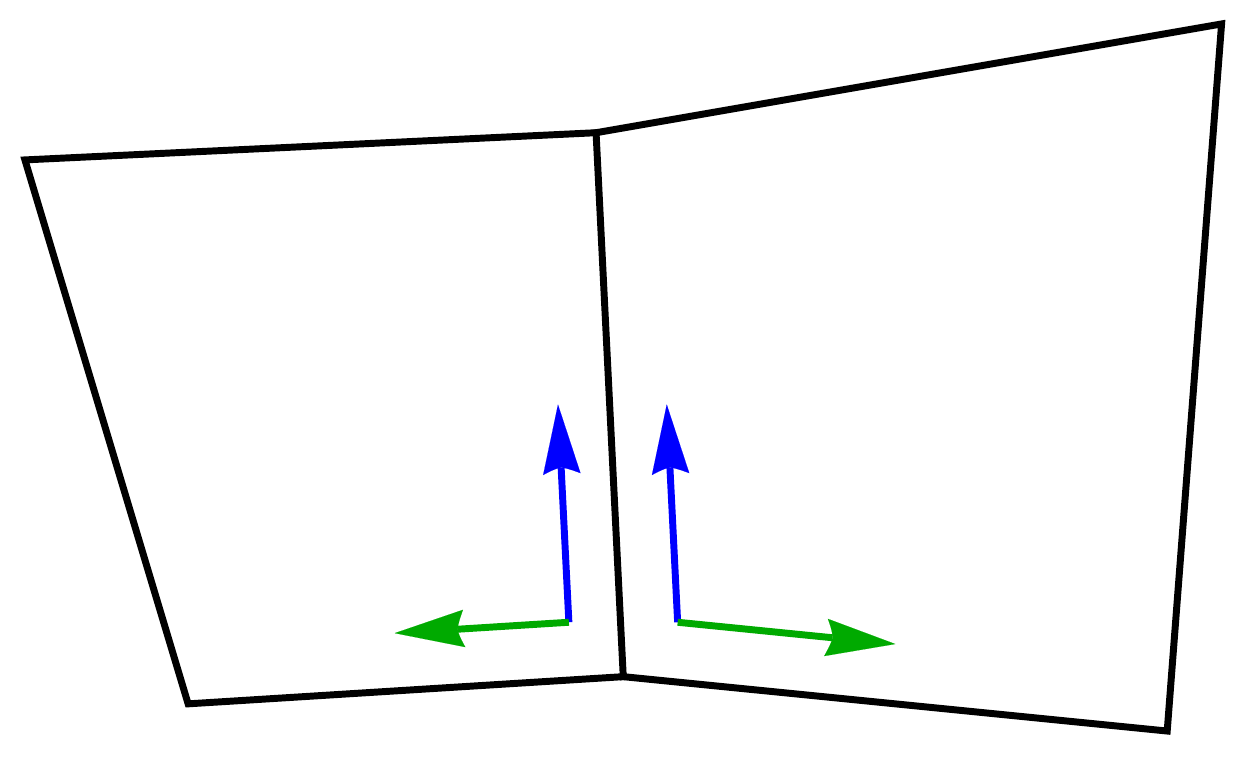}}}
  \put(60,5){$\V{1}$}
  \put(65,85){$\V{2}$}
  \put(0,95){$\V{3}$}
  \put(0,0){$\V{4}$}
  \put(140,80){$\V{5}$}
  \put(120,5){$\V{6}$}
  \put(32,45){$\Omega^{(1)}$}
  \put(95,45){$\Omega^{(2)}$}
 \end{picture}
\caption{A quadrilateral--quadrilateral pair with parameter directions for $u$ (green) and $v$ (blue).}
\label{fig:quad_quad}
\end{figure}

Now, the functions $\alpha_i, \beta_i$, $i=1,2$, are linear polynomials
\begin{align*}
& \alpha_i(v) = (1-v)\,\alpha_{i,0} + v \,\alpha_{i,1}, \\ 
& \quad \alpha_{i,0} = \det\left[\V{2+2i}-\V{1},\V{2}-\V{1}\right], \quad
\alpha_{i,1} = \det\left[\V{1+2i}-\V{2},\V{2}-\V{1}\right], \\
& \beta_i(v)=(1-v)\,\beta_{i,0} + v \,\beta_{i,1}, \\
& \quad \beta_{i,0} = \frac{\sprod{\V{2}-\V{1}}{\V{2+2 i}-\V{1}}}{\norm{\V{2}-\V{1}}^2}, \quad 
\beta_{i,1} = \frac{\sprod{\V{2}-\V{1}}{\V{1+2 i}-\V{2}}}{\norm{\V{2}-\V{1}}^2},
\end{align*}
while $\alpha_3 \in \poly_2$. 
Let $f^{(1)}$ and $f^{(2)}$ be two bivariate polynomials of bi-degree $(\pd,\pd)$,
\begin{equation}\label{pol-f1-case3}
f^{(\ell)}(u,v) = \sum_{i,j=0}^\pd {b}^{(\ell)}_{i,j} \bern{\pd}{i}(u) \bern{\pd}{j}(v), \quad \ell=1,2.
\end{equation}
In this case it could happen that $\parDer{\bfm{n}}{}{\varphi^{(\ell)}}\on_{\bE}$ would be  of degree $\pd$, not $\pd-1$. In particular, this can happen if 
$\alpha_\ell$ reduces to a constant, which, for $\ell =1$, happens if $\bE(\V{1},\V{2})$ is parallel to $\bE(\V{3},\V{4})$, and for $\ell =2$ if $\bE(\V{1},\V{2})$ is 
parallel to $\bE(\V{5},\V{6})$. Moreover, in certain configurations, $\alpha_1$ and $\alpha_2$ are linearly dependent, with $\alpha_2 (v) = \lambda \alpha_1(v)$ and can thus 
be replaced by $\alpha'_1\equiv1$ and $\alpha'_2\equiv\lambda$. E.g. if both elements are rectangles, we have constant $\alpha_\ell$ and $\beta_\ell \equiv 0$. 
See~\cite{BeMa14,KaSaTa17} for a more detailed study of the possible cases. So, the additional 
condition $\parDer{\bfm{n}}{}{\varphi^{(\ell)}}\on_{\bE} \in  \poly_{\pd-1}$ is included to make the proceeding construction independent of the geometry of the mesh.  
\begin{proposition}  \label{proposition-case3}
Assume that the two neighboring elements, corresponding geometry mappings and functions $f^{(1)}$, $f^{(2)}$ are given by \eqref{domain-case3}, \eqref{geomMapping-case3} and 
\eqref{pol-f1-case3}. Then the isoparametric function \eqref{eq:fun-phi} is $\C{1}$ continuous across the common interface and satisfies the additional condition that 
$\parDer{\bfm{n}}{}{\varphi^{(\ell)}}\on_{\bE} \in  \poly_{\pd-1}$,  iff the control ordinates satisfy  
\begin{align}
& {b}^{(1)}_{0,j} = {b}^{(2)}_{0,j} = {c}_j, \nonumber \\[-0.5cm]
 \label{eq:Prop3proofCond} \\
&  {b}^{(\ell)}_{1,j} = c_{j} + \frac{1}{\pd}\left(
\frac{1}{\pd} \left(
(\pd-j) \, \alpha_{\ell,0} \,d_j + j \, \alpha_{\ell,1}\, d_{j-1}
\right) + 
 (\pd-j) \, \beta_{\ell,0} \left(c_{j+1} - c_{j}\right) + j\, \beta_{\ell,1} \left(c_{j} - c_{j-1}\right) \right), 
 \quad \ell=1,2, \nonumber
 \end{align}
$ j=0,1,\dots,\pd$,  for any chosen $2\pd+1$ coefficients
 $\left(c_i\right)_{i=0}^\pd$ and $\left(d_i\right)_{i=0}^{\pd-1}$. 
\end{proposition} 
\begin{proof}
Since ${\Omega}^{(1)}$ and ${\Omega}^{(2)}$ are both quadrilaterals, we have
$f^{(\ell)}(0,v) = \sum_{j=0}^\pd {b}^{(\ell)}_{0,j} \bern{\pd}{j}(v)$,
$$
\parDer{u}{}{f^{(\ell)}}(0,v) = \pd \sum_{j=0}^{\pd} \left({b}^{(\ell)}_{1,j} - {b}^{(\ell)}_{0,j} \right) \bern{\pd}{j}(v),\quad 
\parDer{v}{}{f^{(\ell)}}(0,v) = \pd \sum_{j=0}^{\pd-1} \left({b}^{(\ell)}_{0,j+1} - {b}^{(\ell)}_{0,j} \right) \bern{\pd-1}{j}(v),
$$
for $\ell=1,2$. Assuming
$\parDer{\bfm{n}}{}{\varphi^{(\ell)}}\on_{\bE} \in  \poly_{\pd-1}$, it must hold that 
$$
\parDer{u}{}{f^{(\ell)}}(0,v) -\beta_\ell(v) \parDer{v}{}{f^{(\ell)}}(0,v) = \alpha_\ell(v) \sum_{j=0}^{\pd-1} d_j 
\bern{\pd-1}{j}(v), \quad \ell=1,2,
$$
for linear $\alpha_\ell$, $\beta_\ell$, $\ell=1,2$, 
which is further equivalent to \eqref{eq:Prop3proofCond}.
\end{proof}

In Figure~\ref{fig:element_pairs} we plot pairs of elements, triangle--quadrilateral (left), triangle--triangle (center) as well as quadrilateral--quadrilateral (right). For 
some coefficients $c_i$ and $d_i$ we plot the relevant, non-vanishing B\'ezier ordinates in blue and green, respectively. The figure is to be interpreted in the following way: 
if all coefficients $c_i$ and $d_i$ are set to zero, except for $c_1$, then only the B\'ezier ordinates depicted in blue are non-vanishing. On the other hand, if all 
coefficients except for $d_4$ are set to zero, then only the green ordinates are non-vanishing.
\begin{figure}[htb]
 \centering
 \scalebox{-1}[1]{\includegraphics[width=.16\textwidth]{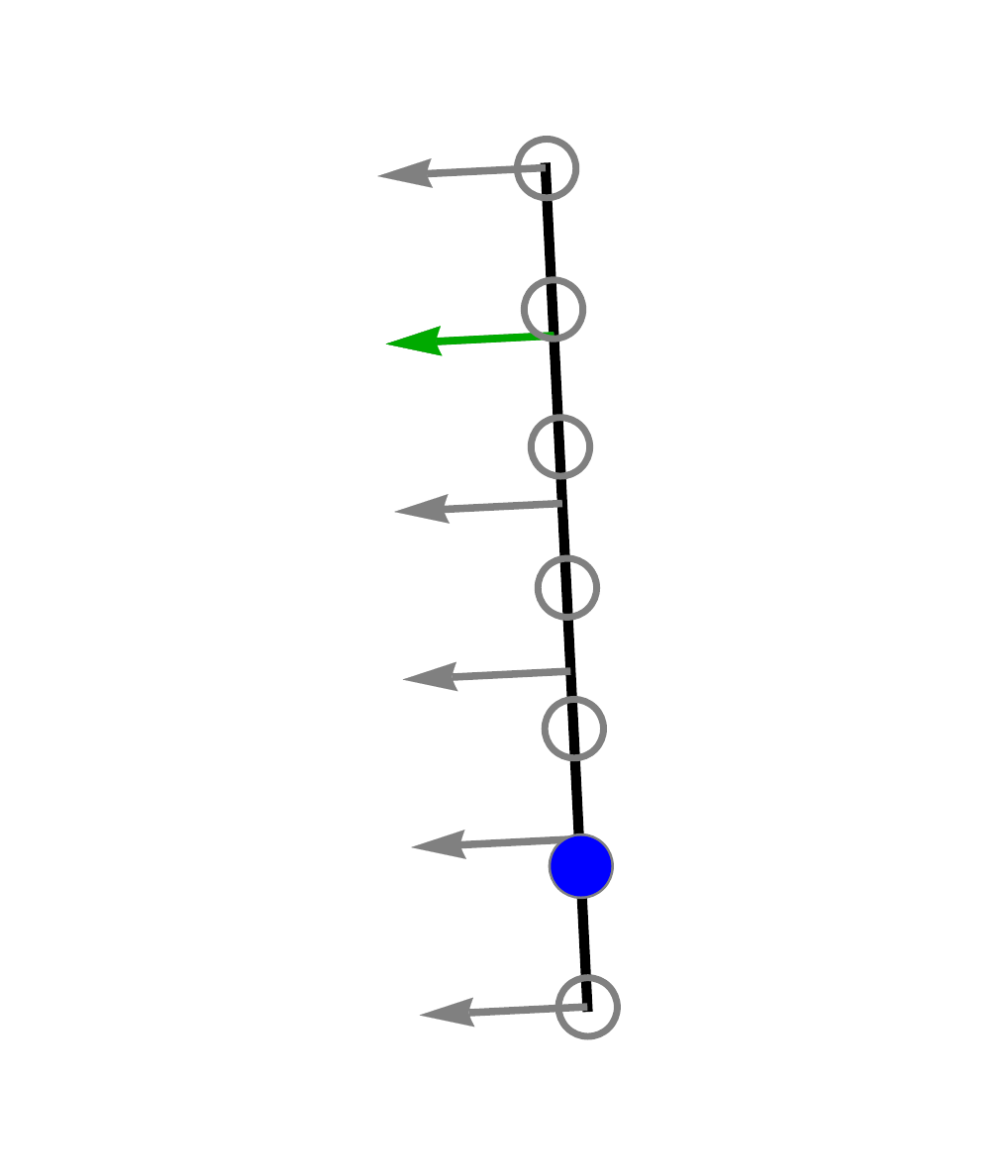}}\;
 \scalebox{-1}[1]{\includegraphics[width=.25\textwidth,clip]{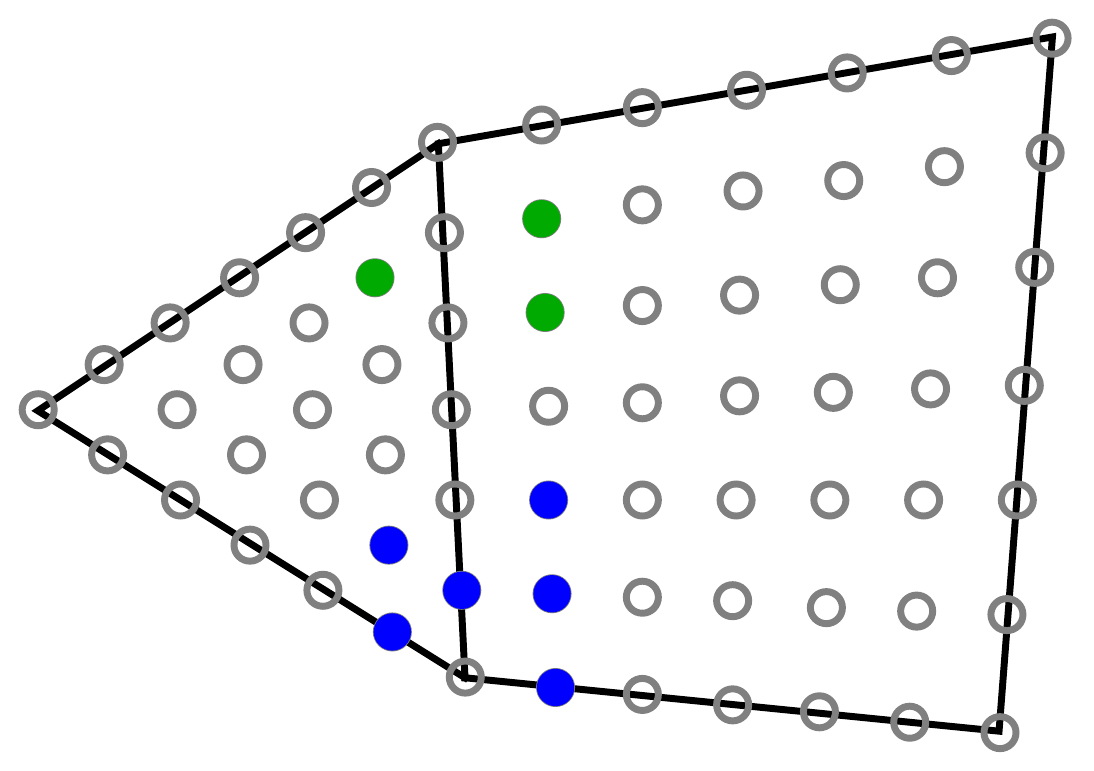}}\quad
\scalebox{-1}[1]{\includegraphics[width=.25\textwidth,clip]{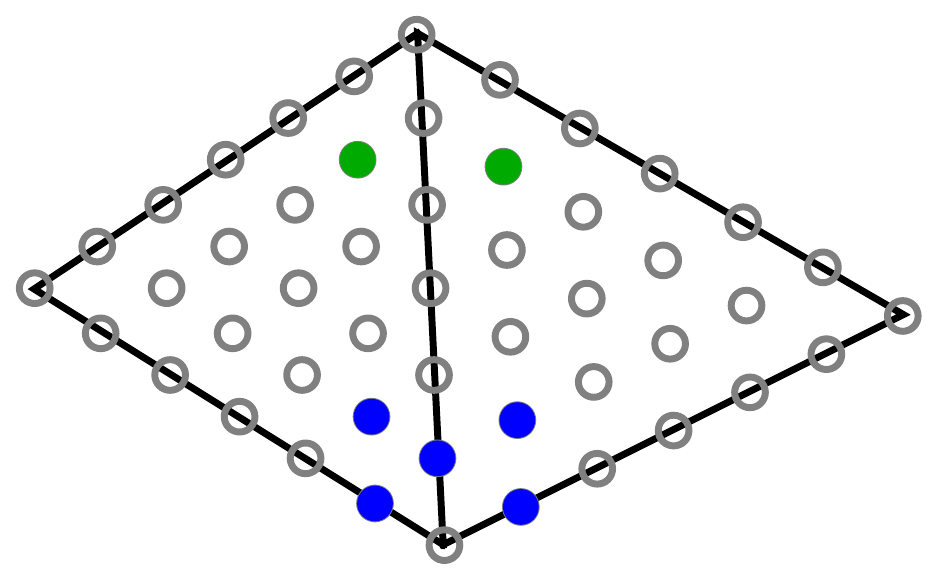}}\quad
\scalebox{-1}[1]{\includegraphics[width=.27\textwidth,clip]{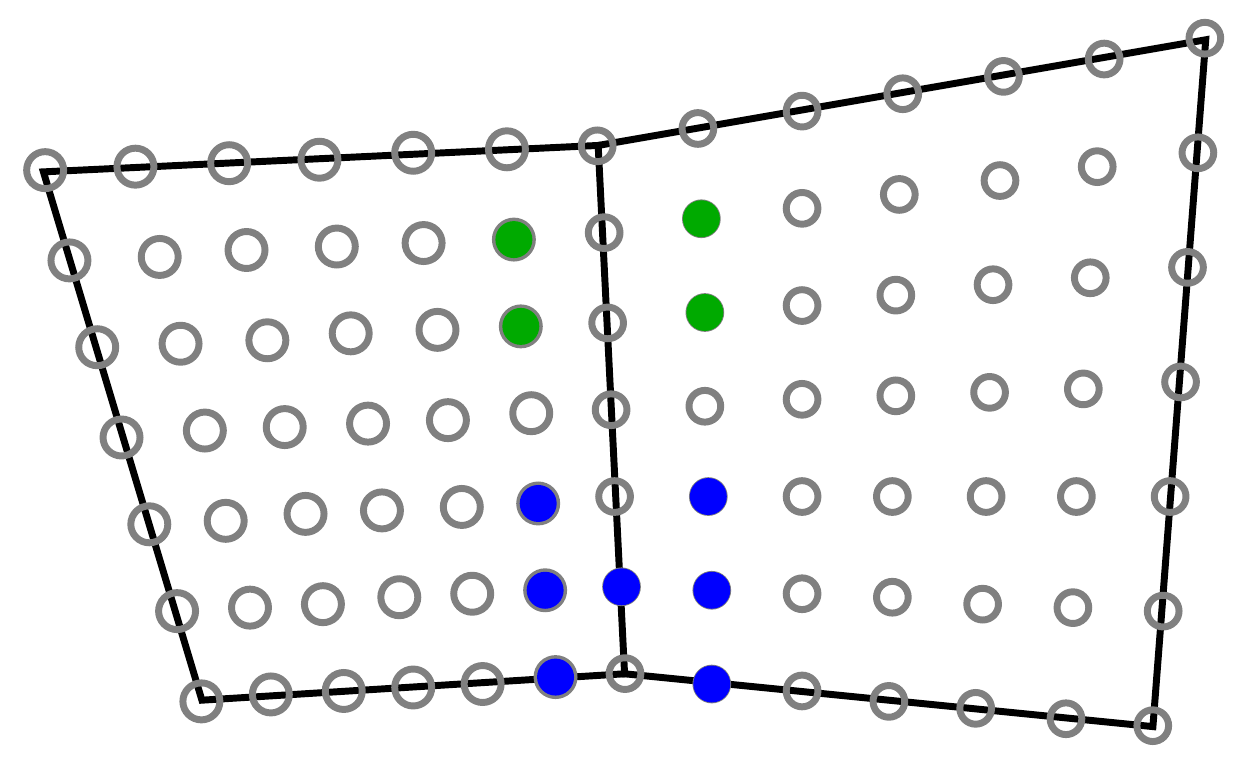}}
\caption{Pairs of elements with non-vanishing B\'ezier ordinates for given coefficients $c_1$ (in blue) and $d_4$ (in green). Note that the structure of non-vanishing ordinates 
is always the same, only shifted by the given index. In the given configurations we have $\pd=6$. The underlying control structure is plotted on the left: the coefficients $c_i$ (controlling function values along the interface) are depicted as circles and the coefficients $d_j$ (controlling crossing derivatives) as arrows. Each coefficient (control variable) $c_i$ or $d_j$ generates one function over the respective element pair.}
\label{fig:element_pairs}
\end{figure} 
As one can see in Figure~\ref{fig:element_pairs}, the $C^1$ functions across an interface couple degrees of freedom in a non-trivial way. The dimension of the $C^1$ space 
around a given vertex and the construction of a basis depends on the geometry, i.e. on the exact configuration of elements around the vertex. In order to simplify the 
construction, we demand $\C{2}$ continuity at vertices, thus fixing the dimension of the space and avoiding special cases. This strategy of imposing super-smoothness is a 
common tool for triangle meshes, see~\cite{ChHe89,LaSc07} or~\cite{KaSaTa19} for spline patches. This leads us to the interpolation problem as described in the following section.

\section{Analysis of the Argyris-like space}
\label{sec:SpaceAnalysis}

In order to analyze the properties of the space $\A_{\pd}$ (defined in \eqref{splineSpaceAm}), we formulate 
an interpolation problem that uniquely characterizes the elements of the spline space. The interpolation problem provides the dimension formula for $\A_{\pd}$ and gives rise to a projection operator that is used to prove the approximation properties of the space.

\subsection{Interpolation problem}

The following theorem states how the elements of $\A_{\pd}$ can be described in terms of interpolation data provided at the vertices, along the edges and in the interior of the mixed mesh. 

Recall that a point set is called \emph{unisolvent} in a function space, if any function in the space is uniquely determined by interpolating the values at the points in this set.

\begin{theorem} \label{Thm-intProblem}
Let $\Omega \subset \RR^2$ be an open domain on which a mixed mesh, satisfying \eqref{DomainQmega}, is defined. Then there exists a unique isoparametric spline function 
$\varphi \in \A_\pd$, $p\geq 5$, that satisfies the following interpolation conditions:
\begin{description}
\item[(A)]
For every vertex $\V{i}$, $i \in \indV$, let
\begin{align*}
\Der{x}{a}\Der{y}{b}\varphi\left(\V{i} \right) = \sigma_{i,a,b}, \quad 0\leq a+b \leq 2,
\end{align*} 
for some given values $\sigma_{i,a,b} \in \RR$.
\item[(B)]
For every edge $\bE^{(i)}$, $i \in \indE$,
choose $\pd-5$ pairwise different points $\bR^{(i)}_{\ell} \in \bE^{(i)}$,   
$\ell = 1,2,\dots,\pd-5$, as well as $\pd-4$ pairwise different points $\bS^{(i)}_{\ell} \in \bE^{(i)}$,   
$\ell = 1,2,\dots,\pd-4$, and let
\begin{align*}
 \varphi\left(\bR^{(i)}_{\ell} \right) = \sigma_{i, \ell}, \quad \ell = 1,2,\dots, \pd-5,\quad \parDer{\bfm{n}_i}{}\varphi \left(\bS^{(i)}_{\ell} \right) = w_{i, \ell}, 
 \quad \ell = 1,2, \dots, \pd-4, 
\end{align*} 
for some given values $\sigma_{i,\ell}, w_{i,\ell}\in\RR$. 
\item[(C)]
For every quadrilateral element~$\Omega^{(i)}$, $i \in \indQ$, choose $(\pd-3)^2$ pairwise different points 
\[
 \bQ^{(i)}_{\ell, k} = \bfm{F}^{(i)}\left({\widehat{\bQ}}^{(i)}_{\ell, k}\right), \quad \ell,k = 1,2,\dots, \pd-3,
\]
where ${\widehat{\bQ}}^{(i)}_{\ell, k}$ are unisolvent in 
$\mathrm{span}_{2\leq  j_1,j_2 \leq \pd-2}\left(\bernB{\pd,\pd}{j_1,j_2}\right)$,
and let
\begin{align*}
& \varphi\left(\bQ^{(i)}_{\ell, k} \right) = \sigmaBox_{ i, \ell,k }, \quad \ell,k = 1,2,\dots, \pd-3,
\end{align*} 
for some given values $\sigmaBox_{ i,\ell,k}\in\RR$. 
\item[(D)]
For every triangular element~$\Omega^{(i)}$, with $i \in \indT$, choose $\binom{\pd-4}{2}$ pairwise different points 
\[
 \bQ^{(i)}_{\ell, k} = \bfm{F}^{(i)}\left({\widehat{\bQ}}^{(i)}_{\ell, k}\right), \quad \ell = 1, 2, \dots, \pd-5, \quad k = 1,2,  \dots, \pd-4-\ell,
\]
where ${\widehat{\bQ}}^{(i)}_{\ell, k}$ 
are unisolvent in $\mathrm{span}_{\stackrel{2\leq  j_1,j_2,j_3}{j_1+j_2+j_3=p}}\left(\bernT{\pd}{j_1,j_2,j_3}\right)$,
and let 
\begin{align*}
& \varphi\left(\bQ^{(i)}_{\ell, k} \right) = \sigmaTri_{ i, \ell,k }, \quad \ell = 1, 2, \dots, \pd-5, \quad k = 1, 2, \dots, \pd-4-\ell,
\end{align*} 
for some given values $\sigmaTri_{ i,\ell,k}\in\RR$. 
\end{description}
\end{theorem}
\begin{proof}
We need to show that the interpolation conditions {\bf (A)}--{\bf (D)} uniquely determine the bivariate polynomial on every patch $\overline{\Omega^{(i)}} = 
\bfm{F}^{(i)}\left(\mathcal{D}^{(i)}\right)$, $i \in \indQ \cup \indT$, and that the continuity conditions are satisfied. 
Let $\V{j} $ be a vertex of $\overline{\Omega^{(i)}}$, obtained as
$\V{j} = \bfm{F}^{(i)}\left(u^{(i)}_j,v^{(i)}_j\right)$ for some 
$$\bfm{u}^{(i)}_j := \left(u^{(i)}_j,v^{(i)}_j\right) \in \left\{(0,0), (0,1), (1,0), (1,1) \right\}.$$  
From conditions in {\bf (A)} we get
$$
\grad{\varphi}\left(\V{j}\right) = \left(\sigma_{j,1,0}, \sigma_{j,0,1}\right), \quad
\Hess{\varphi}\left(\V{j}\right) = \begin{bmatrix}
\sigma_{j,2,0} & \sigma_{j,1,1}\\
\sigma_{j,1,1} & \sigma_{j,0,2}
\end{bmatrix}, 
$$
and from 
\begin{align*}
\grad{f^{(i)}}\left(\bfm{u}^{(i)}_j\right) & = \grad{\varphi}\left(\V{j}\right) \cdot \jac{\bfm{F}^{(i)}} \left(\bfm{u}^{(i)}_j\right) =:  
 \left(s^{(i)}_{j,1,0}, s^{(i)}_{j,0,1}\right) =: \bfm{s}^{(i)}_j,\\
\Hess{f^{(i)}}\left(\bfm{u}^{(i)}_j\right) & = \jac{\bfm{F}^{(i)}}\left(\bfm{u}^{(i)}_j\right)^{T} \cdot \Hess{\varphi}\left(\V{j}\right) \cdot 
\jac{\bfm{F}^{(i)}}\left(\bfm{u}^{(i)}_j\right) + \\
& + \Der{x}{}{\varphi}\left(\V{j}\right) \Hess{\bfm{F}_1^{(i)}}\left(\bfm{u}^{(i)}_j\right)+
\Der{y}{}{\varphi}\left(\V{j}\right) \Hess{\bfm{F}_2^{(i)}}\left(\bfm{u}^{(i)}_j\right) =:
\begin{bmatrix}
s^{(i)}_{j,2,0} & s^{(i)}_{j,1,1}\\
s^{(i)}_{j,1,1} &s^{(i)}_{j,0,2},
\end{bmatrix} =: S^{(i)}_j,
\end{align*} 
we obtain the $\C{2}$ interpolation conditions for $f^{(i)}$ at  $\bfm{u}^{(i)}_j$, i.e.,
\begin{align*}
\parDer{u}{a}\parDer{v}{b} f^{(i)} \left(\bfm{u}^{(i)}_j\right) = s^{(i)}_{j,a,b}, \quad 0\leq a+b \leq 2.
\end{align*} 
Further, let $\bE^{(k)}$, $k \in \indE$, be any edge of $\overline{\Omega^{(i)}}$, with boundary vertices 
$\V{k_0}$, $\V{k_1}$, $k_0, k_1 \in \indV$, parameterized as
$$
\bE^{(k)} = \defset{\bfm{F}^{(i)} \left(\bfm{\epsilon}^{(k)}(t)\right)}{t\in (0,1)}, \quad 
\bfm{\epsilon}^{(k)}(t):=  (1-t) \, \bfm{u}^{(i)}_{k_0} + t\, \bfm{u}^{(i)}_{k_1},
$$
and let $$\theta_k(t) := \sum_{\ell=0}^\pd c^{(k)}_\ell \bern{\pd}{\ell}(t) = 
f^{(i)} \left(\bfm{\epsilon}^{(k)}(t)\right), 
\quad \w_k(t) :=  \sum_{\ell=0}^{\pd-1} d^{(k)}_\ell \bern{\pd-1}{\ell}(t) = 
\parDer{\bfm{n}_k}{}\varphi \left( 
\bfm{F}^{(i)}\left(\bfm{\epsilon}^{(k)}(t)\right)
 \right), \quad t \in [0,1],
$$
be the restriction of $\varphi$ and $\parDer{\bfm{n}_k}{}\varphi $ on the edge $\bE^{(k)}$ expressed in local coordinates. Further, let
$t^{(i)}_{k,\ell}$ be the parameters, such that $\bR_{k, \ell} = 
\bfm{F}^{(i)}\left(\bfm{\epsilon}^{(k)}\left(t^{(i)}_{k,\ell}\right)\right)$. 
From {\bf (A)} and  {\bf (B)} we get $\pd+1$ conditions 
\begin{align*}
 & \theta_k(\ell) = \sigma_{k_\ell,0,0},  \quad  
 \theta'_k(\ell) = \sprod{\bfm{s}^{(i)}_{k_\ell}}
{\bfm{u}^{(i)}_{k_1}-\bfm{u}^{(i)}_{k_0}}, \quad
\theta''_k(\ell) =
\sprod{\bfm{u}^{(i)}_{k_1}-\bfm{u}^{(i)}_{k_0}}
{S^{(i)}_{k_\ell}
\left(\bfm{u}^{(i)}_{k_1}-\bfm{u}^{(i)}_{k_0}\right)^T},  \quad \ell = 0,1,\\
& \theta_k\left(t^{(i)}_{k,\ell}\right) = \sigma_{k, \ell}, \quad \ell=1,2,\dots, \pd-5,
\end{align*}
which uniquely determine $\theta_k$. Note that 
$\bfm{u}^{(i)}_{k_1}-\bfm{u}^{(i)}_{k_0} \in \left\{(\pm 1, 0), (0, \pm 1), \pm (1,-1) \right\}$.
From \eqref{directional-der} it is straightforward to see that {\bf (A)} and  {\bf (B)} give also the values of 
$\bfm{G}^{(k)}$, $\parDer{u}{}{\bfm{G}^{(k)}}$ and 
$\parDer{v}{}{\bfm{G}^{(k)}}$ at 
$\bfm{u}^{(i)}_{k_0}, \bfm{u}^{(i)}_{k_1}$, where $\bfm{G}^{(k)}$
is defined in  \eqref{directional-der2}. The conditions
\begin{align*}
\w_k(\ell) & = \sprod{\bfm{n}_k}{\bfm{G}^{(k)}\left(\bfm{u}^{(i)}_{k_{\ell}}\right)}, \quad \ell =0,1,\\
\w'_k(\ell)   & =  \left(u^{(i)}_{k_1} - u^{(i)}_{k_0}\right) 
\sprod{\bfm{n}_k}{\parDer{u}{}\bfm{G}^{(k)}\left(\bfm{u}^{(i)}_{k_{\ell}}\right)} + 
\left(v^{(i)}_{k_1} - v^{(i)}_{k_0}\right) 
\sprod{\bfm{n}_k}{\parDer{v}{}\bfm{G}^{(k)}\left(\bfm{u}^{(i)}_{k_{\ell}}\right)}, \quad \ell =0,1, \\
\w_k\left(t^{(i)}_{k,\ell}\right) & = w_{k, \ell}, \quad \ell=1,2,\dots, \pd-4,
\end{align*}
then uniquely determine $w_k$. 

Suppose now that $i \in \indQ$ and 
$
f^{(i)}(u,v) = \sum_{j, \ell=0}^\pd {b}^{(i)}_{j, \ell} \bern{\pd}{j}(u) \bern{\pd}{\ell}(v). 
$
Following Proposition~\ref{proposition-case1}, \ref{proposition-case2} and
\ref{proposition-case3} we see that polynomials $\theta_k$ and $w_k$ uniquely determine the ordinates 
${b}^{(i)}_{j, \ell}$ and ${b}^{(i)}_{\ell, j}$ for $\ell=0,1,\pd-1, \pd$, $j=0,1,\dots,\pd$. The remaining $(\pd-3)^2$ ordinates ${b}^{(i)}_{j, \ell}$, $j, \ell=2,3,\dots,\pd-2$, 
are computed uniquely from conditions  {\bf (C)}.  Similarly for 
 $i \in \indT$ and 
$
f^{(i)}(u,v) = \sum_{j+ \ell+r = \pd} {b}^{(i)}_{j, \ell,r} \bernT{\pd}{j, \ell,r}(u,v)
$.
Polynomials $\theta_k$ and $\w_k$ uniquely determine the ordinates ${b}^{(i)}_{j, \ell,r}$ with $j,\ell, r \in \{\pd, \pd-1\}$, while the remaining ones follow from conditions 
{\bf (D)}. Since the continuity conditions are satisfied by the construction, the proof is completed. 
\end{proof}

One possible choice for the interpolation points, which we further use in the examples, is the following. 
For a given edge~$\bE^{(i)}=\bE(\V{1},\V{2})$ we first compute an equidistant set of points 
$$\widetilde{\bR}_{\ell} = \frac{2 \left\lfloor\frac{p}{2}\right\rfloor-2-\ell}{2 \left\lfloor\frac{p}{2}\right\rfloor-2}\left(\frac{p-2}{p} \V{1} + \frac{2}{p} \V{2}\right) 
+\frac{\ell}{2 \left\lfloor\frac{p}{2}\right\rfloor-2} \left(\frac{2}{p} \V{1} + \frac{p-2}{p} \V{2}\right),
\quad \ell = 1, 2,\dots,  2\left\lfloor\frac{p}{2}\right\rfloor -3.
$$ 
Then for odd degree $p$ we choose
\begin{equation} \label{eq:definition-Ril-a}
\bR^{(i)}_{\ell} := \widetilde{\bR}_{\ell}, \quad \ell =1,\dots, \frac{p-5}{2}, \quad
\bR^{(i)}_{\ell} := \widetilde{\bR}_{\ell+1}, \quad \ell =\frac{p-3}{2}, \dots, p-5, \quad
\bS^{(i)}_{\ell} := \widetilde{\bR}_{\ell}, \quad \ell = 1,\dots, p-4, 
\end{equation}
and for even $p$
\begin{equation}  \label{eq:definition-Ril-b}
\begin{split}
& \bR^{(i)}_{\ell} := \widetilde{\bR}_{\ell}, \quad \ell =1,\dots, \frac{p-6}{2}, \quad
\bR^{(i)}_{\frac{p-4}{2}} := \widetilde{\bR}_{\frac{p-2}{2}}, \quad 
\bR^{(i)}_{\ell} := \widetilde{\bR}_{\ell+2}, \quad \ell =\frac{p-2}{2}, \dots, p-5, \\
& \bS^{(i)}_{\ell} := \widetilde{\bR}_{\ell}, \quad \ell =1,\dots, \frac{p-4}{2}, \quad
\bS^{(i)}_{\ell} := \widetilde{\bR}_{\ell+1}, \quad \ell =\frac{p-2}{2}, \dots, p-4.
\end{split}
\end{equation}
Additional interpolation points in the interior are chosen as  
\begin{equation}\label{eq:definition-Qil-quad}
\bQ^{(i)}_{\ell, k} := \bfm{F}^{(i)}\left( \frac{\ell}{\pd},  \frac{k}{\pd}  \right)
\in \Omega^{(i)}, \quad \ell,k = 2,\dots, \pd-2,
\end{equation}
for quadrilaterals and
\begin{equation}\label{eq:definition-Qil-tri}
\bQ^{(i)}_{\ell, k} := \bfm{F}^{(i)}\left( \frac{\ell}{\pd},  \frac{k}{\pd}  \right) \in \Omega^{(i)},
\quad \ell = 2,  \dots, \pd-2, \quad k = 2,  \dots, \pd-2-\ell,
\end{equation}
for triangles. For the graphical interpretation of these interpolation points in the case $p=8$ and $p=9$ see Fig.~\ref{fig:IntData}.

\subsection{Properties of the space $\A_\pd$}

Due to the interpolation conditions {\bf (A)}--{\bf (D)}, we have the following dimension formula (cf. Fig.~\ref{fig:IntData}).
\begin{corollary}\label{coro:dimension}
The dimension of the space $\A_\pd$ equals 
$$\dim \A_\pd = 6 \, |\indV|  + \left(2 \pd - 9\right) \, |\indE| + (\pd-3)^2 \, |\indQ| + \binom{\pd-4}{2} \, |\indT|.$$
\end{corollary}

Moreover, the space $\A_\pd$ contains bivariate polynomials of total degree $\pd$.
\begin{lemma}\label{lem:poly-reproduction}
We have $\poly^{2}_{\pd} \subset \A_\pd $.
\end{lemma}
This lemma follows directly from the definition of the space, hence the local space for triangles is equal to $\poly^{2}_{\pd}$, whereas it contains $\poly^{2}_{\pd}$ for quadrilaterals (see~\cite{KaSaTa20}). Another useful consequence of Theorem~\ref{Thm-intProblem}  is  the following: based on this theorem we define the global projection operator
\begin{equation}\label{eq:projector-global}
\proj_\pd: \C{2}\left(\overline{\Omega}\right) \to \A_\pd
\end{equation}
that assigns to every function $f \in  \C{2}\left(\overline{\Omega}\right)$ the $C^1$-spline $\proj_\pd f \in \A_\pd$ that satisfies the interpolation conditions {\bf (A)}--{\bf (D)} for data sampled from the function $f$. We can moreover define the local projection operators
\begin{equation}\label{eq:projector-local}
\proj^{(i)}_\pd: \C{2}\left(\overline{\Omega^{(i)}}\right) \to \A_\pd|_{\overline{\Omega^{(i)}}}.
\end{equation}
By definition we have
\begin{equation*}
\left(\proj_\pd \varphi\right)|_{\overline{\Omega^{(i)}}} = \proj^{(i)}_\pd \left(\varphi|_{\overline{\Omega^{(i)}}}\right).
\end{equation*}
The global and local projection operators $\proj_\pd$ and $\proj^{(i)}_\pd$, respectively, are defined by interpolation using the conditions specified in Theorem~\ref{Thm-intProblem}. In case of the global projector $\proj_\pd$ all interpolation conditions are needed. In case of the local projector $\proj^{(i)}_\pd$ only those interpolation conditions are needed, which are defined on $\overline{\Omega^{(i)}}$.

The operator $\proj_\pd$ is bounded, if all elements $\Omega^{(i)}$ of the mesh are shape regular. 
\begin{definition}\label{def:shape-regular}
 A mesh (or more precisely a sequence of refined meshes) is said to be \emph{shape regular}, with shape regularity parameter $\rho>0$, if for each triangle of the mesh all angles are bounded from below by $\rho$ and for each quadrilateral of the mesh all angles of all triangles obtained by splitting the quadrilateral along the two diagonals are bounded from below by $\rho$, see Figure~\ref{fig:SR}.
\end{definition}
\begin{figure}[htb]
\centering\footnotesize
\includegraphics[width=.4\textwidth]{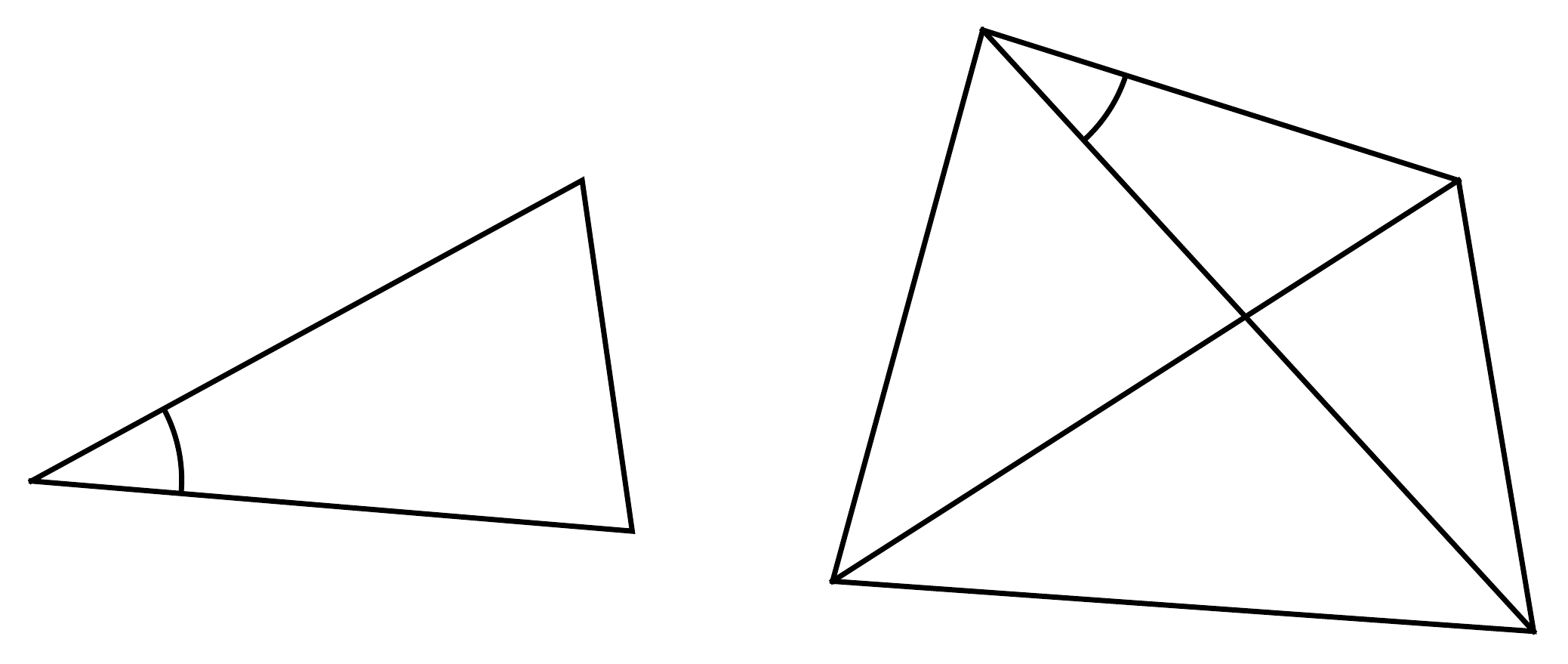}
\caption{A visualization of the shape regularity condition: the mesh is shape regular if the minimum of all angles as depicted in the two examples is uniformly bounded from below by $\rho$. The smallest angles for the triangle and quadrilateral are marked.}
\label{fig:SR}
\end{figure}

Given a shape regular mesh, there exists a constant $C(\rho)$, depending only on $\rho$, such that 
\[
 1\leq \frac{\max|\det\jac{\bfm{F}^{(i)}}|}{\min|\det\jac{\bfm{F}^{(i)}}|} \leq C(\rho)
\]
as well as 
\[
 1\leq \frac{\max_{1\leq i<j\leq\nu}|\V{i}-\V{j}|}{\min_{1\leq i<j\leq\nu}|\V{i}-\V{j}|} \leq C(\rho),
\]
where $\nu=3$ and $\Omega^{(i)} = \T(\V{1},\V{2},\V{3})$ for triangles and $\nu=4$ and $\Omega^{(i)} = \Q(\V{1},\V{2},\V{3},\V{4})$ for quadrilaterals.
In the following we denote by $|\cdot|_{H^\ell(D)}$ the $H^\ell$-seminorm, where $|\cdot|^2_{H^\ell(D)}$ is the sum of squares of $L^2$-norms of all derivatives of order $\ell$ over the domain~$D$, and by
\[
 \|\varphi\|_{H^m(D)} = \left(\sum_{\ell=0}^m |\varphi|^2_{H^\ell(D)}\right)^{\frac{1}{2}}
\]
the $H^m$-norm. We have the following lemma for shape regular elements.
\begin{lemma}\label{lem:boundedness}
Let $\Omega^{(i)}$ be any mesh element and let $\bR^{(i)}_{\ell}$, $\bS^{(i)}_{\ell}$ and $\bQ^{(i)}_{\ell,k}$ be defined by~\eqref{eq:definition-Ril-a}--\eqref{eq:definition-Qil-tri}. Then the local projector as defined in~\eqref{eq:projector-local} satisfies 
\begin{equation*}
 \| \proj^{(i)}_\pd \psi \|_{H^2(\Omega^{(i)})} \leq \sigma \|\psi \|_{C^2(\overline{\Omega^{(i)}})},
\end{equation*}
as well as 
\begin{equation*}
 \| \proj^{(i)}_\pd \psi \|_{L^\infty(\Omega^{(i)})} \leq \sigma \|\psi \|_{C^2(\overline{\Omega^{(i)}})},
\end{equation*}
where $ \sigma $ depends only on the degree~$\pd$, on the element size $\mathrm{diam}(\Omega^{(i)})$ and on the shape regularity constant $\rho$ and where $\| \psi \|_{C^2(\overline{\Omega^{(i)}})}$ takes the supremum of all derivatives up to second order on the element $\Omega^{(i)}$.
\end{lemma}
\begin{proof}
A proof of this lemma can be found in~\cite{KaSaTa20} for quadrilaterals. Here, for the sake of completeness, we repeat and extend it shortly to triangles. Assume that the degree $p$ is fixed. For each element $\Omega^{(i)}$ there exists a constant $C(\Omega^{(i)})\in \mathbb{R}$, such that 
\[
 \max\left( \| \proj^{(i)}_\pd \psi \|_{H^2(\Omega^{(i)})} , \| \proj^{(i)}_\pd \psi \|_{L^\infty(\Omega^{(i)})} \right) \leq C(\Omega^{(i)}) \|\psi \|_{C^2(\overline{\Omega^{(i)}})}.
\]
This follows directly from Theorem~\ref{Thm-intProblem} as the projector $\proj^{(i)}_\pd$ is well-defined, yields a bounded (mapped) polynomial and is completely determined by values, first and second derivatives of $\psi$ in $\overline{\Omega^{(i)}}$. Note that the definition of the projector is invariant with respect to translations of the element $\Omega^{(i)}$. Thus, the element may be moved such that one vertex is at the origin $\mathbf{O} = (0,0)^T$. We denote this translated element by $q$. The constant $C(\Omega^{(i)}) = C(q)$ depends only on the size and the shape of $q$ and not on the position of $\Omega^{(i)}$ within the mesh. Moreover, it depends continuously on the position of the vertices of $q$, since the definition of the projector depends continuously on the vertices. Since the set of all shape regular triangles and quadrilaterals $q$ which contain the origin in their boundary and have fixed size $\theta = \mathrm{diam}(q)$ is compact, the maximum of $C(q)$ over all shape regular $q$ of size $\theta$ exists and is attained for some triangle or quadrilateral $q'$, i.e.,
\[
 \max_{q \in E_{\rho,\theta}} C(q) = C(q') =: \sigma
\]
for
\[
E_{\rho,\theta}=\left\{q \subseteq \mathbb{R}^2: \; \mathbf{O} \in \partial q, \; q \mbox{ is a shape regular triangle or quadrilateral with constant }\rho\mbox{ and }\; \mathrm{diam}(q) = \theta\right\}.
\]
By construction, $\sigma$ depends only on $\rho$ and $\theta$ and on the degree $p$, which we have assumed to be fixed. This completes the proof.
\end{proof}
We can now prove approximation error bounds in standard Sobolev norms.
\begin{theorem}
 Let the mesh on $\Omega$ be shape regular and let $\Omega^{(i)}$ be an element of the mesh, let $\ell$ and $m$ be integers, with $0\leq \ell\leq 2$ and $4 \leq m \leq \pd+1$, and let the projector $\proj^{(i)}_\pd$ be defined as in~\eqref{eq:projector-local}. There exists a constant $C>0$ such that we have for all $\varphi \in H^m(\Omega^{(i)})$
 \[
  \left| \varphi - \proj^{(i)}_\pd \varphi \right|_{H^\ell(\Omega^{(i)})} \leq C \,  {h_{i}}^{m-\ell} \left| \varphi \right|_{H^m(\Omega^{(i)})},
 \]
 where $h_{i} = \mathrm{diam}(\Omega^{(i)})$. The constant $C$ depends on the shape regularity parameter $\rho$ and on $p$. We moreover have 
  \[
  \left\| \varphi - \proj^{(i)}_\pd \varphi \right\|_{L^\infty(\Omega^{(i)})} \leq C \,  {h_{i}}^{m} \left| \varphi \right|_{W^{m,\infty}(\Omega^{(i)})}
 \]
 for all $\varphi \in W^{m,\infty}(\Omega^{(i)})$, where $|\cdot|_{W^{m,\infty}(D)}$ takes the essential supremum of all derivatives of order $m$ over~$D$.
\end{theorem}

\begin{proof}
The proof follows the proof of~\cite[Theorem 4.4.4]{BrSc07} and~\cite[Theorem 4.6]{KaSaTa20}. Let, for now, $h_{i} = 1$. Then we have the following
\[
 \left| \varphi - \proj^{(i)}_\pd \varphi \right|_{H^\ell(\Omega^{(i)})} \leq \left\| \varphi - \proj^{(i)}_\pd \varphi \right\|_{H^\ell(\Omega^{(i)})} \leq \left\| \varphi - \psi_p \right\|_{H^\ell(\Omega^{(i)})} + \left\| \psi_p - \proj^{(i)}_\pd \varphi \right\|_{H^\ell(\Omega^{(i)})} 
\]
for any $\psi_p \in \poly^2_\pd$. Due to Lemma~\ref{lem:poly-reproduction} we have $\psi_p - \proj^{(i)}_\pd \varphi = \proj^{(i)}_\pd (\psi_p - \varphi)$. Hence, Lemma~\ref{lem:boundedness} yields for $\ell \leq 2$
\[
 \left\| \psi_p - \proj^{(i)}_\pd \varphi \right\|_{H^\ell(\Omega^{(i)})} = \left\| \proj^{(i)}_\pd (\psi_p - \varphi) \right\|_{H^\ell(\Omega^{(i)})} \leq \sigma \|\psi_p - \varphi \|_{C^2(\overline{\Omega^{(i)}})}.
\]
A standard Sobolev inequality~\cite[Lemma 4.3.4]{BrSc07} gives the bound
\[
 \|\psi_p - \varphi \|_{C^2(\overline{\Omega^{(i)}})} \leq C_{SI} \|\psi_p - \varphi \|_{H^4({\Omega^{(i)}})},
\]
where $C_{SI}$ depends only on the shape regularity parameter $\rho$. Altogether we obtain
\[
 \left| \varphi - \proj^{(i)}_\pd \varphi \right|_{H^\ell(\Omega^{(i)})} \leq (1+\sigma C_{SI}) \inf_{\psi_p \in \poly^2_\pd} \|\psi_p - \varphi \|_{H^4({\Omega^{(i)}})} \leq C |\varphi |_{H^m({\Omega^{(i)}})},
\]
with $C=(1+\sigma C_{SI}) C_{BH}$, where the last bound, for $4\leq m\leq p+1$, comes from a Bramble-Hilbert estimate, with constant $C_{BH}$, as in~\cite[Lemma 4.3.8]{BrSc07}. The dependence on the diameter $h_i$ of the element $\Omega^{(i)}$ follows from a standard homogeneity argument. The estimate for
\[
 \left\| \varphi - \proj^{(i)}_\pd \varphi \right\|_{L^\infty(\Omega^{(i)})} 
\]
follows similar steps as the $H^\ell$-estimate and is omitted here.
\end{proof}
From this local error estimate one can easily derive the following global estimate.
\begin{corollary}
 We assume to have a shape-regular, mixed mesh on $\Omega$. Let $\ell$ and $m$ be integers, with $0\leq \ell\leq 2$ and $4 \leq m \leq \pd+1$. There exists a constant $C>0$, depending on the degree $\pd$ and on the shape regularity constant $C_{SR}$, such that we have for all $\varphi \in H^m(\Omega)$
 \[
  \left| \varphi - \proj_\pd \varphi \right|_{H^\ell(\Omega)} \leq C \,  {h}^{m-\ell} \left| \varphi \right|_{H^m(\Omega)},
 \]
 as well as for all $\varphi \in W^{m,\infty}(\Omega)$
  \[
  \left\| \varphi - \proj_\pd \varphi \right\|_{L^\infty(\Omega)} \leq C \,  {h}^{m} \left| \varphi \right|_{W^{m,\infty}(\Omega)}.
 \]
 Here the projector $\proj_\pd$ is defined as in~\eqref{eq:projector-global} and $h$ denotes the length of the longest edge of the mesh.
\end{corollary}

Examples of interpolants $\proj_\pd f$ together with numerical observation of the approximation order are provided in the next section. 

\section{Numerical examples}
\label{sec:examples}

This section provides some numerical examples that confirm the derived theoretical results. In particular, 
the super-smooth $\C{1}$ Argyris-like space~$\A_\pd$ is employed for three different applications on several mixed triangle and quadrilateral meshes. The first two applications are 
the interpolation of a given function, based on the projection operator $\proj_\pd$, and the $L^2$-approximation. Both applications numerically verify the optimal 
approximation properties of the Argyris-like space $\A_\pd$. Thirdly, we
solve a particular fourth order PDE given by the biharmonic equation, which requires the use of globally $\C{1}$ functions for solving the PDE via its weak form 
and Galerkin discretization. 

\subsection{Mixed meshes, refinement strategy $\&$ super-smooth $C^1$ spline spaces}

We consider the three mixed triangle and quadrilateral meshes shown in Fig.~\ref{fig:meshes_examples-Mesh1}--\ref{fig:meshes_examples-Mesh3}~(first column) denoted by Mesh~$1$--$3$. The three mixed meshes are refined 
by splitting each triangle and each quadrilateral of the corresponding mesh into four triangles and into four quadrilaterals, respectively, as visualized in 
Fig.~\ref{fig:SplittingQuadTriangle}. As an example, the resulting refined meshes for the third level of refinement (i.e. for Level~$3$) are presented in Fig.~\ref{fig:meshes_examples-Mesh1}--\ref{fig:meshes_examples-Mesh3}~(second column). We further construct for all three meshes Argyris-like spaces~$\A_\pd$ as described in the previous section for the levels of refinement~$L=0,1,\ldots 5$, and denote 
the resulting super-smooth $\C{1}$ spline spaces by $\A_{\pd,h}$, where $h=\mathcal{O}\left(2^{-L}\right)$
is the length of the longest edge in the mesh.

\begin{figure}[htb]
\centering\footnotesize
 \begin{minipage}{0.51\textwidth}
\includegraphics[width=.9\textwidth]{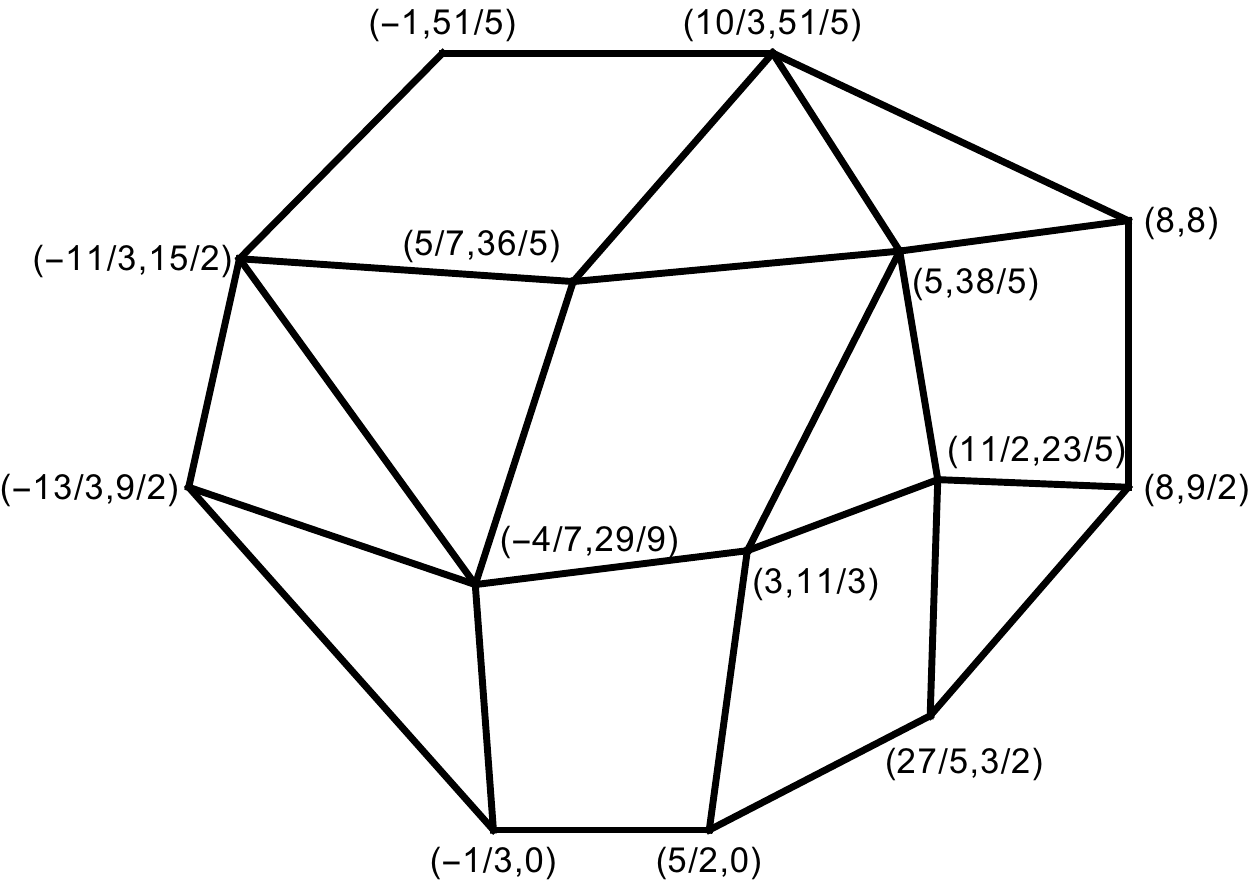}
\end{minipage}
 \begin{minipage}{0.42\textwidth}
\includegraphics[width=.9\textwidth]{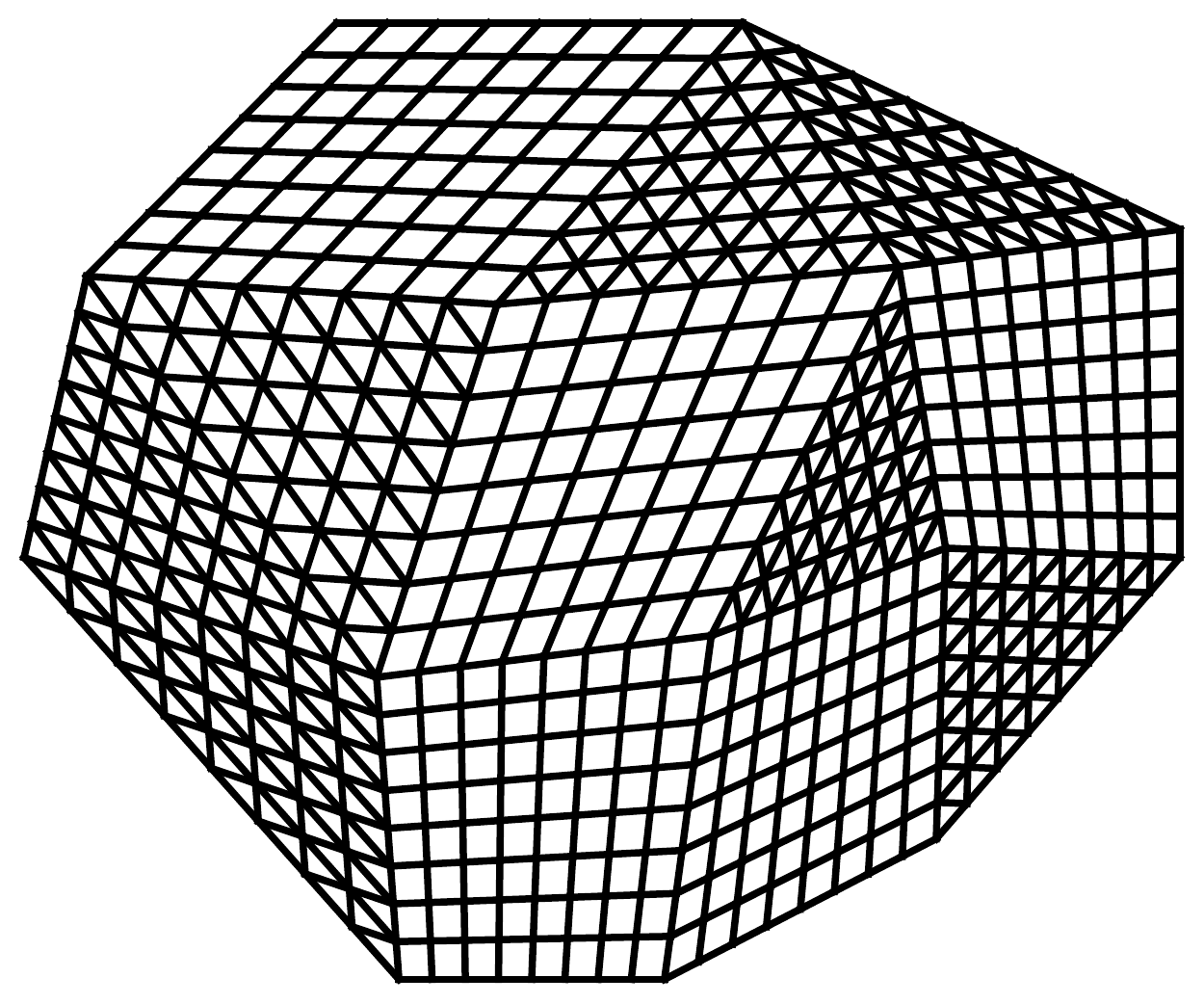}
\end{minipage}
\caption{Mesh~$1$ -- mixed triangle and quadrilateral mesh for the initial level of refinement (i.e. Level~$0$) with the coordinates of the vertices and for the third level of refinement (i.e. Level~$3$).}
\label{fig:meshes_examples-Mesh1}
\end{figure}
\begin{figure}[htb]
\centering\footnotesize
 \begin{minipage}{0.51\textwidth}
\includegraphics[width=.9\textwidth]{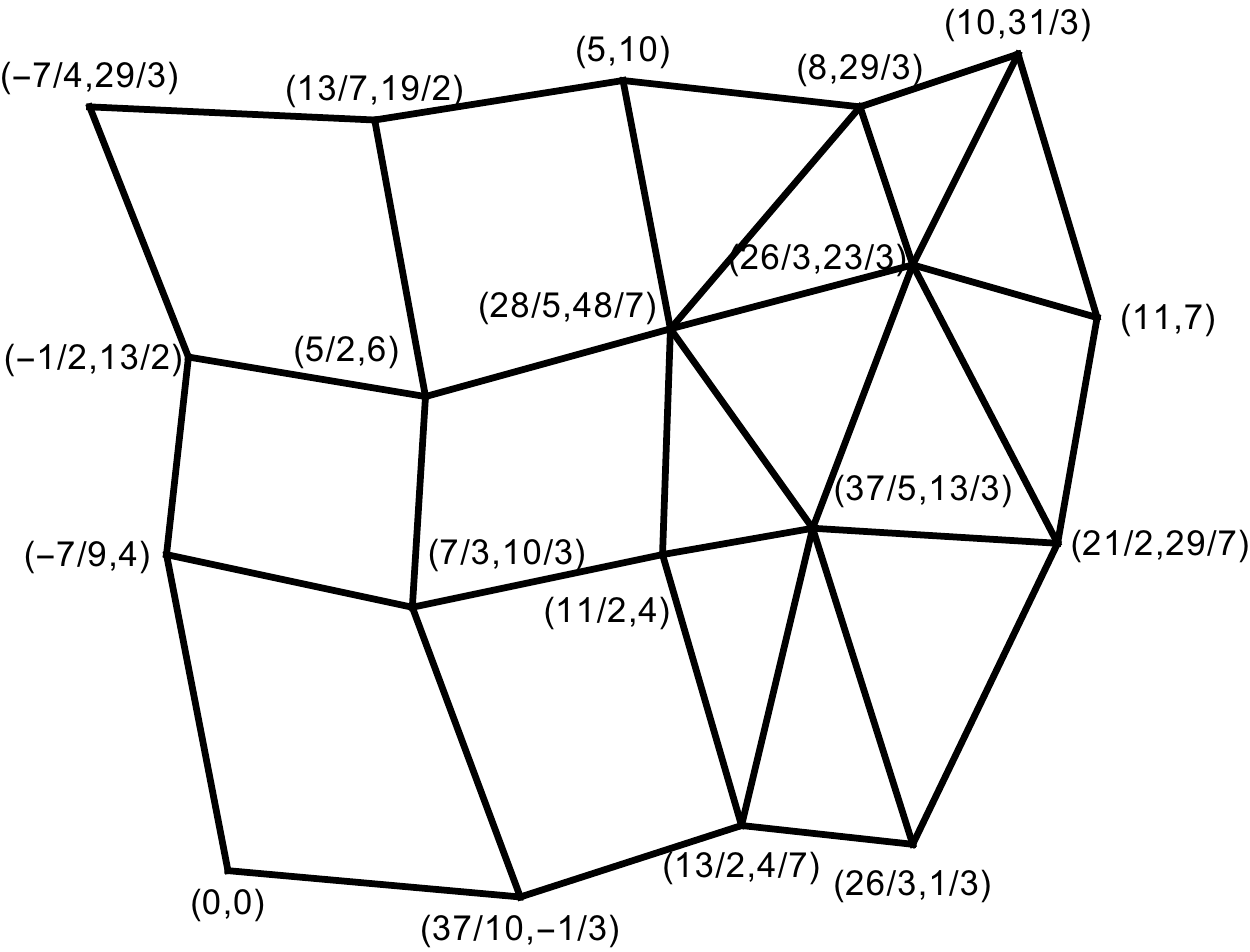}
\end{minipage}
 \begin{minipage}{0.42\textwidth}
\includegraphics[width=.9\textwidth]{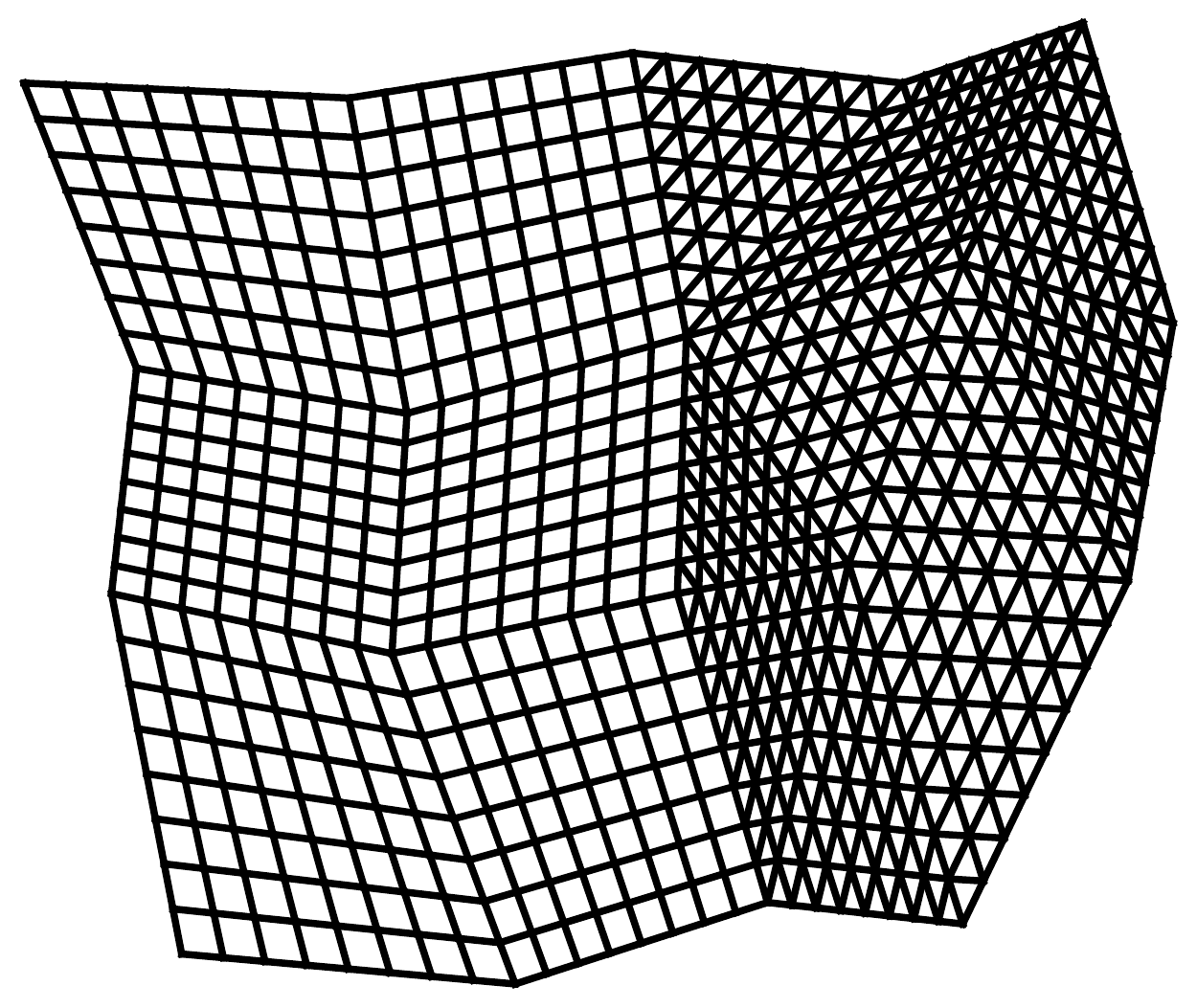}
\end{minipage}
\caption{Mesh~$2$ -- mixed triangle and quadrilateral mesh for the initial level of refinement with the coordinates of the vertices and for the third level of refinement.}
\label{fig:meshes_examples-Mesh2}
\end{figure}
\begin{figure}[htb]
\centering\footnotesize
 \begin{minipage}{0.51\textwidth}
\includegraphics[width=.9\textwidth]{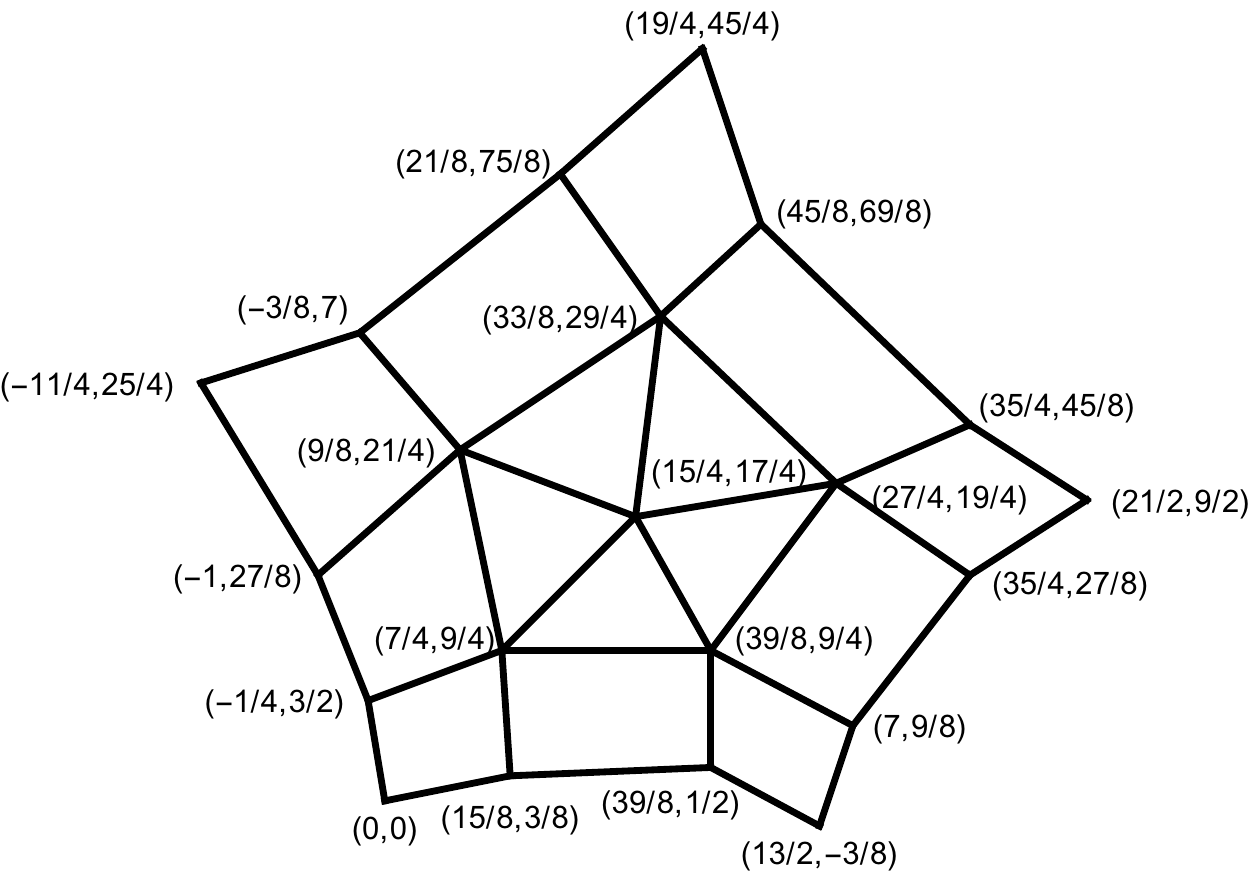}
\end{minipage}
 \begin{minipage}{0.42\textwidth}
\includegraphics[width=.9\textwidth]{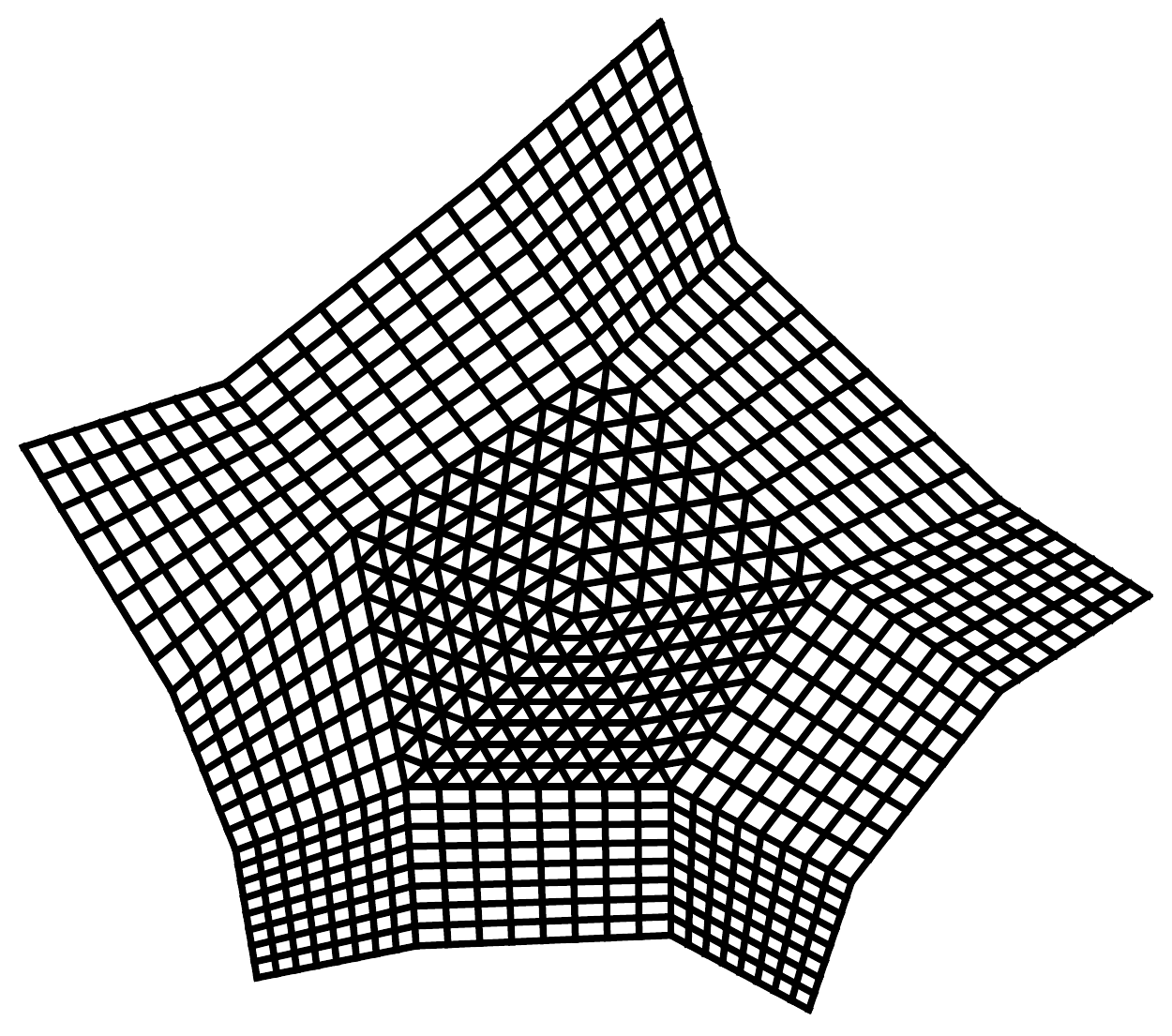}
\end{minipage}
\caption{Mesh~$3$ -- mixed triangle and quadrilateral mesh for the initial level of refinement with the coordinates of the vertices and for the third level of refinement.}
\label{fig:meshes_examples-Mesh3}
\end{figure}
\begin{figure}[htb]
 \centering\footnotesize
 \begin{tabular}{c c}
\includegraphics[width=4.5cm,clip]{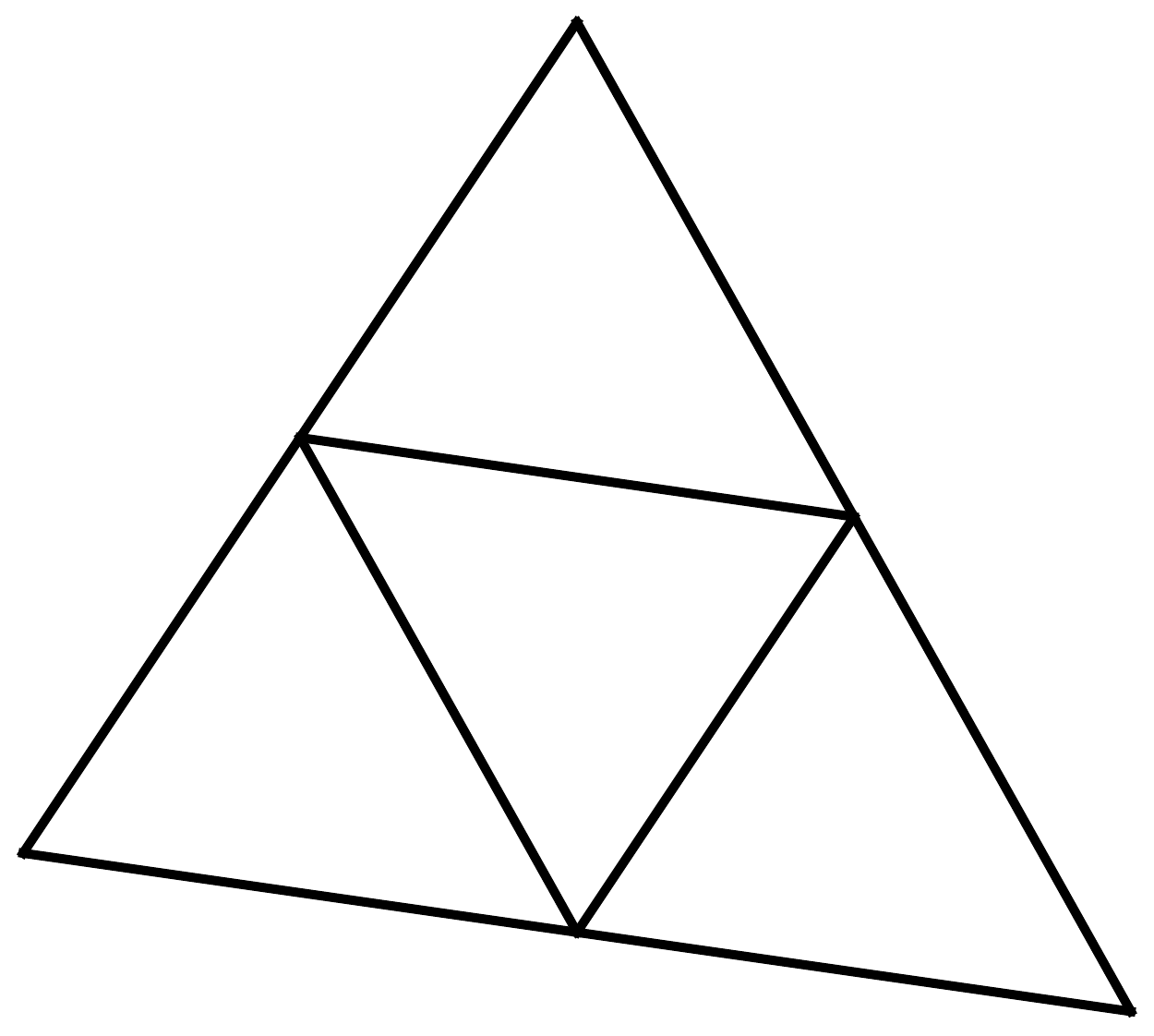} &
\includegraphics[width=4.0cm,clip]{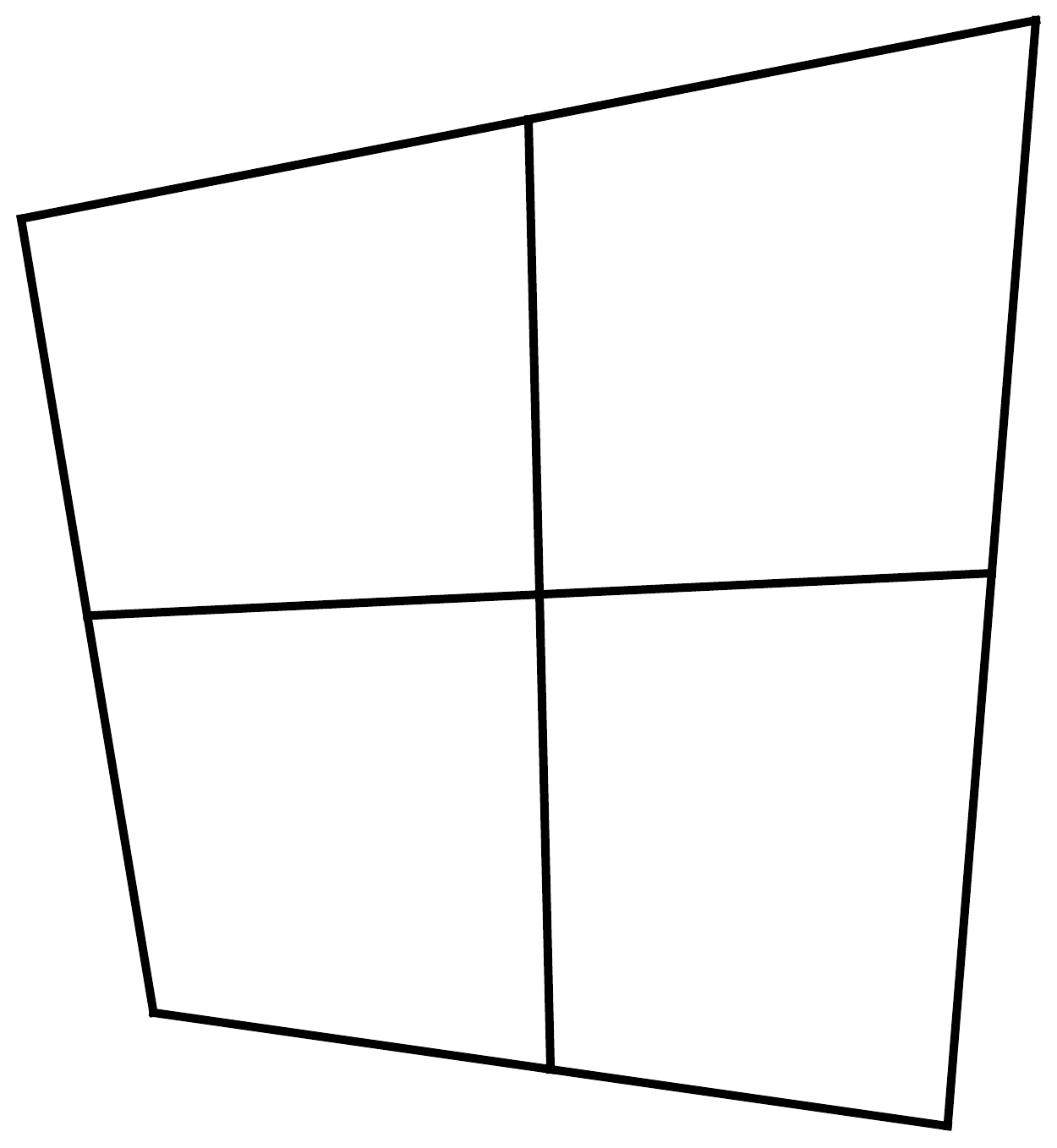}
\end{tabular}
\caption{Each triangle (left) and each quadrilateral (right) is split as shown into four triangles and into four quadrilaterals, respectively.}
\label{fig:SplittingQuadTriangle}
\end{figure}

\subsection{Interpolation}

To test the interpolation error we choose a smooth function
\begin{equation} \label{eq:exact_solutions}
 f(\bfm{x}) = f(x,y) = 4 \cos \left(\frac{2 x}{3}\right) \sin \left(\frac{2 y}{3}\right)
\end{equation}
and compute the interpolants $\proj_\pd f$ for degrees $\pd=5,6,\dots,10$ on mixed meshes  Mesh~$1$--$3$.
Fig.~\ref{fig:IntData} schematically presents the interpolation data we have used for degrees $\pd=8$ and $\pd=9$ 
when constructing interpolants over Mesh~$1$.  
\begin{figure}[htb]
 \centering\footnotesize
 \begin{minipage}{0.46\textwidth}
\includegraphics[width=0.8\textwidth]{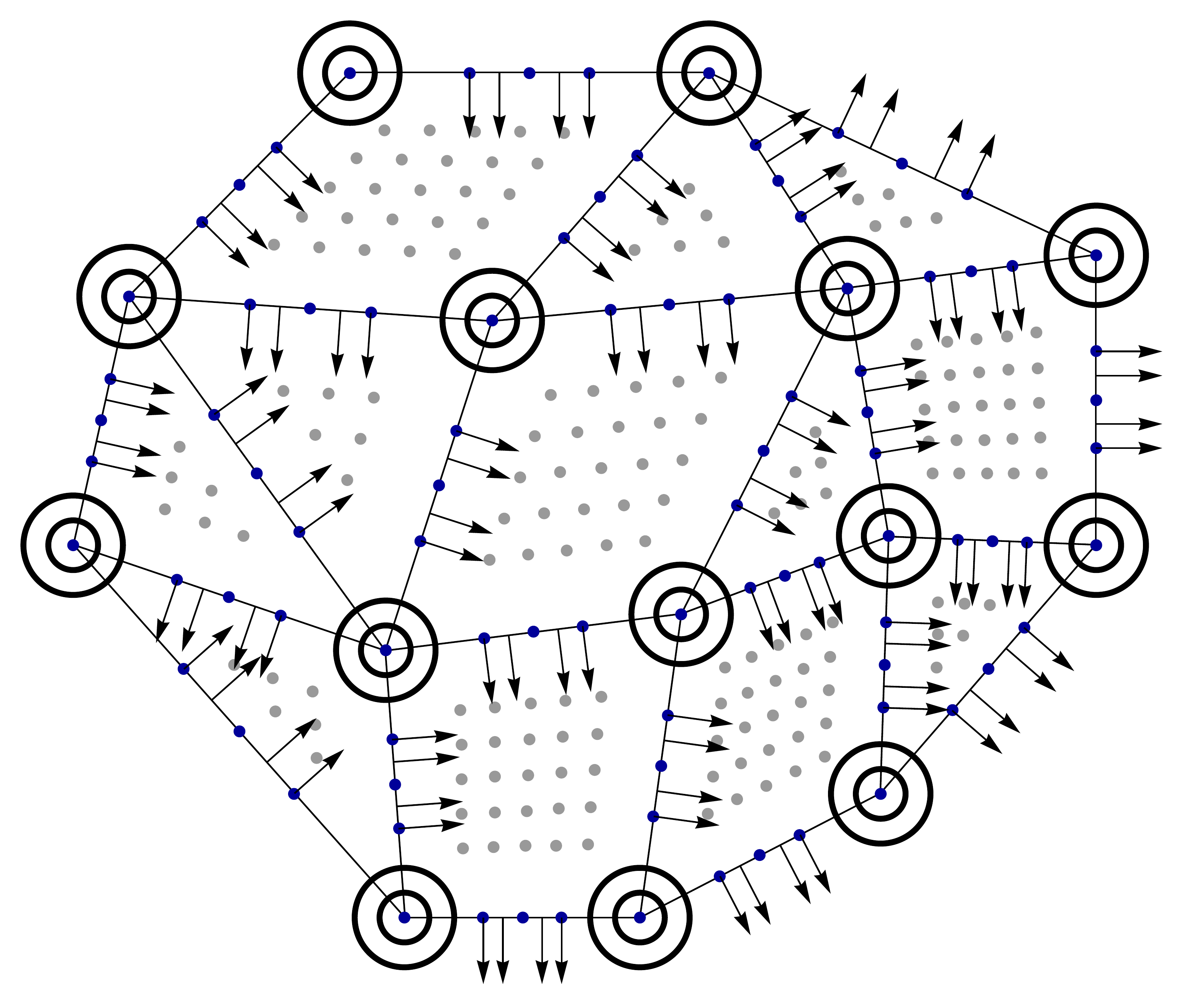}
\end{minipage}
\begin{minipage}{0.46\textwidth}
\includegraphics[width=0.8\textwidth]{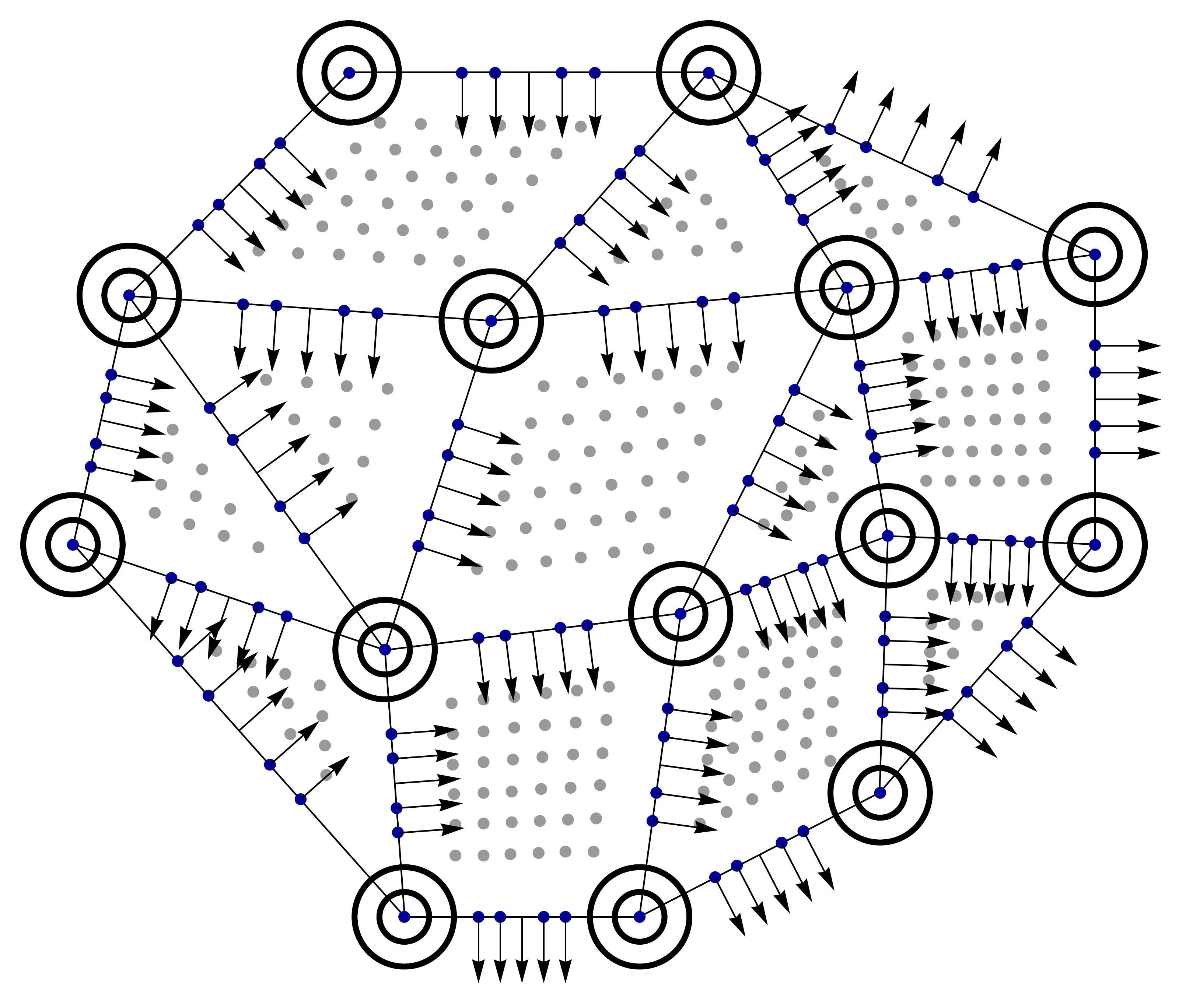}
\end{minipage}
 \caption{Schematic interpretation of interpolating data on Mesh~$1$ for degrees $\pd=8$ (left) and $\pd=9$ (right). }
\label{fig:IntData}
\end{figure}
\begin{figure}[htb]
 \centering\footnotesize
 \begin{minipage}{0.32\textwidth}
\includegraphics[width=1\textwidth]{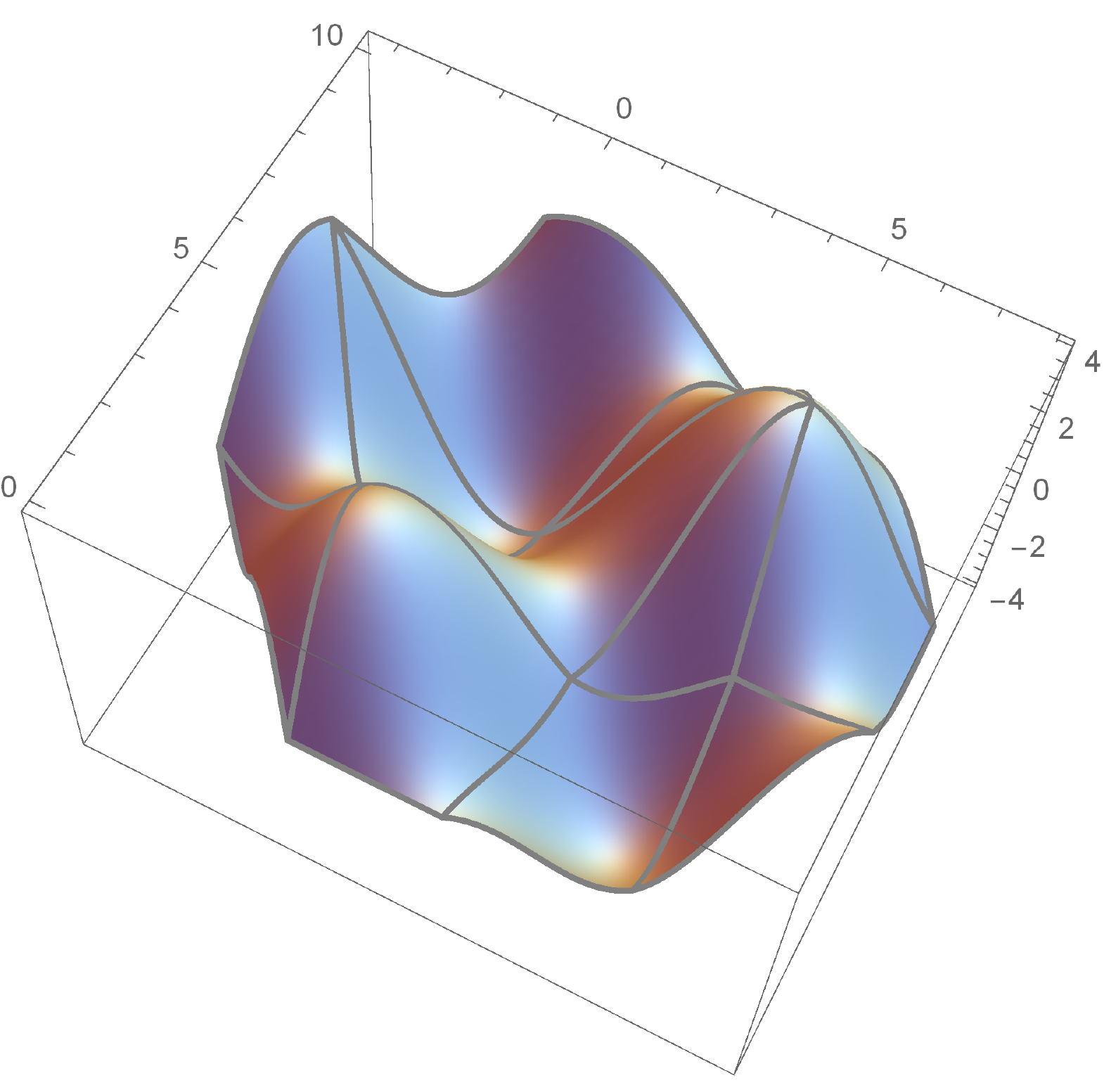}
\end{minipage}
\begin{minipage}{0.32\textwidth}
\includegraphics[width=1\textwidth]{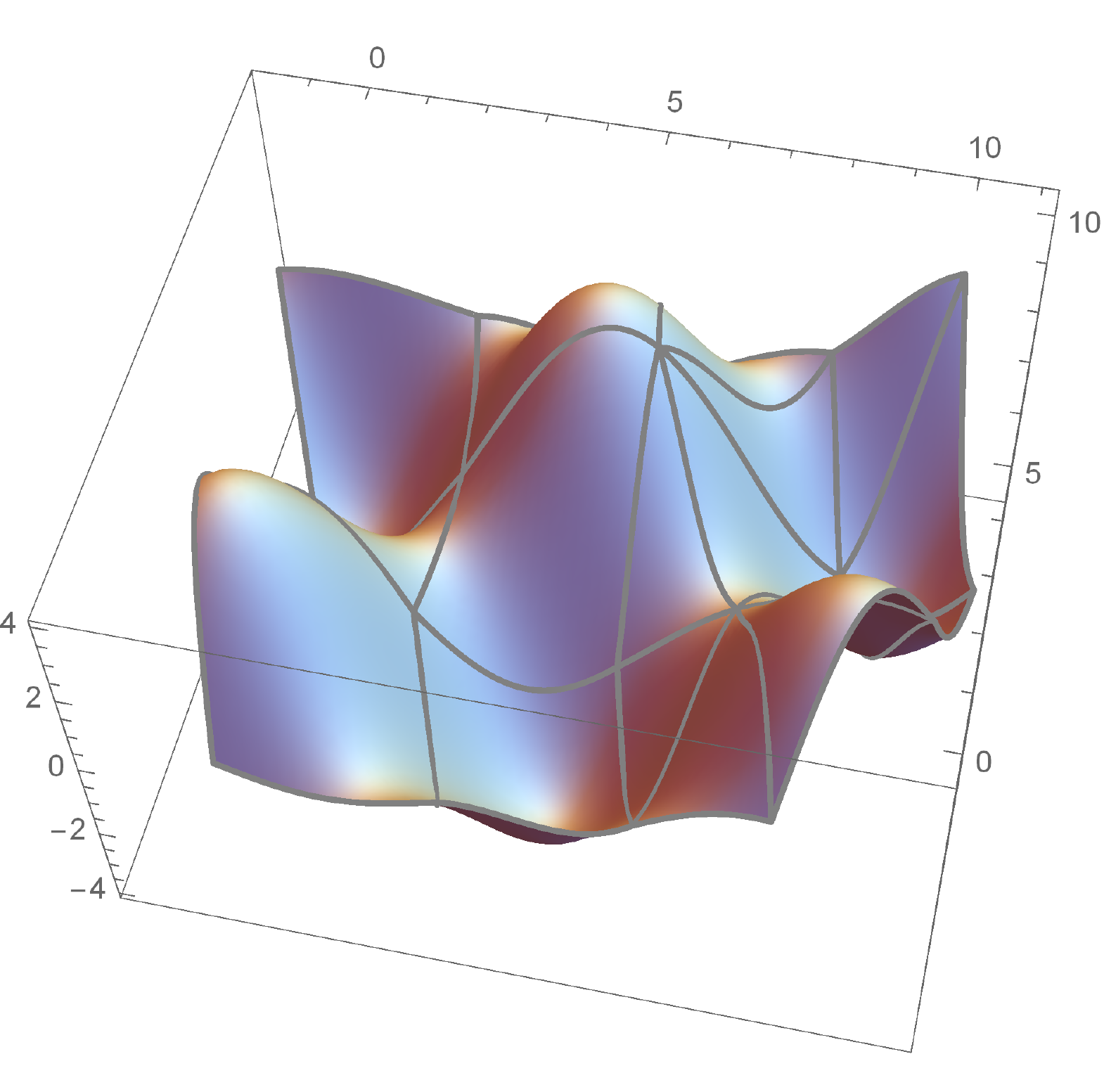}
\end{minipage}
\begin{minipage}{0.32\textwidth}
\includegraphics[width=1\textwidth]{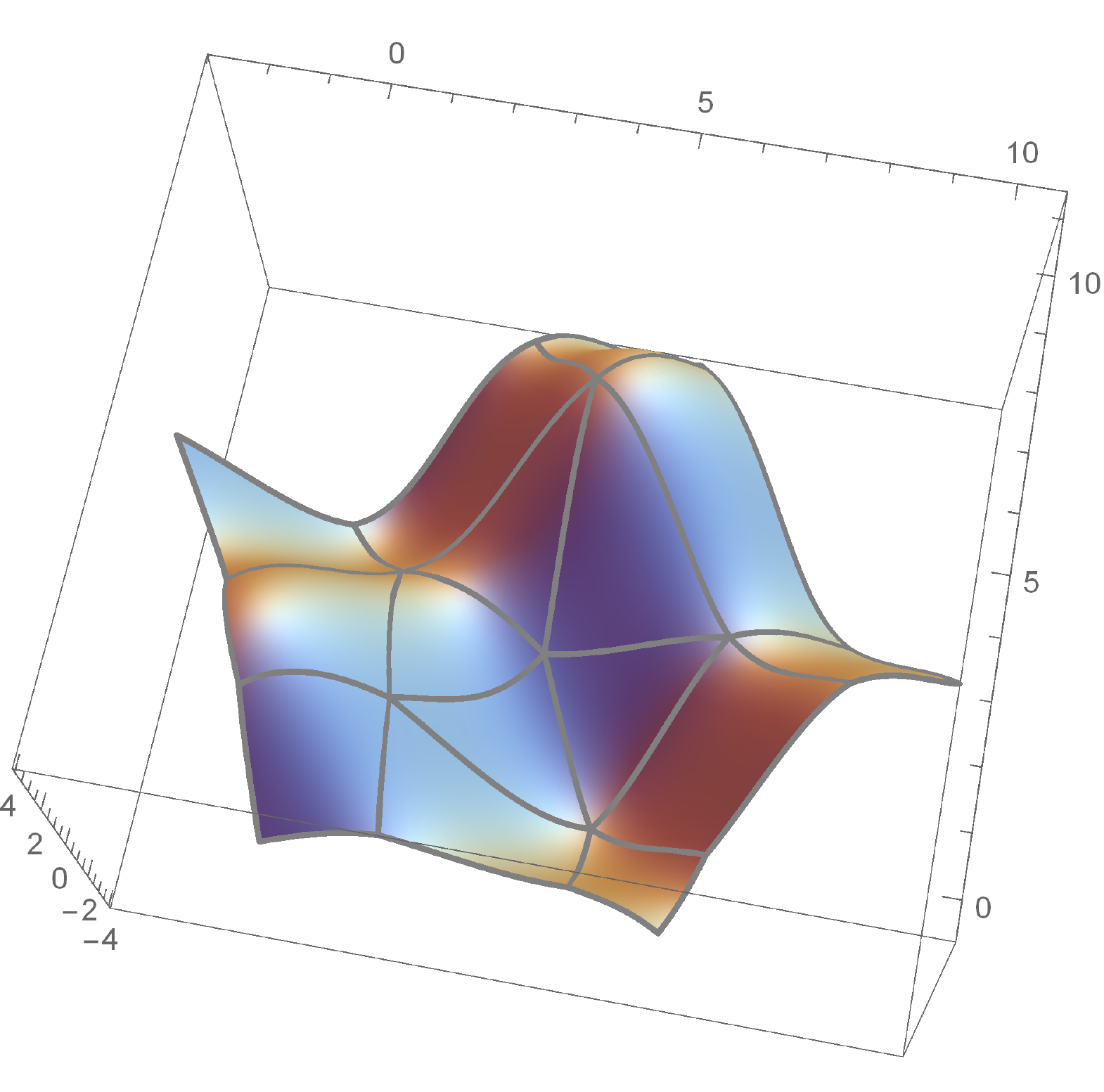}
\end{minipage}
 \caption{Interpolating splines $\proj_6 f$  for the function \eqref{eq:exact_solutions} over three meshes Mesh~$1$--$3$.}
\label{fig:IntDeg6}
\end{figure}
Interpolating splines $\proj_6 f$ over all three meshes are shown in Fig.~\ref{fig:IntDeg6}. 
To  compare between interpolants of different degrees  we use the $L^{\infty}$-error $\|f-\proj_\pd f\|_{\infty}$,
which we compute numerically by evaluating 
in $51^2=2601$ and $\binom{52}{2}=1326$ uniformly spaced points on every quadrilateral and triangle respectively. 
The errors for degrees 
$\pd=5,6,\dots, 10$ for all three meshes are given in Table~\ref{table-errorsDifferentDeg}.   
\begin{table}[htb]
\begin{center}
\begin{tabular}{|c|c|c|c|}
\hline
degree~$\pd$ & Mesh~$1$  & Mesh~$2$   & Mesh~$3$  \\
\hline
$5$ & $4.22291 \cdot10^{-1}$ & $1.01635 \cdot10^{-1}$ &  $3.17295 \cdot10^{-2}$ \\
$6$ & $9.72889\cdot10^{-2}$ & $3.98448\cdot10^{-3}$ & $1.16293\cdot10^{-2}$\\ 
$7$ & $4.65926\cdot10^{-3}$ & $6.25145\cdot10^{-4}$ & $1.91039\cdot10^{-4}$\\
$8$ & $1.99952\cdot10^{-3}$ & $3.25673\cdot10^{-5}$ & $1.17307\cdot10^{-4}$ \\
$9$ & $8.47799\cdot10^{-5}$ & $6.50782\cdot10^{-6}$ & $2.27334\cdot10^{-6}$\\
$10$ & $4.27483\cdot10^{-5}$ & $2.77481\cdot10^{-7}$ & $1.27041\cdot10^{-6}$\\
\hline
\end{tabular}
\caption{Table of errors $\|f-\proj_\pd f\|_{\infty}$ of interpolants of different degrees over meshes Mesh~$1$--$3$.}
\label{table-errorsDifferentDeg}
\end{center}
\end{table}

\begin{table}[htb]
\begin{center}
\begin{tabular}{|c||c  c|| c  c|| c c|}
\hline
 & \multicolumn{2}{c||}{$\pd=5$} & \multicolumn{2}{c||}{$\pd=6$} & \multicolumn{2}{c|}{$\pd=7$}\\
\hline
level $L$ & error $\mbox{err}_{5,L}$  &  $\gamma_{5,L}$  & error $\mbox{err}_{6,L}$  &  $\gamma_{6,L}$ & error $\mbox{err}_{7,L}$  &  $\gamma_{7,L}$   \\
\hline
level $0$ & $4.22291 \cdot 10^{-1}$ & / & $9.72889 \cdot 10^{-2}$ & / & $4.65926\cdot 10^{-3}$ & /  \\
level $1$ &  $2.00986\cdot 10^{-2}$ & $ 4.39307$ & $9.50777\cdot 10^{-4}$ & $6.67702$ & $6.53520\cdot 10^{-5}$ & $6.15573$ \\
level $2$ &  $3.40837\cdot 10^{-4}$ & $5.88187$ & $ 8.10833 \cdot 10^{-6}$ & $6.87356$ & $2.69193\cdot 10^{-7}$ & $7.92345$\\
level $3$ &  $5.44531 \cdot 10^{-6}$ &  $5.96792$ &  $6.54896 \cdot 10^{-8}$ & $6.95199$ & $1.06583 \cdot 10^{-9}$ & $7.98052$\\
\hline
\hline
& \multicolumn{2}{c||}{$\pd=8$} & \multicolumn{2}{c||}{$\pd=9$} & \multicolumn{2}{c|}{$\pd=10$}\\
\hline
level $L$ & error $\mbox{err}_{8,L}$  &  $\gamma_{8,L}$  & error $\mbox{err}_{9,L}$  &  $\gamma_{9,L}$ & error $\mbox{err}_{10,L}$  &  $\gamma_{10,L}$   \\
\hline
level $0$ & $1.99952 \cdot 10^{-3}$ &  / &  $8.47799\cdot 10^{-5}$ &  / & $4.27483\cdot 10^{-5}$ & / \\
level $1$ &  $4.54388\cdot 10^{-6}$ &  $8.78151$ & $3.68637\cdot 10^{-7}$ & $7.84538$ & $2.32038\cdot 10^{-8}$ & $10.84729$\\
level $2$ &  $9.46490\cdot 10^{-9}$ & $8.90712$ & $3.74188\cdot 10^{-10}$ & $9.94423$ & $1.19579\cdot 10^{-11}$ & $10.92218$ \\
level $3$ &  $1.89790\cdot 10^{-11}$ & $8.96204$ & $3.69027\cdot 10^{-13}$ & $9.98582$ & $5.97942\cdot 10^{-15}$ & $10.96567$\\
\hline
\end{tabular}
\caption{Table of $L^{\infty}$-errors for interpolants of different degrees $\pd$ over Mesh~$1$ with different levels of refinement $L$ together with 
estimates $\gamma_{\pd;L}$ of the decay exponent. }
\label{table-errDifferentLevels-1}
\end{center}
\end{table}
Further, to numerically observe the approximation order, we compute interpolants $\proj_\pd f$ over meshes with different refinement levels. Let us denote the 
$L^\infty$-error of the interpolant of degree $\pd$ on level $L$ by $\mbox{err}_{\pd,L}$. For Mesh $1$, these values are shown in Table~\ref{table-errDifferentLevels-1} 
together with the estimated decay exponents~$\gamma_{\pd,L}$. 
The values in Table~\ref{table-errDifferentLevels-1} numerically confirm that the approximation order for splines of degree $\pd$ is optimal, i.e. $\pd+1$.  Similar results 
hold true for the other two example meshes as one can see in Fig.~\ref{fig:DecayExp}, where the errors of interpolants of different degrees over different refinement levels 
are plotted in log--log-scale in dependence on the number of degrees of freedom.
\begin{figure}[htb]
 \centering\footnotesize
\begin{minipage}{0.3\textwidth}
\includegraphics[width=1\textwidth]{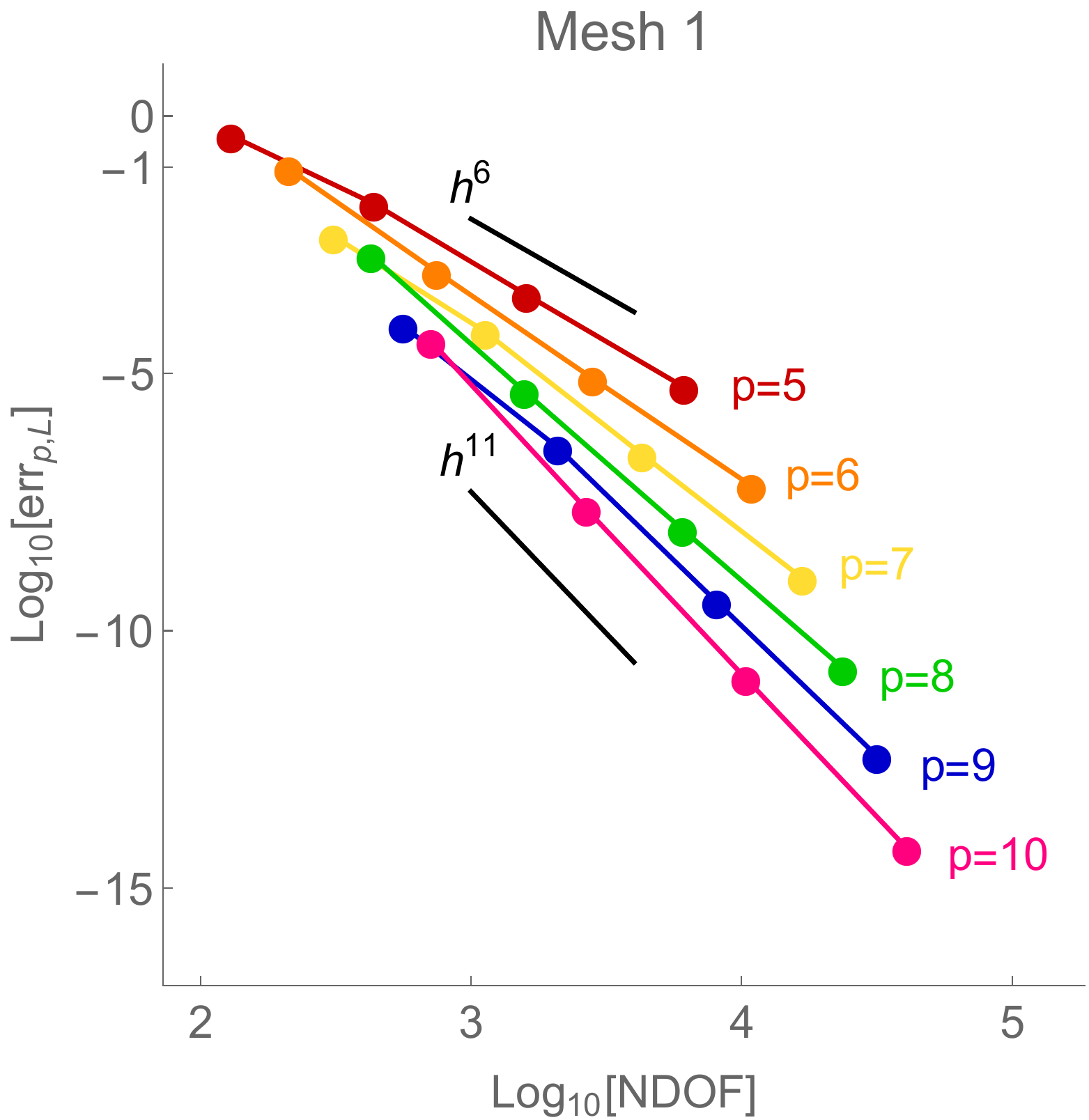}
\end{minipage}
\begin{minipage}{0.3\textwidth}
\includegraphics[width=1\textwidth]{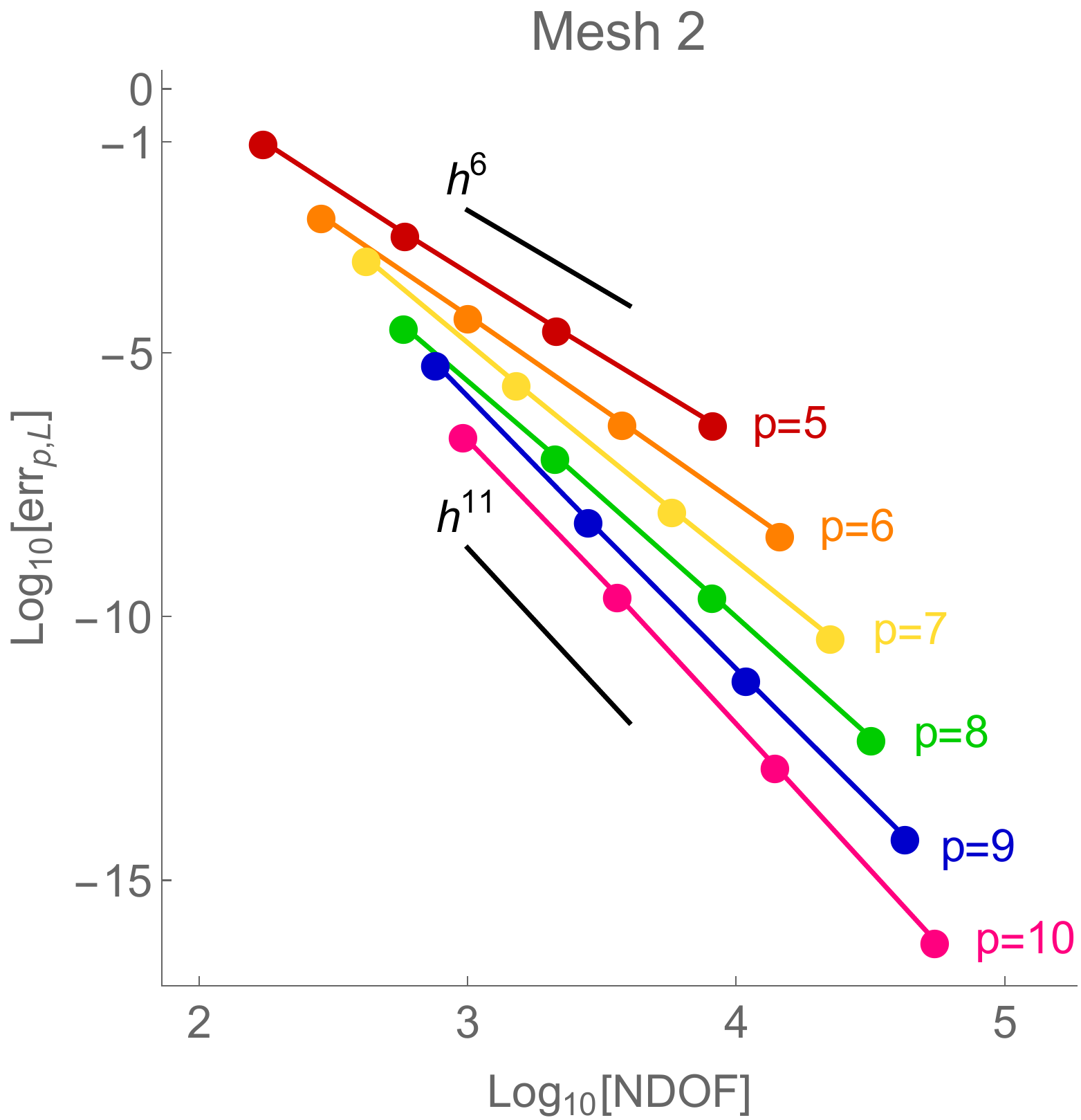}
\end{minipage}
\begin{minipage}{0.3\textwidth}
\includegraphics[width=1\textwidth]{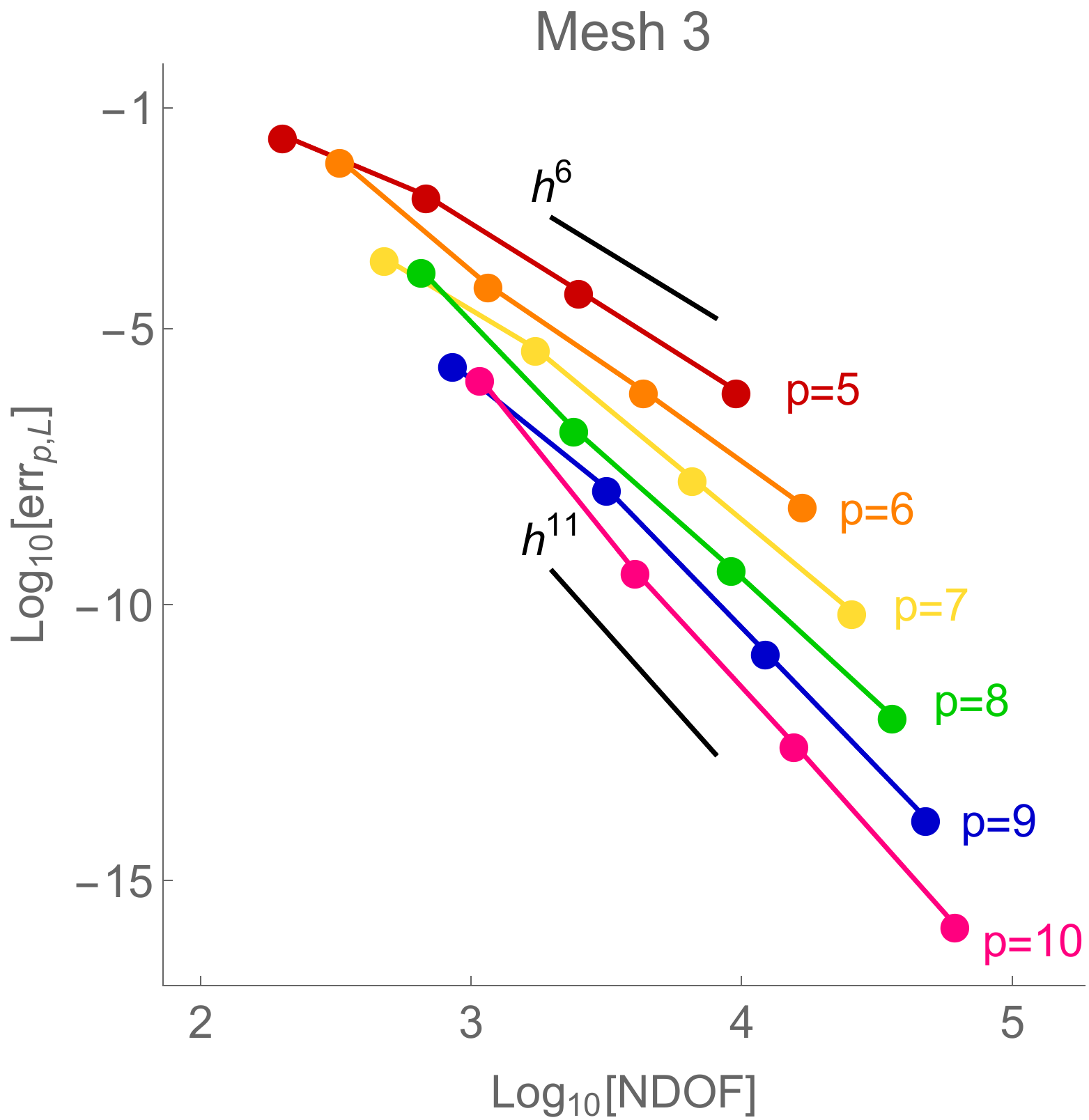}
\end{minipage}
 \caption{$L^{\infty}$-errors of interpolants of different degrees $\pd$ over meshes Mesh~$1$--$3$ with different refinement levels. Errors are shown in log--log-scale in dependence on the number of degrees of freedom.}
\label{fig:DecayExp}
\end{figure}

\subsection{$L^2$-approximation}

We use the constructed super-smooth $\C{1}$ spline spaces~$\A_{5,h}$ to approximate in a least-squares sense the function \eqref{eq:exact_solutions} on meshes Mesh~$1$--$3$. 
That is, we compute in each case that function~$f_h \in \A_{5,h}$, 
which minimizes the objective function
\[
 \int_{\Omega} \left(f_h(\bfm{x})- f(\bfm{x})\right)^2 d \bfm{x}.
\]
Fig.~\ref{fig:examples} (second row) reports the resulting $L^{\infty}$-errors as well as the resulting $L^2$-errors for different levels of refinement with respect to 
the number of degrees of freedom (NDOF). The obtained results indicate that both errors decrease with rates of optimal order of $\mathcal{O}(h^6)$, which numerically 
verify the optimal approximation power of the super-smooth $\C{1}$ Argyris-like space~$\mathcal{A}_{5}$.

\subsection{Solving the biharmonic equation}

We solve a particular fourth order PDE, namely the biharmonic equation
\begin{equation} \label{eq:problem_biharmonic}
\left\{
 \begin{array}{rll} 
 \Delta^{2} u(\bfm{x}) & = g(\bfm{x}) & \bfm{x} \in \Omega \\
  u(\bfm{x}) & = g_1(\bfm{x})  & \bfm{x} \in \partial \Omega \\
  \frac{\partial u}{\partial \bfm{n}}(\bfm{x})  & = g_2(\bfm{x}) &  \bfm{x} \in \partial \Omega
        \end{array} \right. 
\end{equation}
on the Meshes~$1$--$3$ by employing the super-smooth $\C{1}$ spline spaces~$\A_{5,h}$ as discretization spaces. For all three meshes, 
the functions $g$, $g_1$ and $g_2$ are derived from the same exact solution~\eqref{eq:exact_solutions} 
as before for the case of $L^2$-approximation. The biharmonic equation~\eqref{eq:problem_biharmonic} is solved via its weak form and Galerkin projection by at 
first strongly imposing the Dirichlet boundary conditions to the numerical solution~$u_h \in \A_{5,h}$. 
Let $\widehat{\A}_{5,h}$ and 
$\widetilde{\A}_{5,h}$ be the two isogeometric spline spaces defined via
\[
 \widehat{\A}_{5,h} = \defset{\varphi \in \A_{5,h}}{\varphi(\bfm{x})=\frac{\partial \varphi}{\partial \bfm{n}}(\bfm{x})=0, \mbox{ }\bfm{x} \in \partial \Omega} 
\]
and
\[
 \A_{5,h}=\widehat{\A}_{5,h} \oplus \widetilde{\A}_{5,h},
\] 
respectively. 
We aim at finding for the biharmonic equation~\eqref{eq:problem_biharmonic} an approximated solution~$u_h = \widehat{u}_h+\widetilde{u}_h$ 
with $\widehat{u}_h \in \widehat{\A}_{5,h}$ and $\widetilde{u}_h \in \widetilde{\A}_{5,h}$, where $\widetilde{u}_h$ is at first computed via $L^2$-projection 
to approximately satisfy the Dirichlet boundary conditions, and is then used to find~$\widehat{u}_h$ by solving the problem
\begin{equation} \label{eq:weak_problem}
 \int_{\Omega} \Delta \widehat{u}_h(\bfm{x}) \Delta \widehat{v}_h(\bfm{x}) \mathrm{d}\bfm{x} = 
 \int_{\Omega} g(\bfm{x}) \widehat{v}_h(\bfm{x}) \mathrm{d}\bfm{x} -
 \int_{\Omega} \Delta \widetilde{u}_h(\bfm{x}) \Delta \widehat{v}_h(\bfm{x}) \mathrm{d}\bfm{x}
\end{equation}
for all~$\widehat{v}_h \in \widehat{\A}_{5,h}$. 
An isogeometric formulation of the problem~\eqref{eq:weak_problem} can be found e.g. 
in \cite{BaDeQu15, KaViJu15}.

The resulting relative $L^2$-, $H^1$- and $H^2$-errors for the different levels of refinement, again with respect to the number 
of degrees of freedom, are shown in Fig.~\ref{fig:examples} (third row). The estimated convergence rates are for all examples of optimal order of $\mathcal{O}(h^6)$, $\mathcal{O}(h^5)$ 
and $\mathcal{O}(h^4)$ with respect to the $L^2$-, $H^1$- and $H^2$-norm, respectively.

\begin{figure}[htb]
\centering\footnotesize
\begin{tabular}{ccc}
\includegraphics[width=5.0cm,clip]{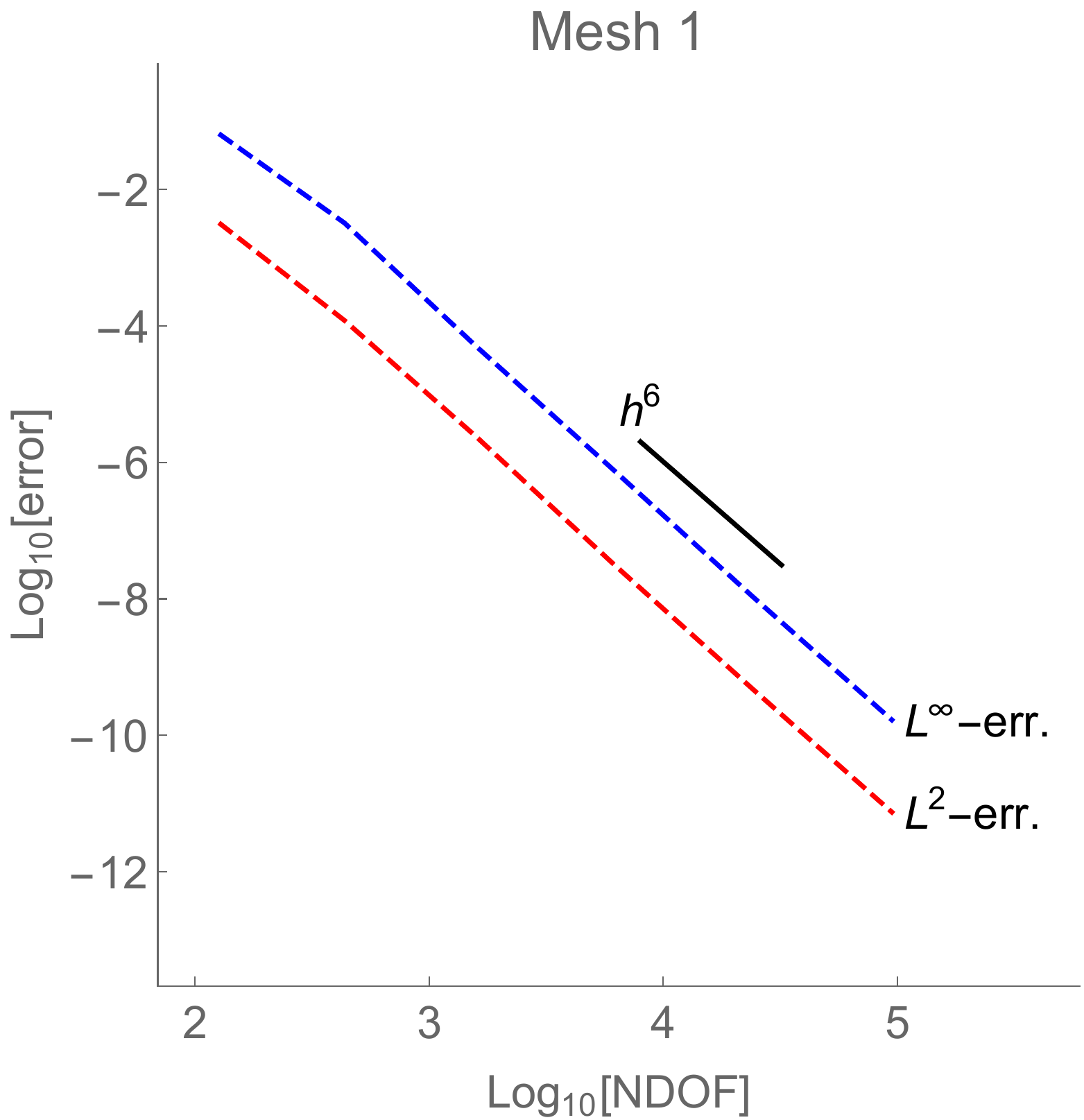} &
\includegraphics[width=5.0cm,clip]{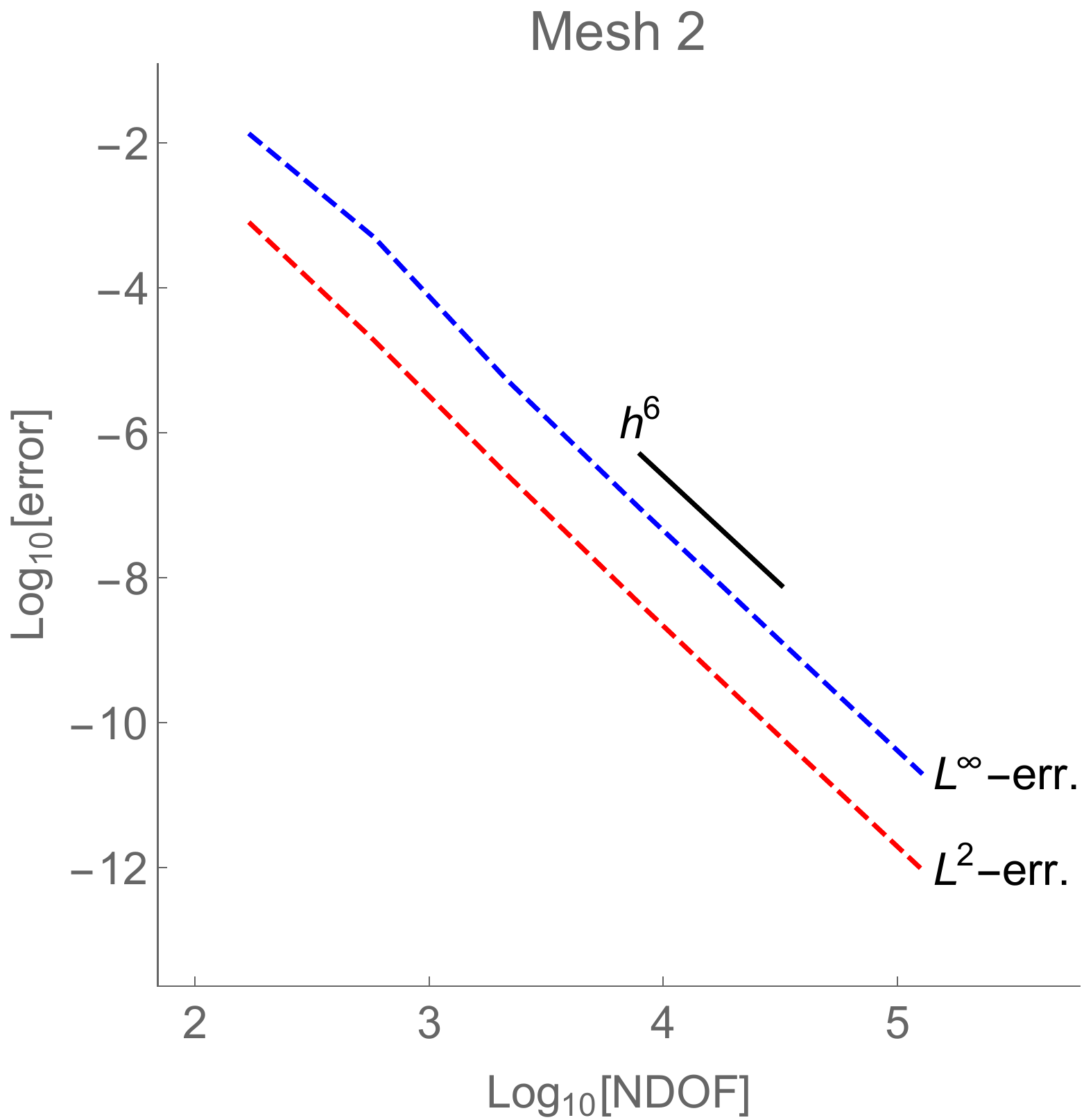} &
\includegraphics[width=5.0cm,clip]{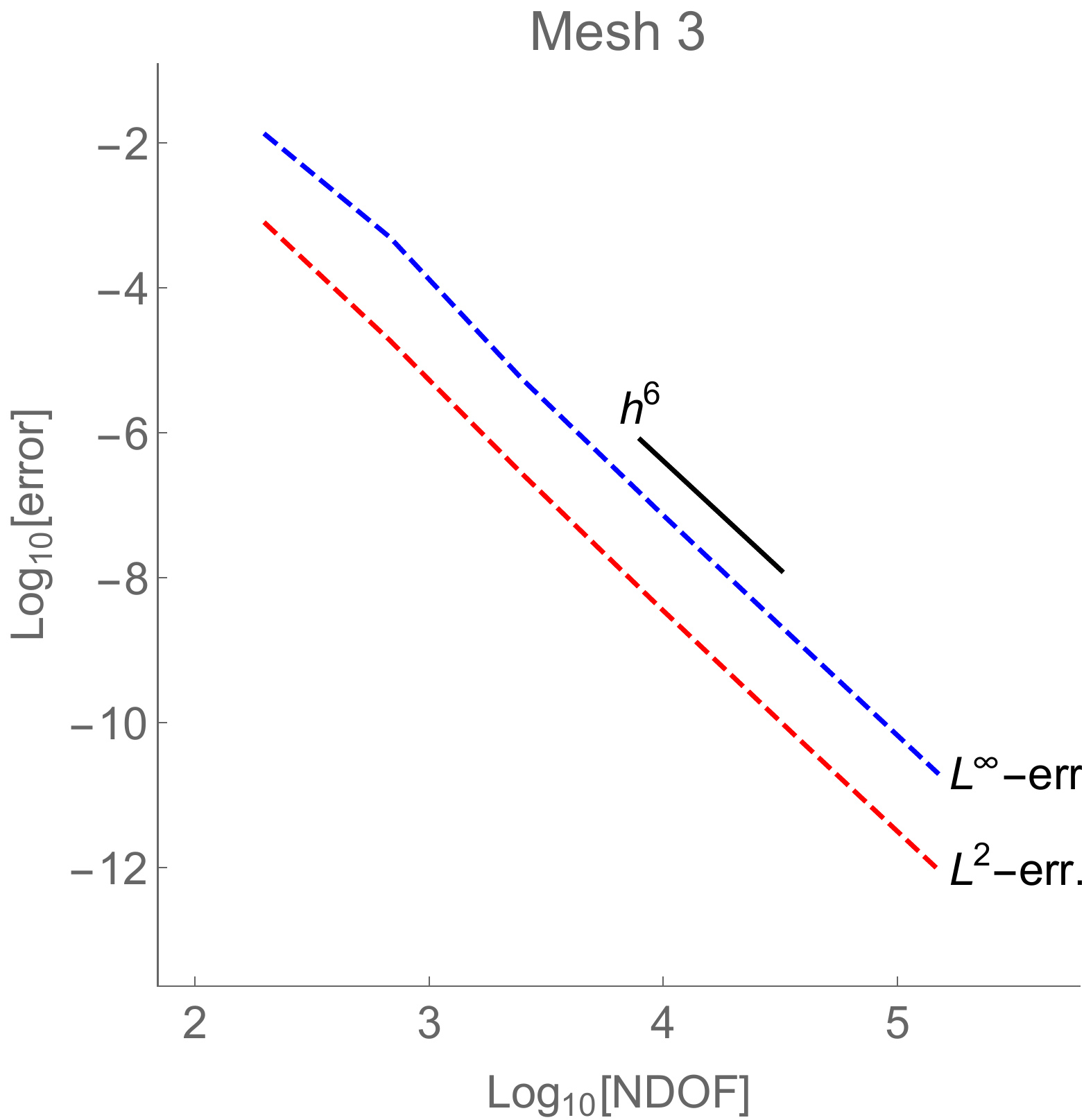} \\
\multicolumn{3}{c}{Performing $L^2$-approximation: resulting $L^{\infty}$-errors and resulting relative $L^2$-errors}\\
\includegraphics[width=5.0cm,clip]{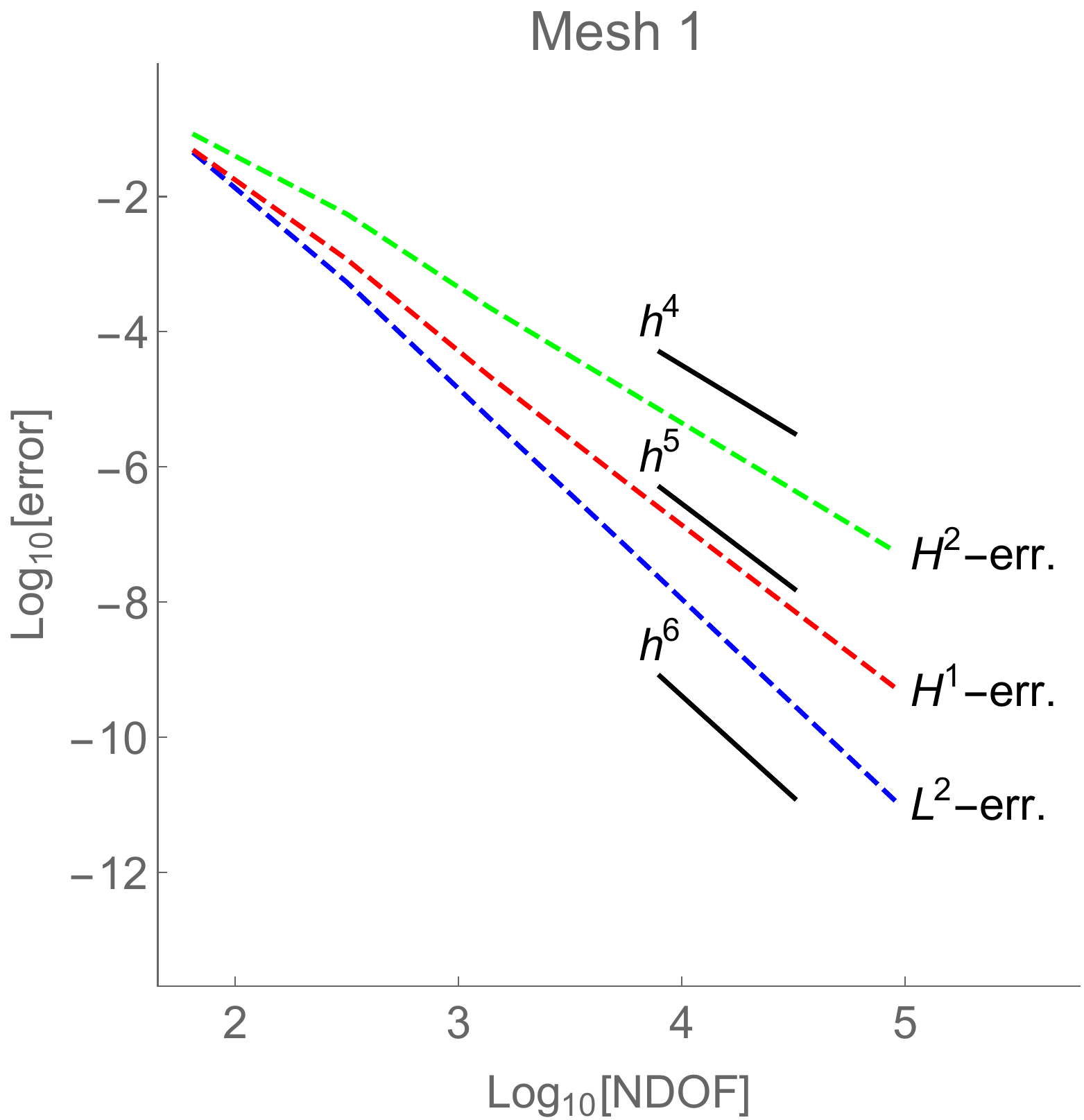} &
\includegraphics[width=5.0cm,clip]{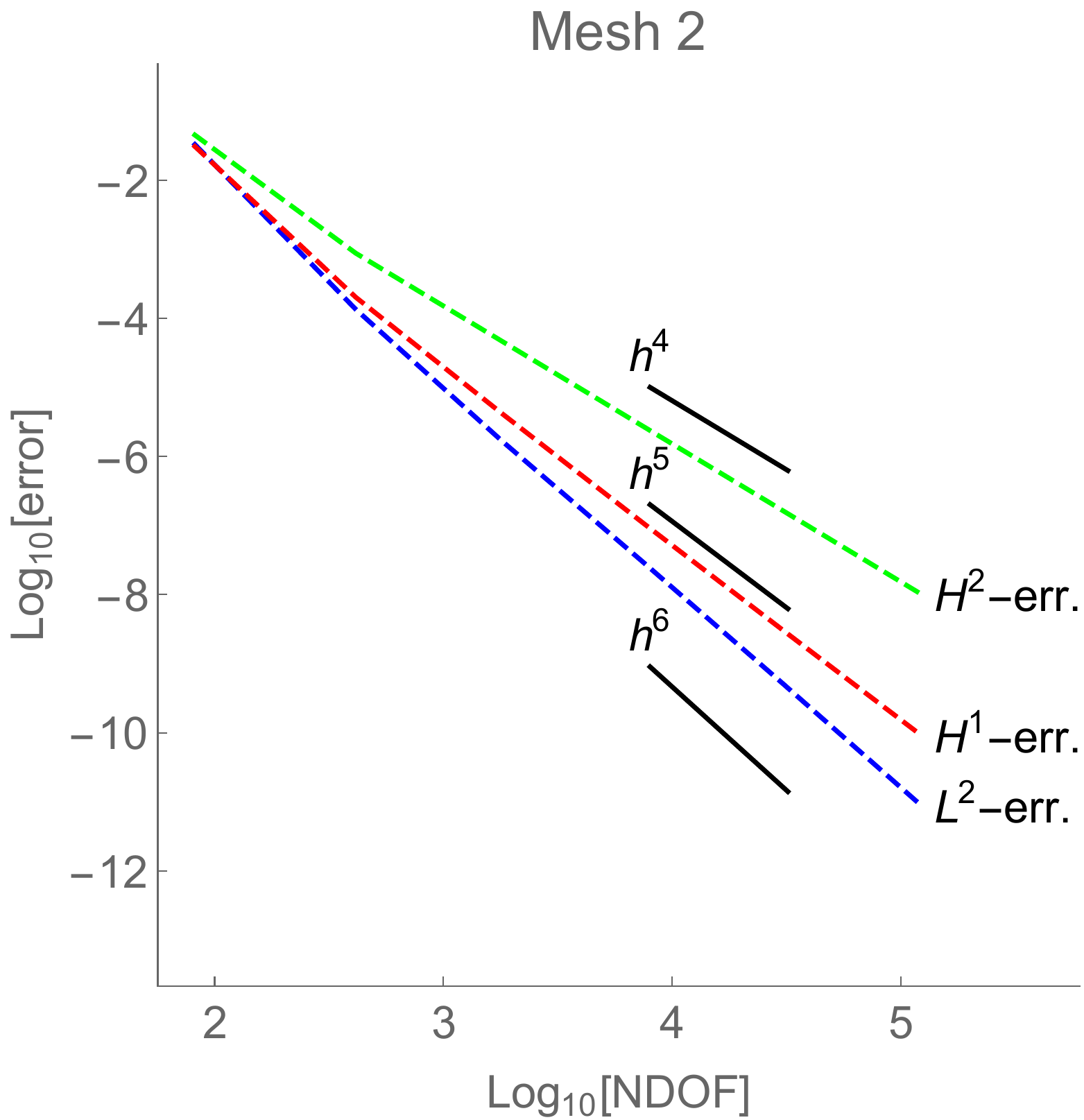} &
\includegraphics[width=5.0cm,clip]{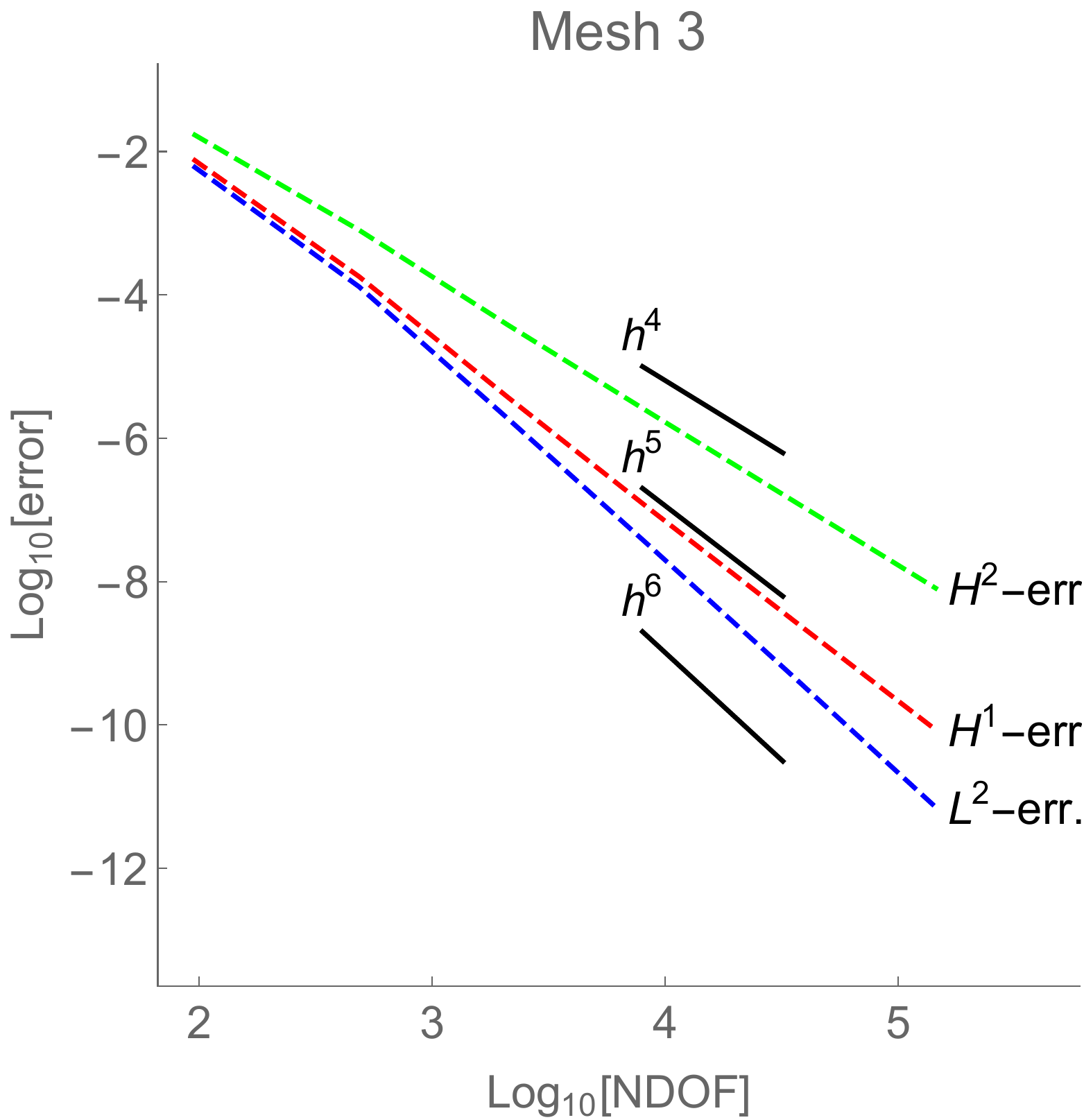} \\
\multicolumn{3}{c}{Solving the biharmonic equation: resulting relative $L^2$-, $H^1$- and $H^2$-errors}
\end{tabular}
\caption{Performing $L^2$-approximation and solving the biharmonic equations for the exact solutions~\eqref{eq:exact_solutions} 
on the Meshes~$1$--$3$ from 
Fig.~\ref{fig:meshes_examples-Mesh1}--\ref{fig:meshes_examples-Mesh3} 
and the resulting errors.
}
\label{fig:examples}
\end{figure}

\clearpage

\section{Conclusions}
\label{sec:Conclusion}

We studied a construction of $C^1$ splines over mixed triangle and quadrilateral meshes for polynomial degrees $\pd\geq 5$. The degrees of freedom are given by $C^2$-data at the vertices, point data and normal derivative data at suitable points along the edges as well as additional point data in the interior of the elements. The resulting space is $C^1$ globally and $C^2$ at all vertices. The degrees of freedom define a stable projection operator which, together with the local polynomial reproduction, yields optimal approximation error bounds with respect to the mesh size for $L^\infty$, $L^2$ as well as Sobolev norms $H^1$ and $H^2$. 

In this paper we only considered planar (bi)linear elements. Extensions to domains with curved boundaries or to surface domains were already discussed separately in case of  quadrilateral meshes as well as triangle meshes. For quadrilateral meshes, \cite{Ma01,BeMa14} provided extensions to elements with curved boundaries, i.e. elements where one boundary edge is curved and the other three are straight. Extensions to surface domains were briefly discussed for spline patches in~\cite{CoSaTa16,KaSaTa17b}. Constructions of $C^1$ surfaces of arbitrary topology using triangle meshes where developed in~\cite{HaBo00}. To extend the construction to mixed surface meshes and mixed meshes with curved boundaries is of vital interest for the applicability of the proposed elements in a general isogeometric framework based on CAD geometries. However, to work out the details of such extensions in the mixed case requires further studies which we intend to do in the future.

\section*{Acknowledgments}
The authors wish to thank the anonymous reviewers for their comments that helped to improve the paper. 
This paper was developed within the Scientific and Technological Cooperation ``Splines in Geometric Design and Numerical Analysis'' between Austria and Slovenia 2018-19, funded by 
the OeAD under grant nr. SI 28/2018 and by ARRS bilateral project nr. BI-AT/18-19-012. 

The research of M. Kapl is partially supported by the Austrian Science Fund (FWF) through the project P~33023. The research of T. Takacs is partially supported by the Austrian 
Science Fund (FWF) and the government of Upper Austria through the project P~30926-NBL. 
The research of M.~Knez is partially supported by the research program P1-0288 and the research project J1-9104 from ARRS, Republic of Slovenia.
The research of J.~Gro\v{s}elj is partially supported by the research program P1-0294 from ARRS, Republic of Slovenia.
The research of V.~Vitrih is partially supported by the research program P1-0404 from ARRS, Republic of Slovenia. This support is gratefully acknowledged.

\end{document}